\documentclass[twoside,11pt]{article}

\usepackage{placeins}
\usepackage{jmlr2e}
\usepackage{tabulary}
\usepackage{booktabs}
\usepackage[dvipsnames]{xcolor}
\usepackage[shortlabels]{enumitem}
\usepackage{mathbbol}
\usepackage{amsmath}
\usepackage{cleveref}
\numberwithin{equation}{section}
\newcommand{\R}{{\mathbb{R}}}
\newcommand{\N}{{\mathbb{N}}}
\newcommand{\U}{\mathcal{U}}
\newcommand{\B}{\mathbb{B}}
\newcommand{\E}{\mathbb{E}}
\renewcommand\P{\mathcal{P}}
\newcommand\bP{\mathbb{P}}
\newcommand\bp{\mathbb{p}}
\newcommand{\X}{\mathcal{X}}
\newcommand{\ones}{\mathbf{1}}

\newcommand{\dto}{\stackrel{d}{\to}}
\newcommand{\pto}{\stackrel{\bP_0}{\to}}
\newcommand{\asto}{\stackrel{a.s.}{\to}}

\DeclareMathOperator*{\argmin}{arg\,min}
\DeclareMathOperator*{\argmax}{arg\,max}

\usepackage{lastpage}
\begin{document}
\jmlrheading{26}{2025}{1-\pageref{LastPage}}{11/23; Revised 9/24}{7/25}{23-1504}{Bohan Wu, César A. Uribe}
\ShortHeadings{BvM for Distributed Bayes}{Wu and Uribe}
\firstpageno{1}

\title{Frequentist Guarantees of Distributed (Non)-Bayesian Inference}

\author{\name Bohan Wu \email bw2766@columbia.edu\\
       \addr Department of Statistics\\
       Columbia University\\
       New York, NY 10027, USA \\
       \\
       \name César A. Uribe \email cauribe@rice.edu\\
       \addr Department of Electrical and Computer Engineering\\
       Rice University\\
        Houston, TX 77005, USA
       }
\editor{Daniel Roy}

\maketitle

\begin{abstract}%   <- trailing '%' for backward compatibility of .sty file
We establish frequentist properties, i.e., posterior consistency, asymptotic normality, and posterior contraction rates, for the distributed (non-)Bayes Inference problem for a set of agents connected over a network. These results are motivated by the need to analyze large, decentralized datasets, where distributed (non)-Bayesian inference has become a critical research area across multiple fields, including statistics, machine learning, and economics. Our results show that, under appropriate assumptions on the communication graph, distributed (non)-Bayesian inference retains parametric efficiency while enhancing robustness in uncertainty quantification.  We also explore the trade-off between statistical efficiency and communication efficiency by examining how the design and size of the communication graph impact the posterior contraction rate. Furthermore, we extend our analysis to time-varying graphs and apply our results to exponential family models, distributed logistic regression, and decentralized detection models. 

\end{abstract}


\section{Introduction}
Modern datasets are frequently generated and stored by distributed systems, including social media, sensor networks, blockchain, and cloud-based databases. However, transmission costs make analyzing these datasets on a centralized machine prohibitively expensive and, in some cases, infeasible. To address this challenge, researchers have turned to distributed algorithms that enable decentralized data-driven decision-making under communication constraints~{\citep{borkar1982asymptotic, tsitsiklis1984convergence, gubner1993distributed}}. In such systems, a set of agents operates within a communication network structure, where each agent can only share information locally with its neighbors. The agents sequentially analyze the data, with each agent performing inference independently and sharing the results through edges defined by the network structure, which may vary over time~\citep{Nedic2017, Uribe2022-2}. 

Decentralized or distributed Bayesian inference originates in statistics~\citep{degroot1974reaching, gilardoni1993reaching}. However, it wasn't until the massive advances in computing power in the past decade that the ideas of distributed inference started regaining interest in the statistics community~\citep{Uribe2022-1,Uribe2022-2}. There is a growing line of works on \textit{Distributed Bayesian inference}, which aims to develop scalable and efficient algorithms for posterior computation on large datasets~\citep{jordan2018communication}. One of the main challenges in this area is to design data-parallel procedures that can handle massive datasets by breaking them into smaller blocks that can be processed independently on individual machines. Much of the current literature focuses on ``one-shot'' or ``embarrassingly parallel'' approaches, which involve only one round of communication between local machines and a central node at the end of the computational pipeline. These approaches compute estimators or posterior samples in parallel on local machines, then communicate the estimates to a central node to form a global estimator or approximation to the posterior (e.g., by computing a Wasserstein barycenter of the posteriors computed on local machines).

From the Markov chain Monte Carlo (MCMC) perspective, there have been several developments in parallel MCMC methods for distributed Bayesian inference~\citep{Neiswanger2013, Wang2013, Minsker2014, Wang2015, Rabinovich2015, Scott2016, Li2017, Minsker2017}. These methods draw samples from the subset posterior in parallel agents and combine the samples to obtain an approximation to the posterior measure for the complete data. 

From the variational Bayes perspective, algorithms such as stochastic variational inference \citep{hoffman2013stochastic} have been proposed for distributed Bayesian inference. These algorithms distribute the data across machines, implement the local variational updates in parallel through stochastic gradient descent (SGD), and update the global variational parameters as a weighted average of local optimums. The variational interpretation of the Bayes rule~\citep{Walker2006} allows the representation between the variational optimization problem and posterior to go both ways. 

In parallel to the success in the statistics community, Distributed Bayesian inference has also been studied in microeconomics under the name ``non-Bayesian social learning.'' Notable works in this area have focused on its axiomatic foundations \citep{Epstein2010}, conditions for achieving consensus \citep{acemoglu2011bayesian}, various learning rules \citep{Golub2017}, and the effects of information aggregation \citep{Molavi2018}. This cross-disciplinary interest underscores the broad relevance and applicability of Distributed Bayesian inference techniques. Here, the agents represent individuals seeking to learn about an underlying state of the world $\theta$. Unlike traditional Bayesian approaches, where each agent makes inferences based on the full data, the ``non-Bayesian learning'' (as economists call it) model captures how individuals make inferences in the presence of other decision-makers, often with limited access to information and interaction with social networks. The Bayesian distributed learning framework offers a promising solution that retains the desirable properties of Bayesian learning, such as ease of uncertainty quantification and flexibility, while incorporating a form of information aggregation that aligns with realistic behavioral assumptions. Indeed, the distributed Bayes rule has been shown to reflect reasonable assumptions about individual behavior in society~\citep{Jadbabaie2012}. The distributed framework is analytically tractable under certain distributional assumptions and computationally feasible in general.  For a more comprehensive literature review, see 
\citep{Molavi2018}.

Although economic theory addresses social learning in various strategic environments, it is almost entirely behavioral, with little data involved. Recent work in the signal processing community has explored the statistical properties of social learning in greater depth, much of which focuses on finite parameter spaces. For example, \cite{Braca2010} studies the relative efficiency of binary testing based on a distributed test statistic under local alternative hypotheses and establishes an asymptotic normality result for the test statistic. In the non-Bayesian social learning setting, \cite{Shahrampour2015} assumes a finite parameter space and provides non-asymptotic bounds on the KL divergence between distributed and centralized beliefs. \cite{lalitha2014social} establishes the exponential convergence of beliefs to the truth from a large-deviation perspective. \cite{Bordignon2021} studies the asymptotic normality of agents’ beliefs in the regime of vanishing step sizes. In contrast, \cite{Inan2022} investigates the consistency and convergence rate of agents’ beliefs toward the truth when the communication graph is random.

The technique for establishing concentration bounds in the discrete finite parameter space relies heavily on bounding the KL divergence between two candidate distributions. Work beyond the finite discrete parameter space has been sparse. One example is \cite{Uribe2022-1,Uribe2022-2}, which establishes consistency and concentration bounds for $p_t^j(\theta)$ when $\Theta$ is discrete or compact, respectively.

Our work studies the asymptotic properties of non-Bayesian social learning in the regime of continuous parameter spaces (as a subset of $\R^p$), the sample size (or time) tending to $\infty$. Although this regime is underexplored in the social learning literature, it is arguably more natural from the perspective of classical asymptotic theory \citep{VanderVaart2000}. Establishing such a theory fills the gap in the existing social learning literature and builds a connection with the distributed Bayesian inference literature in statistics.

Distributed Bayesian procedures have attracted substantial interest across disciplines such as electrical engineering, statistics, and economics. Yet, \textit{the broad adoption of these methods in the statistical community has been hindered by a lack of rigorous analysis of their statistical properties.} Moreover, understanding these properties is key to deepening our knowledge of the consensus behavior of agents within varying communication patterns, a topic of interest to the electrical engineering and economics communities. In this paper, we investigate the distributed Bayes procedures that arise from applying the stochastic mirror descent (SMD) algorithm to statistical estimation problems \citep{Uribe2022-1, Uribe2022-2}. \textit{Our work fills a crucial gap by establishing the Frequentist properties of such distributed Bayes procedures, including posterior consistency, asymptotic normality, and posterior contraction rates.} We also explore the tradeoff between statistical efficiency and communication cost by investigating the relationship between the posterior contraction rate and the structure of the communication graph. Furthermore, we illustrate practical applications of the Bernstein von - Mises results to emphasize their utility in uncertainty quantification. Ultimately, we hope to stimulate further interest in distributed Bayes methods within the statistical community and to establish a solid theoretical foundation for distributed Bayesian inference in fields such as economics and electrical engineering.

\begin{table}[htbp]\caption{Table of Notation}
\begin{center}% used the environment to augment the vertical space
% between the caption and the table
\begin{tabular}{r c p{10cm} }
\toprule
\multicolumn{3}{c}{\underline{Functions:}}\\
$D_{KL}$ & $\triangleq$ & Kullback–Leibler divergence,  \\
$L$ & $\triangleq$ &loss function,  \\
%$L^j$ & $\triangleq$ &loss function $j$\\
$\langle.,.\rangle$ & $\triangleq$ & $L_2$ inner product, as in $\langle f,g \rangle = \int_\R f(x)g(x) dx$. \\
\\
\multicolumn{3}{c}{\underline{Probability Distributions:}}\\
$\bP, \bp$ & $\triangleq$ & data generating measure, data generating density, \\
$\bP f(X)$ & $\triangleq$ &expectation of $f(X)$ when $X \sim \bP$, same as $\E_\bP f(X)$, \\
$\Pi, \pi$ & $\triangleq$ &prior measure, prior density. \\
$P, p$ & $\triangleq$ &posterior measure, posterior density, \\
$P_t^j, p_t^j$ & $\triangleq$ & posterior measure, density for agent $j$ at iterant (time) $t$, \\
\\
\multicolumn{3}{c}{\underline{Others:}}\\
$G,A, A_{ij}$ & $\triangleq$ & Graph, adjacency matrix and its $(i,j)^{th}$ entry, \\
$\lambda_j(A)$ & $\triangleq$ & $j^{th}$ largest eigenvalue of $A$, \\
$\ones$ & $\triangleq$ & the vector of all ones. \\

%when a single is not able to identify the true hypothesis, this is when the distributed setting helps. For example, if we have hypothesis 1,2,3, agent 1 cannot identify between 1 and 2, agent 2 cannot identify between 2 and 3. But the joint estimation allows us to identify between hypotheses. When there is conflicting hypotheses between agents. or when each agent receives different random variables (e.g. mixed data). Or some sensors are noisy and other sensors are not so noisy. Being able to process the information in a distributed manner. MCMC does not process information in a distributed manner. 
\bottomrule
\end{tabular}
\end{center}
\label{tab:TableOfNotationForMyResearch}
\end{table}

The rest of the paper is outlined as follows: Section \ref{sect: preliminary} introduces the distributed Bayesian inference problem from an optimization-centric viewpoint and rigorously defines the distributed Bayes posterior. In Section \ref{sect:assumptions}, we outline sufficient conditions for the consistency and asymptotic normality of the distributed Bayes posterior. Section \ref{sect: consistency} establishes the consistency of distributed Bayes posterior (Theorem~\ref{thm:const}). Section \ref{sect: bvm} establishes the Bernstein-von Mises theorems under both correct (Theorem~\ref{thm:bvm1}) and incorrect model specifications(Theorem~\ref{thm: bvm2}). Section \ref{sect: contraction} provides both the abstract and concrete upper bounds on the posterior contraction rate (Theorem~\ref{thm:contraction1} and Theorem~\ref{thm:contraction2}), with an emphasis on model misspecification. Our analysis is extended to time-varying graphs in Section \ref{sect:time-varying-graphs}, where we establish posterior contraction rates under various communication frequency regimes (Theorem~\ref{cor:time-varying1}). In Section \ref{sect:examples}, we demonstrate the practical use of our findings by establishing Bernstein-von Mises results for three statistical models, including exponential family (Proposition~\ref{prop:EF-1}), logistic regression models(Proposition~\ref{prop:logistic-1}), and the distributed detection problem (Proposition~\ref{prop:detection-1})- a canonical problem in electrical engineering. The paper concludes in Section \ref{sect:discussion} with a discussion on future research directions following our findings.

%%%%%%%%%%%%%%%%%%%%%%%%%%%%
%%%%%%%%%%%%%%%%%%%%%%%%%%%%
%%%%%%%%%%%%%%%%%%%%%%%%%%%%
\section{Preliminaries} \label{sect: preliminary}
%%%%%%%%%%%%%%%%%%%%%%%%%%%%
%%%%%%%%%%%%%%%%%%%%%%%%%%%%
%%%%%%%%%%%%%%%%%%%%%%%%%%%%

Suppose we observe a sequence of i.i.d random variables $X_1, X_2, \cdots$ all taking values in a probability space $(\mathcal{X}, \mathbb{B}, \bP_0)$ where the true distribution $\bP_0$ is unknown. Moreover, let a family of probability distribution be given in the form of $\{\bP_\theta: \theta \in \Theta\}$ where $(\Theta, \mathcal{A})$ is a measurable space. Each $\bP_\theta$ is a probability measure defined on $(\mathcal{X}, \mathbb{B})$ and the mapping $\theta \to \bP_\theta(B)$ is measurable for every $B \in \mathbb{B}$. We refer to $\{\bP_\theta: \theta \in \Theta\}$ as the statistical model. The parameter space $\Theta$ is typically taken as a subset of the Euclidean or Hilbert space to avoid any measurability issue with $\{\bP_\theta: \theta \in \Theta\}$ \citep{ghosal2017fundamentals}. 

The centralized statistical estimation problem is to find a subset $\Theta_0 \subseteq \Theta$ such that for $\theta_0 \in \Theta_0$, $\bP_{\theta_0}$ is the ``closest'' to $\bP_0$ with respect to a metric.  Geometrically, the goal is to find a point in the subset of the probability measures $\{\bP_\theta: \theta \in \Theta\}$ closest to $\bP_0$ under a given topology. The topology is often defined by divergence on the space of probabilities, such as the Kullback-Leibler (KL) divergence, Rényi divergence, etc. The definition of KL, Rényi divergence, and other divergence functions are reviewed in Section~\ref{subsect:divergence} of the Appendix.

Arguably, the most natural estimation problem to consider is based on KL divergence:
\begin{equation}
\label{eqn: simple-KL}
  \theta_0 \in \arg \min_{\theta \in \Theta} D_{KL}(\bP_0\parallel \bP_\theta).
\end{equation}
\Cref{eqn: simple-KL} is equivalent to maximum likelihood estimation (MLE) at the population level. For example, if the true distribution $\bP_0$ is standard normal $N(0,1)$ and the parametric family $\bP_\theta$ is the normal location family ${N(\theta, 1); \theta \in [-1,1]}$, the optimal value is $\theta_0 = 0$.

There are two main challenges in solving~\eqref{eqn: simple-KL}. First, since $\bP_0$ is unknown, it is typically replaced with samples drawn from the distribution. The resulting approximation error influences the sample complexity of MLE, which has been extensively studied in the statistics literature \citep{VanderVaart2000}. The second challenge is computational. When the set $\Theta$ is discrete, non-smooth, or disconnected, default approaches such as first-order optimization methods cannot be directly applied to solving \Cref{eqn: simple-KL}, even at the population level.

Instead of minimizing over $\Theta$, a common approach is to “lift” the problem by minimizing over the probability distributions on $\Theta$.

Let $\Delta_\Theta$ be the space of probability density functions defined on $\Theta$. Then
\begin{equation}\label{eqn: simple-KL-Bayes}
    \min_{\theta \in \Theta} D_{KL}(\bP_0 \parallel \bP_\theta) = \min_{p  \in \Delta_\Theta} \int_\Theta p(\theta) D_{KL}(\bP_0\parallel \bP_\theta) d \theta,
\end{equation}
where $\Delta_\Theta$ is the (hypothetical) space of all probability distributions over $\Theta$, and \mbox{$\int_\Theta p(\theta) d\theta = 1$}, with $p(\theta) \geq 0$.  If $\theta_0$ solves \Cref{eqn: simple-KL}, $p^* = \delta_{\theta_0}$ solves \Cref{eqn: simple-KL-Bayes}, thus the two problems are equivalent. The advantage, however, is that \Cref{eqn: simple-KL-Bayes} is a continuous optimization problem regardless of the topology in which $\Theta$ is defined; that is, $\Theta$ could be a finite discrete set.

\Cref{eqn: simple-KL-Bayes} introduces an equivalent formulation of the statistical estimation as a minimization problem over the space of probability measures on $\Theta$. Given the linearity of the expectation functional, the Riez representation theorem implies the existence of an inner product $\langle., .\rangle$ that characterizes the expectation over $\Theta$. We then reformulate the KL-minimization problem in terms of linear stochastic optimization up to a constant:
\begin{equation}
    \begin{aligned}
   \min_{\theta \in \Theta} D_{KL}(\bP_0\parallel \bP_\theta)  &\overset{c}{=}  \min_{p \in \Delta_\Theta}  \E_{\bP_0}\langle p,- \log \bp_\theta(x) \rangle.  
    \end{aligned}
\end{equation}
Note that
\begin{align*}
      \min_{\theta \in \Theta} D_{KL}(\bP_0\parallel \bP_\theta) &= \min_{p \in \Delta_\Theta} \int_\Theta p(\theta) D_{KL}(\bP_0\parallel \bP_\theta) d \theta  \\
      &= \min_{p \in \Delta_\Theta} \int_\Theta p(\theta) \int_{\mathcal X} \bp_0(x) \log \frac{\bp_0(x)}{\bp_\theta(x)} dx d \theta \\
      & = \min_{p \in \Delta_\Theta} \int_{\mathcal X} \bp_0(x) \int_\Theta p(\theta) \log \frac{\bp_0(x)}{\bp_\theta(x)} d \theta dx, \qquad \text{by Fubini's theorem} \\
      &= \min_{p \in \Delta_\Theta} \int_{\mathcal X} \bp_0(x) \int_\Theta p(\theta) [\log \bp_0(x) - \log \bp_\theta(x) ] d \theta dx \\
      &\overset{c}{=}  \min_{p \in \Delta_\Theta} - \int_{\mathcal X} \bp_0(x) \int_\Theta p(\theta) \log \bp_\theta(x)  d \theta dx. 
\end{align*}

In the reformulated problem, $\Delta_\Theta$ is the probability simplex over $\Theta$ and $\langle p,- \bp_\theta(x) \rangle$ represents the inner product between the simplex vector $p$ and the negative log-likelihood $\bp_\theta(x)$, taken with respect to the expectation under $\bP_0$.

The resulting problem can be efficiently solved using computational methods from the stochastic optimization literature. One approach is the stochastic mirror descent (SMD) algorithm \cite[Algorithm~6]{Uribe2022-1, Uribe2022-2}, which iterates through a KL-regularized optimization problem:
\begin{equation}
\label{eqn:smd}
    p_{t+1} = \argmin_{p \in \Delta_\Theta} \left\{-\langle p,  \log \bP_\theta(x_{t+1})\rangle + \alpha D_{KL}(p \parallel p_t) \right\}. 
\end{equation}
At time $t$, the agent’s goal is to maximize the expected log-likelihood while staying close to the belief $p_t$. The trade-off between these two objectives is governed by the learning rate $\alpha$.

Equation \eqref{eqn:smd} is also the variational representation of generalized Bayesian inference. 
\begin{equation}
     p_{t+1} = \argmin_{p \in \Delta_\Theta} D_{KL}(p \parallel \tilde p), \quad \tilde p \propto (\bP_\theta(x_{t+1}))^{\frac{1}{\alpha}} p_t(\theta). 
\end{equation}
A generalized posterior is a probability distribution $p(\theta \mid X_1, \ldots, X_t) \propto \pi(\theta) \tilde L_t(\theta)$, where $\pi$ is a prior and $\tilde L_t$ is a surrogate likelihood \citep{miller2021asymptotic}. The stochastic mirror descent update in Equation \eqref{eqn:smd} produces a generalized Bayes posterior:
\begin{equation}
\label{eqn:simple-bayes}
    p_{t+1}(\theta) \propto (\bP_\theta(x_{t+1}))^{\frac{1}{\alpha}} p_t(\theta). 
\end{equation}
The result follows from a standard variational Bayes argument \citep{knoblauch2022optimization}.

The result above suggests a fascinating interplay between stochastic optimization and generalized Bayesian inference. The parameter $\alpha$ serves as both the step size in the SMD update and a temperature parameter in the generalized Bayes’ rule. When $\alpha$ is set to 1, we recover the standard Bayes’ theorem. In the sequel, we explore the connection between generalized Bayes and SMD to formulate a distributed Bayesian posterior based on the distributed SMD algorithm.

%In this framework, $p_{t+1}^j(\theta)$ is integrates both neighbor information, given by $\prod{i=1}^m (p_t^i(\theta))^{[A_t]{ij}}$, and new data via the likelihood $\bp\theta^j(x_{t+1}^j)$. This can be viewed as a Bayesian update, with the prior as a weighted geometric mean of neighbor beliefs and the standard likelihood from the agent's new data. 

\subsection{Distributed Bayesian Inference}
\label{subsect:distributed-Bayesian-inference}
In this section, we introduce the distributed Bayesian inference problem. Consider a set of $m$ independent agents, represented by the index $j = 1, ..., m$. Each agent independently observes a sequence of random variables at discrete time steps $t = 0,1, 2, ...$.  The random variable observed by agent $j$ at time $t$ is denoted as $X_t^j$. The random process observed by agent $j$ is denoted as $X^j$. The collection of random variables all agents observe at time $t$ is denoted as $X_t$. {In general, each $X^j$ may be endowed with a different probability space with a common parameter space. For simplicity, we assume that the set of all random variables ${X_t^j}$ are i.i.d. draws from the probability space $\left(\mathcal{X}, \mathbb{B}, \P\right)$ with density $p$.

Given a  common set of parameters $\Theta$, the private parametric statistical model that agent $j$ can access is defined as $\mathcal{P}_\Theta^j = \{\bP_\theta^j: \theta \in \Theta\}$. Each model $\mathcal{P}_\Theta^j$ has the same support for all $j \in [m]$. 

At time $t$, the agents interact through an undirected communication graph $G_t = (V, E_t)$, where $V = [m]$ denotes the set of agents. An edge $(j, i) \in E_t$ implies that agent $j$ can communicate with agent $i$ at time $t$. 

%We assume that the initial communication graph $G_0$ is fully connected.

The weighted adjacency matrix associated with $G_t$ is denoted by $A_t$.  We assume that $A_t$ is a doubly stochastic matrix obtained by normalizing the matrix $A_t'$, where $[A_t']_{ij} = 1$ if there exists a communication link between agent $i$ and agent $j$ and $[A_t']_{ij} = 0$ otherwise. {For undirected graphs, one can always construct a doubly stochastic matrix by normalizing the adjacency matrix, so $A_t$ is guaranteed to exist. }

In our framework, agents share information by communicating their beliefs, represented by $p_t^j(\theta)$. This notation represents the posterior distribution of $\theta$ as perceived by agent $j$ at time $t$. 

Similar to the centralized case, we formulate distributed statistical inference as a distributed optimization problem:
\begin{equation}
\label{eqn: dKL}
\min_{\theta \in \Theta} \sum_{j = 1}^m D_{KL}(\bP_0\parallel \bP_\theta^j).
\end{equation}

For each agent $j \in [m]$, let $\pi^j$ represent their initial belief or prior. Therefore, we have $p_0^j = \pi^j$. Similar to the statistical estimation for a single agent, the distributed statistical inference problem admits a reformulation via the distributed stochastic mirror descent algorithm \citep{Uribe2022-1, Uribe2022-2}.  The belief of the agent $j$ at time $t + 1$ is obtained through the following mirror descent update. 
\begin{equation}
\label{eqn:dSMD}
        p_{t+1}^j = \argmin_{p \in \Delta_\Theta} \left\{  -\langle \log \bp_\theta^j(x_{t+1}^j), p \rangle + \sum_{i  = 1}^m [A_t]_{ij} D_{KL}(p \parallel p_t^i) \right\}, 
\end{equation}
which is solved via the update
\begin{equation}
\label{eqn:dBr}
p_{t+1}^j(\theta) \propto \bp_\theta^j(x_{t+1}^j)\prod_{i = 1}^m (p_t^i(\theta))^{[A_t]_{ij}}. 
\end{equation}
The derivation from \eqref{eqn:dSMD} to \eqref{eqn:dBr} is identical to that of the case of a single agent \eqref{eqn:smd}--\eqref{eqn:simple-bayes}.

Equation \eqref{eqn:dBr} is aptly referred to as the \textit{distributed Bayes rule}, as it generalizes the Bayesian rule to a distributed setting. This distributed Bayes rule induces a \textit{distributed Bayes posterior} $p_{t+1}^j$, which might be viewed as the belief of each agent after updating its prior belief based on the new data and information received from its neighbors. 

Despite its name, the distributed Bayes posterior is not a standard posterior distribution. The classic Bayes rule is a special case of the distributed Bayes rule where the communication graph has no edges, i.e., $A_t = I_m$. Then, each agent updates independently, 
\begin{equation*}
    \begin{aligned}
        p_{t + 1}^j(\theta) 
        &\propto  \bp_\theta(x_{t + 1}^j)p_t^j(\theta). 
    \end{aligned}
\end{equation*}
At the other extreme, where the communication graph $G_t$ is fully connected, each agent can communicate with every other agent. Then, the distributed Bayes rule effectively acts as a weighted Bayesian update rule, with equal weights assigned to all agents. The resulting posterior
is a product of the likelihood of the incoming data and a ``tempered'' posterior with power $1/m$
\begin{equation*}
    \begin{aligned}
        p_{t + 1}^j(\theta) 
        &\propto  \bp_\theta(x_{t+1}^j) \prod_{i = 1}^m \prod_{k = 1}^t (\bp_\theta(x_k^i))^{1/m} \prod_{i = 1}^m \pi^i(\theta)^{1/m}. 
    \end{aligned}
\end{equation*}
These examples show that distributed Bayesian posterior can adapt to the underlying communication structure, balancing individualistic updates (when there are no communications) and collective updates (when the communication happens according to the network).

Our strategy is to analyze $p_{t+1}^j(\theta)$ through the lens of generalized posterior. In the next section, we provide the measure-theoretic definition of the distributed Bayes posterior as a generalized Bayes posterior.

%The Frequentist properties of $\alpha-$posterior is an active area of research, with results such as its posterior consistency and Bernstein von-Mises phenomena proven (see \cite{ghosal2017fundamentals,bhattacharya2019bayesian, yang2020alpha, miller2018robust, medina2022robustness}). 

\subsection{Distributed Bayes Posterior} 
\label{subsect:distributed-Bayes-posterior}
This section provides a rigorous definition of the distributed Bayes posterior we study throughout this paper. 

Define the prior measure $\Pi$ as follows: 
\begin{align*}
    \Pi(B) = \frac{\int_B \prod_{j = 1}^m \pi^j( \theta)^{1/m} d\theta}{\int_\Theta \prod_{j = 1}^m \pi^j( \theta)^{1/m} d\theta} . 
\end{align*}
For each $t$ and $j$, denote $z_t^j = \int_\Theta \bp_\theta^j(x_t^j) \prod_{k = 1}^{t - 1} \prod_{i = 1}^m \bp_\theta^i(x_k^i)^{\left[\prod_{\tau = k}^{t-1}A_\tau \right]_{ij}}  \Pi(d\theta)$ and assume $z_t^j < \infty$.  

The \textit{distributed Bayes posterior} is defined as the probability measure $P_t^j$ such that for $B \in \B$, 
\begin{equation} \label{def-db-posterior}
    P_{t}^j(B) = \frac{1}{z_t^j}\int_B \bp_\theta^j(x_t^j) \prod_{k = 1}^{t - 1} \prod_{i = 1}^m \bp_\theta^i(x_k^i)^{\left[\prod_{\tau = k}^{t-1}A_\tau \right]_{ij}}  \Pi(d\theta). 
\end{equation}
The measure-theoretic definition of distributed Bayes posterior is equivalent to the recursive formulation \eqref{eqn:dBr}. The definition shows that the distributed Bayes posterior is the product of the prior distribution and a surrogate likelihood. 

The asymptotic properties of posteriors with surrogate likelihood have been studied in \cite{miller2021asymptotic}, which establish sufficient assumptions for generalized posterior to achieve consistency, asymptotic normality (Bernstein von-Mises), and correct frequentist coverage.  However, the one key difference between the distributed Bayes posterior and setting in \citep{miller2021asymptotic} is the latter depends only on a sample of size $t$. In contrast, the surrogate likelihood used in distributed Bayes posterior depends on the sample size $mt$, the statistical model $\bP_\Theta^j$, and the communication graph structure $G$. Therefore, the theoretical analysis of distributed Bayes posterior requires careful adjustments of the standard results for generalized posteriors for new constraints and challenges. 

We define the following surrogate loss functions $f_t^j, f_t, f$ on $\Theta$. 
\begin{align}
    \label{def:f_t^j}
    f_t^j(\theta) &= -\frac{1}{t} \log \bp_\theta^j(x_t^j) 
    - \frac{1}{t} \sum_{k = 1}^{t-1} \sum_{i = 1}^m \left[ \prod_{\tau = k}^{t-1} A_\tau \right]_{ij} \log \bp_\theta^i(X_k^i), \quad j \in [m], \\
    \label{def:f_t}
    f_t(\theta) &= -\frac{1}{mt} \sum_{k = 1}^t \sum_{i = 1}^m \log \bp_\theta^i(X_k^i), \\
    \label{def:f}
    f(\theta) &= - \frac{1}{m} \sum_{i = 1}^m \bP_0 \log \bp^i_\theta.
\end{align}
The function  $f_t^j$ plays the role of a generalized likelihood,  $f_t$ corresponds to the empirical mean of $\log p^j_\theta(X)$ and a special case of $f_t^j$ when the adjacency matrix has weights of all $1/m$. Informally, $f_t^j$ and $f_t$ is asymptotically indistinguishable under mild assumptions, but $f_t$ has much nicer statistical behaviors, so one should expect distributed Bayes posterior to behave like a generalized posterior with likelihood given by~$f_t$. 

\section{Assumptions} \label{sect:assumptions}
Throughout the paper, we impose the following assumptions:
\begin{enumerate}[label=(\roman*)]
    \item $\Theta$ is an open subset of $\R^p$ with standard Euclidean metric $d$ . 
    \item There exists a unique parameter $\theta_0 \in \Theta$ that minimizes Problem \eqref{eqn: dKL}, that is,
\begin{equation}\label{def:theta_0}
\theta_0 = \argmin_{\theta \in \Theta}\sum_{j = 1}^m D_{KL}(\bP_0\parallel \bP_\theta^j). 
\end{equation}
\end{enumerate}
The remaining assumptions fall into four categories: communication graph structure, regularity of private statistical models, prior mass assumptions, and consistent testing assumptions. The latter three are standard in Bayes asymptotics with a well-established history (see Appendix Section~\ref{subsect:bvm-history}). The assumptions on the communication graph are relaxed in Section~\ref{sect:time-varying-graphs} to allow for temporal dependence. 

We assume that the communication graph $G_t, t \geq 1$ is static and undirected, i.e., $G_t = G = (V, E)$. 
\begin{assumption}[Graph Assumptions]\label{assumption: graph}
The communication graph $G_t$ is static and undirected, i.e., $G_t = G = (V, E), t \geq 1$. Moreover, $G$ and the adjacency matrix $A$ satisfy
\begin{enumerate}[(a)]
    \item $A$ is symmetric and row stochastic. 
    \item $A$ has positive diagonal entries, i.e., $a_{ii} > 0$ for all $i \in [m]$. %in other words, all states in $A$ are aperiodic 
    \item $G$ is connected, i.e., a directed path exists from any agent $i$ to any agent $j$. 
\end{enumerate}
\end{assumption}
The assumption of connectedness is standard in the literature on network communication to ensure information flow between agents \citep{Shahrampour2015, Nedic2017}. The three assumptions ensure that the Markov chain with transition matrix $A$ is irreducible and aperiodic. 

We proceed to establish a result for the convergence of $A^t$ to $\frac{1}{m}\ones \ones^T$ and an upper bound on the convergence rate of $\|A^t - \frac{1}{m}\ones \ones^T\|_1$. 
\begin{lemma}
\label{lemma:graph_conv}[\cite{acemoglu2011bayesian,Nedic2017}]
Let Assumption~\ref{assumption: graph} hold. The matrix $A$ satisfies
\begin{equation*}
   \limsup_{t \to \infty} \sum_{k = 1}^t \sum_{j = 1}^m |[A^{t - k}_{ij}] - \frac{1}{m}| \leq \frac{16m^2 \log m}{\nu}, \qquad i \in [m]
\end{equation*}
where $\nu$ is the smallest positive entry in $A$. 
\end{lemma}
 
Lemma~\ref{lemma:graph_conv} states that given fixed $A$, the sum $\sum_{k = 1}^t \sum_{j = 1}^m |[A^{t - k}_{ij}] - \frac{1}{m}|$ is uniformly bounded in $t$, and scales in the order of $m^2 \log m$. The argument for Lemma \ref{lemma:graph_conv} relies crucially on the geometric convergence rate of an irreducible, aperiodic Markov chain.  See Section~\ref{app:support} in the Appendix for the proof.

The next group of assumptions concerns the smoothness of the log-likelihood of the private statistical models. 
\begin{assumption}[Regularity Assumptions]\label{assumption: regularity}
The regularity assumptions on the statistical model $\bP_\theta^j = \{\bP_\theta^j,  \theta \in \Theta\}$ involves, for every $j \in [m]$ and almost surely $[\bP_0]$, 
\begin{enumerate}[(a)]
    \item For every $\theta \in \Theta$,  $\bP_0 |\log \bp_\theta^j| < \infty$, {i.e., the expectation of $|\log \bp_\theta^j(X)|$ is finite under $X \sim \bP_0$}. 
    \item The mapping $\theta \mapsto \log \bp_\theta^j$ is convex for every $x$ in a neighborhood of $\theta_0$, and the entrywise inequality $\nabla f(\theta_0 - \epsilon) < 0 < \nabla f(\theta_0 + \epsilon)$ holds.  
    \item The statistical model $\bP_\theta^j$ is differentiable in quadratic mean at $\theta_0$ with nonsingular Fisher information matrix $V_{\theta_0}^j$.  
    \item  there exists a measurable function $s^j$ with $\bP_0 [s^j(X)]^2 < \infty$ such that, for every $\theta_1, \theta_2$ in a neighborhood of $\theta_0$, 
        \begin{equation*}
            |\nabla \log \bp_{\theta_1}^j - \nabla \log \bp_{\theta_2}^j| \leq s^j(x) \|\theta_1 - \theta_2 \|
        \end{equation*}
    \item the mapping $\theta \mapsto \log \bp_\theta^j$ is twice continuously differentiable for every $x$ in a neighborhood of $\theta_0$, and the Fisher information matrix $V_{\theta_0}^j$ exists and is nonsingular. 
    \item the mapping $\theta \mapsto \log \bp_\theta^j$ is three-times continuously differentiable for every $x$ in a neighborhood of $\theta_0$, the Fisher information matrix $V_{\theta_0}^j$ exists and is nonsingular, and the third-order partial derivatives uniformly bounded by an integrable function in a neighborhood of $\theta_0$. 
\end{enumerate}
\end{assumption}
{The regularity conditions described in Assumption~\ref{assumption: regularity} are common prerequisites for the asymptotic analysis of posteriordistributions \citep{VanderVaart2000}. The intuition is that the differentiability or convexity of the log-likelihood ensures the consistency of the maximum likelihood estimators, and the quadratic means differentiability enables a valid second-order Taylor expansion around the truth. This allows us to describe the asymptotic properties of the posterior using properties of the MLEs.}

{The next group of assumptions is typically referred to as the prior mass or prior thickness assumption. }
\begin{assumption}[Prior Assumptions]\label{assumption: pmc}
For every $j \in [m]$, the prior distribution $\Pi_j$ satisfies
\begin{enumerate}[(a)]
    \item $\Pi_j(\U_\epsilon) > 0$ for all $\epsilon > 0$, where $\U_\epsilon = \{\theta \in \Theta: \frac{1}{m}\sum_{j = 1}^m D_{KL}(\bP_0 \parallel \bP_\theta^j)<  \epsilon\}$. 
    \item the density $\pi_j$ is continuous and positive at $\theta_0$. 
\end{enumerate}
\end{assumption}

The first assumption states that prior put a sufficient amount of mass in a KL neighborhood of the target distribution $\bP_0$. An extra continuity assumption of the density $\pi$ at $\theta_0$ is sometimes assumed to connect the first assumption to a statement about the neighborhood of $\theta_0$ \citep{VanderVaart2000}. 

Uniform, consistent testing assumptions enable $\theta_0$ to be distinguished with a sequence of test functions.  This assumption ensures that asymptotically negligible mass is placed outside a neighborhood of $\theta_0$. 
\begin{assumption}[Uniform Consistent Testing]\label{assumption: uct}
Let $\Theta \subseteq \R^p$.  For each $j \in [m]$,  
\begin{enumerate}[(a)]
    \item For all $\epsilon > 0$, there exists $\delta > 0$ such that $\lim_{t \to \infty} \bP_0(\inf_{\|\theta - \theta_0\| > \delta}|f_t^j(\theta) - f_t^j(\theta_0)| \geq \epsilon) = 1$. 
    \item There exists a sequence of test functions $\phi_t$ such that $\bP_{\theta_0}^j (1 - \phi_t(X^{(mt)})) \to 0$ and $\sup_{\|\theta - \theta_0\| \geq \epsilon} \bP_\theta^j (1 - \phi_t(X^{(mt)})) \to 0$ for every $\epsilon$.
\end{enumerate} 
\end{assumption}
Assumption \ref{assumption: uct}(b) is attributed to Theorem 10.1 in \cite{VanderVaart2000}, and Assumption \ref{assumption: uct}(a) is given first in \cite{ghoshramamoorthi}. The two assumptions in \ref{assumption: uct} play an identical role in the theory but are slightly different. We mainly use assumptions \ref{assumption: uct}(a). 

%our proving strategy for the main therorems borrow ideas from Miller (2021) but is still novel because
%see if the distributed posterior is close to the true posterior distribution (in KL)

\section{Posterior Consistency}
\label{sect: consistency}

This section establishes the posterior consistency for the distributed Bayes posterior.  Theorem~\ref{thm:const}, to be presented next, provides a general concentration result for the distributed Bayes posterior $P_t^j$ over the measurable space $(\Theta, \mathcal{A})$. The proof of Theorem~\ref{thm:const} takes its cues from Schwartz's theorem \citep{Schwartz1965}. This fundamental theorem implies two supporting findings. {Corollary \ref{cor:const-1} shows that Theorem~\ref{thm:const} holds under model misspecification and alternative definitions of the neighborhood around the true parameter}. Moreover, Lemma \ref{lemma:const_condition} lays down more user-friendly sufficient assumptions for the regularity assumption on the log-likelihood (Assumption \ref{assumption: regularity}(a)) in Theorem~\ref{thm:const}.

{First, we introduce a supporting lemma to demonstrate that the surrogate loss functions $f_t^j$ and $f_t$ are asymptotically equivalent, and both converge to the population loss function $f$. This lemma suggests that under the conditions of a connected communication graph and a first-moment condition satisfied by the statistical model, the distributed Bayes posterior $P_t^j$ becomes asymptotically equivalent to a posterior derived from $f_t$.}

\begin{lemma}
\label{lemma:ptwise_conv}
Let Assumptions \ref{assumption: graph} and \ref{assumption: regularity}(a) hold. Then $f_t^j$ and $f_t$ converges to $f$ on $\Theta$ in $[\bP_0]-$probability. 
\end{lemma}
The proof can be bound in Section~\ref{app:consistency} of the Appendix. The argument is based on a distributed version of the law of large numbers and geometric convergence of the adjacency matrices described in Lemma~\ref{lemma:graph_conv}. 

We are ready to state the main result of this section. 
\begin{theorem}[Consistency]
\label{thm:const}
Let $\theta_0 \in \Theta$ be defined in \eqref{def:theta_0}, $\epsilon>0$, and denote $\U_\epsilon = \{\theta \in \Theta: \frac{1}{m}\sum_{j = 1}^m D_{KL}(\bP_0 \parallel \bP_\theta^j)<  \epsilon\}$. Moreover, let Assumptions \ref{assumption: graph}, \ref{assumption: regularity}(a), \ref{assumption: pmc}(a) hold.  Then, the distributed Bayes posterior $P_t^j$ defined in \eqref{def-db-posterior} has the following property:
\begin{equation*}
    \lim_{t \to \infty} P_t^j (\U_\epsilon) = 1, 
\end{equation*}
in $[\bP_0]-$probability for each $\epsilon > 0$ and $j \in [m]$. 
\end{theorem}

See Section~\ref{app:consistency} for the proof. The proof uses the standard technique based on Schwartz's theorem \citep{miller2021asymptotic}. The primary insight is that, given the defined regularity assumptions on private statistical models and the presence of a connected communication graph, the distributed Bayes posterior concentrates at the same point as the $\frac{1}{m}-$geometric average of the individual Bayes posteriors. 

Theorem~\ref{thm:const} does not assume correct model specification. When the model is misspecified, the distributed Bayes posterior concentrates around the unique minimizer $\theta_0$ of Problem \eqref{eqn: dKL}, which is assumed to exist. 

Under additional assumptions,  the neighborhoods on the space of probability measures used in stating Theorem~\ref{thm:const} can be substituted with neighborhoods on $\Theta$. 

{
\begin{corollary} \label{cor:const-1}
Let $(\Theta, d)$ be metric space, $\theta_0 \in \Theta$, $\epsilon>0$, and denote $\mathcal N_\epsilon = \{\theta \in \Theta:d(\theta, \theta_0) < \epsilon\}$.   Moreover, let Assumptions \ref{assumption: graph}, \ref{assumption: regularity}(a), \ref{assumption: regularity}(b), \ref{assumption: pmc}(c) hold.  Then the distributed Bayes posterior $P_t^j$ defined in \eqref{def-db-posterior} has the following property
\begin{equation*}
    \lim_{t \to \infty} P_t^j(\mathcal N_\epsilon) = 1, 
\end{equation*}
in $[\bP_0]-$probability for each $\epsilon > 0$ and $j \in [m]$. 
\end{corollary}
}

The result follows directly from Theorem 3 of \cite{miller2021asymptotic}. Hence, the proof is omitted.

{In practice, Assumption \ref{assumption: regularity}(a) is challenging to verify due to the lack of information on the moments of the unknown distribution. The expectation may not exist under a heavy-tailed distribution like the Cauchy distribution.} In the result below, we provide an information-theoretic equivalent of Assumption \ref{assumption: regularity}(a) that is easier to check.
%In particular, the derived convergence rates use results from the measure concentration of random variables. In the most common cases, the random variables have a sub-Gaussian or subexponential behavior \citep{Boucheron2013}.%
\begin{lemma}
\label{lemma:const_condition}
Assumption \ref{assumption: regularity}(a) is equivalent to either of the following assumptions:
\begin{enumerate}
\item[(a).] $\bP_0$ is absolutely continuous with respect to $\bP_\theta^j$ for all $\theta \in \Theta$,
\item[(b).] $D_{KL}(\bP_0 \parallel \bP_\theta^j) < \infty$ for all $\theta \in \Theta$. 
\end{enumerate}
\end{lemma}

When $\bP_0$ has a density $\bp_0$, the absolute continuity assumption is equivalent to the assumption that if $\bp_\theta^j(x) > 0$, then $\bp_0(x) > 0$  for every $x$ \citep{cover1999elements}. The second assumption is the minimum assumption for the problem of distributed statistical estimation to be tractable. Both assumptions are easy to check, as illustrated by the following example. 

\begin{example}
Assume that $\bP_\theta^j = \{N(\theta, \sigma^2_j), \theta \in \Theta\}$ for known $\sigma^2_j > 0$ and $\bP_0 = N(\theta_0, \sigma_0^2)$ for unknown mean $\theta_0$ and variance $\sigma_0^2$. The Gaussian distribution with positive variance has support on the whole real line. Then $\bP_0$ is absolutely continuous with respect to $\bP_\theta^j$ for all $j$.  For the second assumption, the KL divergence between $\bP_\theta^j$ and $\bP_0$ has an explicit formula: 
\begin{equation*}
    D_{KL}(\bP_\theta^j \parallel \bP_0) = \frac{1}{2} \left( \frac{(\theta - \theta_0)^2}{\sigma_0^2} + \frac{\sigma_j^2}{\sigma_0^2} - 1 - \log \frac{\sigma_j^2}{\sigma_0^2} \right).
\end{equation*}
For finite $\sigma_j^2, \sigma_0^2, \theta_0$, the KL divergence is finite for every $\theta \in \R$. 
\end{example}

%posterior contraction? 
%%%%%%%%%%%%%%%%%%%%%%%%%%%%%%%%%%%%%%%%%%%%%%%%%
%%%%%%%%%%%%%%%%%%%%%%%%%%%%%%%%%%%%%%%%%%%%%%%%%
%%%%%%%%%%%%%%%%%%%%%%%%%%%%%%%%%%%%%%%%%%%%%%%%%
\section{Asymptotic Normality} \label{sect: bvm}
%%%%%%%%%%%%%%%%%%%%%%%%%%%%%%%%%%%%%%%%%%%%%%%%%
%%%%%%%%%%%%%%%%%%%%%%%%%%%%%%%%%%%%%%%%%%%%%%%%%
%%%%%%%%%%%%%%%%%%%%%%%%%%%%%%%%%%%%%%%%%%%%%%%%%
%efficiency

The Bernstein-Von Mises (BvM) theorem states that under certain conditions on the prior, the posterior distribution approximates a Gaussian distribution centered at a consistent estimator, such as the maximum likelihood estimator, as data increases. This theorem is pivotal in Bayesian statistics for at least two reasons. First, it provides a quantitative description of how the posterior contracts to the truth. Second,  this theorem justifies Bayesian credible sets as valid frequentist confidence sets, i.e., sets of posterior probability $1 - \alpha$ contain the true parameter at the confidence level $1 - \alpha$. 

In this section,  we establish Bernstein von - Mises theorems for the distributed Bayes posterior defined in Equation \eqref{def-db-posterior}. Our results address the asymptotic normality of the distributed Bayes posterior under both correct and incorrect model specifications. We consider scenarios where all agents accurately specify the model, and some agents gather observations from a 'true distribution' that does not belong to their statistical models.

Theorem~\ref{thm:bvm1} provides sufficient assumptions for the distributed posterior to converge to a normal distribution centered around a sequence of M-estimators $\hat \theta_t^j$. Theorem~\ref{thm: bvm2} generates an analogous result under classical assumptions, with the normal approximation centered around $\theta_0$ (defined in Equation \eqref{def:theta_0}). The supporting lemmas are provided preceding the main results. Given the abstract nature of the assumptions, Theorem~\ref{thm:bvm1} is supplemented by Corollary \ref{cor:bvm1} and Corollary \ref{cor:bvm2}, which outlines more user-friendly assumptions that guarantee the same Bernstein von-Mises (BvM) results.

The Bernstein von - Mises argument relies critically on the following sequence of M-estimators: 
\begin{equation} \label{eqn:m-estimator}
\hat \theta_t^j = \argmin_{\theta \in \Theta} f_t^j(\theta). 
\end{equation}
The existence and consistency of $\hat \theta_t^j$ is not always guaranteed; one set of sufficient assumptions is provided as follows.

\begin{lemma}\label{lemma:M-est-1}
    Let Assumptions  \ref{assumption: graph}, \ref{assumption: regularity}(a), \ref{assumption: regularity}(f), \ref{assumption: pmc}(b) hold.  Then the probability that the equation $\nabla f_t^j(\hat \theta_t^j) = 0$ has at least one solution converges to $1$ as $t \to \infty$, and there exists a sequence of solutions $\hat \theta_t^j$ such that $\hat \theta_t^j \pto \theta_0$. 
\end{lemma}

This is a direct consequence of Theorem 5.42 of \cite{VanderVaart2000}; thus, the proof is omitted. {Lemma \ref{lemma:M-est-1}} states that the sequence of M estimators exists under third-order smoothness assumptions on the log - log-likelihood.  From now on, we assume the existence of M estimators $\hat \theta_t^j$ that satisfies \eqref{eqn:m-estimator} for every $t, j$. 

The high-order smoothness assumptions on the private log-likelihoods (Assumption (f)) is a strong assumption for guaranteeing consistency. It can be replaced with a more relaxed and amenable convexity assumption (Assumption \ref{assumption: regularity}(b)). 
\begin{lemma}\label{lemma:M-est-2}
 Let Assumptions  \ref{assumption: graph}, \ref{assumption: regularity}(a), \ref{assumption: regularity}(b) hold. Then $\hat \theta_t^j \pto \theta_0$. 
\end{lemma}

The next two lemmas are useful in the proof of Theorem~\ref{thm:bvm1}. 
\begin{lemma}\label{lemma:z-est-uct}
Let Assumptions \ref{assumption: graph}, \ref{assumption: regularity}(c), \ref{assumption: uct}(a) hold and $\hat \theta_t^j \pto \theta_0$. Then for every $\epsilon > 0$, there exists $\delta > 0$ such that 
\begin{equation*}
\lim_{t \to \infty} \bP_0(\inf_{\|\theta - \hat\theta_t^j\| > \delta}|f_t^j(\theta) - f_t^j(\hat\theta_t^j)| \geq \epsilon) = 1
\end{equation*}
\end{lemma}
The result states that if a sequence of estimators $\hat \theta_t^j$ is consistent at $\theta_0$, then the existence of uniformly consistent tests at $\theta_0$ (Assumption \ref{assumption: uct}(a)) implies analogous result at $\hat \theta_t^j$. 

\begin{lemma}
\label{lemma:bvm1}
Let $\hat \theta_t \in \R^p$ such that $\hat \theta_t \pto \theta_0$ for some $\theta_0 \in \R^p$, $\pi_t$ be a density with respect to Lebesgue measure on $\R^p$. Suppose $q_t$ is the density of $a_t (\theta - \hat \theta_t)$ where $\theta \sim \pi_t$ and $a_t^{-1} = o(1)$. If $\int |q_t(x) - q(x)|dx \to 0$ in $[\bP_0]-$probability for some probability density $q$, then for every $\epsilon > 0$, 
\begin{equation*}
    lim_{t \to \infty}\Pi_t(\mathcal N_\epsilon) = 1,
\end{equation*}
where $N_\epsilon$ is the neighborhood of $\theta_0$ defined in Corollary \ref{cor:const-1}. 
\end{lemma}

The proofs of Lemma~\ref{lemma:M-est-1},  Lemma~\ref{lemma:M-est-2},  Lemma~\ref{lemma:z-est-uct}, and Lemma~\ref{lemma:bvm1} can be found in Section~\ref{app: bvm} of the Appendix. Lemma \ref{lemma:M-est-1}  and \ref{lemma:M-est-2} relies on the usual consistency argument of M - estimators. Lemma \ref{lemma:z-est-uct} and Lemma \ref{lemma:bvm1} use measure-theoretic arguments.

Our first main result of the section provides general sufficient assumptions under which a distributed Bayes posterior exhibits asymptotic normality and an asymptotically correct Laplace approximation, along
with the posterior concentration at $\theta_0$. 
\begin{theorem}[BvM]
\label{thm:bvm1}
Let $\Theta$ be an open subset of $\R^p$. Assume that there exists $\theta_0 \in \Theta$ such that $P_0 = P_{\theta_0}^j$ for every $j \in [m]$. Moreover, let Assumptions \ref{assumption: graph},  \ref{assumption: regularity}(a), \ref{assumption: regularity}(c), \ref{assumption: regularity}(d), \ref{assumption: pmc}(b), \ref{assumption: uct}(a)  hold.  If the sequence $\hat \theta_t^j$ defined in \eqref{eqn:m-estimator} satisfies that $\hat \theta_t^j \pto \theta_0$, then 
\begin{equation} \label{eqn-bvm-1}
  f_t^j(\theta)  = f_t^j(\hat \theta_t^j) - \frac{1}{2}(\theta - \hat \theta_t^j)^T \hat V_t^j (\theta -\hat \theta_t^j) + r_t^j(\theta -\hat \theta_t^j),  
\end{equation}
where $\hat V_t^j$ is a sequence of matrices that converges in probability to the average Fisher information matrix $V_{\theta_0} = \frac{1}{m} \sum_{i = 1}^m V_{\theta_0}^i$, and $|r_t^j(h)| = O(|h|^3)$ for large enough $t$. 

Let Equation \eqref{eqn-bvm-1} and Assumption \ref{assumption: pmc}(b), \ref{assumption: uct}(a) hold. As $t \to \infty$, we have 
\begin{equation} \label{eqn-bvm-2}
   \int_{B_\epsilon(\theta_0)} p_t^j(\theta) d\theta \pto 1 \quad \forall \epsilon >0. 
\end{equation}
that is, the distributed Bayes posterior $p_t^j$ is weakly consistent around $\theta_0$.

Let $q_t^j$ be the density of $\sqrt{t}(\theta -\hat \theta_t^j)$ when $\theta \sim P_t^j$. Then, 
\begin{equation} \label{eqn-bvm-3}
    \int_\Theta \left|q_t^j(x) - N(0, V_{\theta_0}^{-1}) \right| dx \pto 0. 
\end{equation}
i.e., the total variational distance between $q_t^j$ and $N(0, V_{\theta_0}^{-1})$ vanishes in $[\bP_0]-$probability.  
\end{theorem}

See Section~\ref{app: bvm} for the proof. 
The assumptions of Theorem~\ref{thm:bvm1} involve structural and statistical assumptions that are in line with the Bernstein von Mises literature. These include a connected communication graph (Assumption \ref{assumption: graph}), bounded entropy condition of the private statistical models (Assumption \ref{assumption: regularity}(a)), a distributed version of the differentiable in quadratic means (DQM) assumption (Assumption \ref{assumption: regularity}(c)) and the Lipschitz gradient regularity assumption on the log-likelihood around $\theta_0$ (Assumption \ref{assumption: regularity}(d)). The M-estimators, $\hat \theta_t^j$, are assumed to be consistent, with sufficient conditions outlined in Lemmas \ref{lemma:M-est-1} and \ref{lemma:M-est-2}. These regularity assumptions are fairly mild and applicable to, for example, most exponential family models. They serve as the foundation for Equation \eqref{eqn-bvm-1}, which mirrors the local asymptotic normality assumption in classical BvM theory \citep{LeCam2000}. Additionally, we specify a prior mass assumption (Assumption \ref{assumption: pmc}(b)) and an assumption related to uniform, consistent testing (Assumption \ref{assumption: uct}(a)). The prior mass assumption (Assumption \ref{assumption: pmc}(b)) is slightly different from the one used in Theorem~\ref{thm:const}, but they are equivalent when $f$ is continuous at $\theta_0$. These, together with Equation \eqref{eqn-bvm-1}, are the main sufficient assumptions for the BvM results in Equations \eqref{eqn-bvm-2} and \eqref{eqn-bvm-3}.

The result establishes a sampling complexity of $\sqrt{t}$, indicating that all agents achieve \textbf{parametric efficiency} when performing distributed inference in a fully connected network comparable to a single-node network. While parametric efficiency is not surprising from BvM, our result also shows that communication among agents enhances the \textbf{robustness of uncertainty quantification}. In Theorem~\ref{thm:bvm1}, the precision matrix of the limiting Gaussian is the average of Fisher information $V_{\theta_0}$ across agents. The intuition here is that each agent’s uncertainty about the ground truth eventually aligns with the network’s average level of uncertainty, regardless of the agent’s initial uncertainty or the specific statistical model they use. This suggests that communication can act as a protective mechanism in adversarial settings, where the posterior of a few agents might suffer from abnormally high uncertainty due to data contamination.

In the social learning context, \Cref{thm:bvm1} extends existing asymptotic results to a large-time, continuous $\Theta$ setting. Previous results, such as \cite{Uribe2022-1,Uribe2022-2}, established consistency and concentration bounds for $p_t^j(\theta)$ when $\Theta$ is discrete or compact. In comparison, \Cref{thm:bvm1} assumes $\Theta$ to be an open subset of $\R^d$. Beyond recovering posterior consistency, \Cref{thm:bvm1} provides a richer asymptotic characterization of the distributed posterior in the sense that it guarantees 1) the frequentist coverage of posterior credible sets and 2) (rate-)efficient and robust inference under the distributed posterior.

The differentiability in quadratic means (DQM) assumption (Assumption \ref{assumption: regularity}(c)) in Theorem~\ref{thm:bvm1} may be difficult to verify in practical settings. We replace the abstract DQM assumptions with a second-order smoothness assumption on the private log-likelihoods (Assumption \ref{assumption: regularity}(e)). 

\begin{corollary} \label{cor:bvm1}
 Theorem~\ref{thm:bvm1} holds if  Assumption \ref{assumption: regularity}(c) is replaced with Assumption \ref{assumption: regularity}(e).
\end{corollary}

Both differentiability in quadratic means (DQM) assumption (Assumption \ref{assumption: regularity}(c)) and the Lipschitz gradient assumption (Assumption \ref{assumption: regularity}(d)) can be replaced with a third-order smoothness assumption (Assumption \ref{assumption: regularity}(f)) which is often called the classical condition for asymptotic normality of M - estimators \citep{VanderVaart2000}. 
\begin{corollary} \label{cor:bvm2}
 Theorem~\ref{thm:bvm1} holds if  Assumption \ref{assumption: regularity}(c) and \ref{assumption: regularity}(d) are replaced with Assumption \ref{assumption: regularity}(f).
\end{corollary}
The proofs of Corollary~\ref{cor:bvm1} and Corollary~\ref{cor:bvm2} are based on bounding the second and third-order terms in the Taylor expansion of $f_t^j(\theta)$, respectively. We can substitute more model regularity assumptions by bounding higher-order terms in the Taylor expansion. 

The classical Bernstein-von Mises theorem relies on a stochastic version of the Local Asymptotic Normality(LAN) assumption \citep{VanderVaart2000}. For social learning, we introduce a distributed version of the stochastic LAN condition. A distributed family of statistical models $\left(\{\bP_\theta^j\}_{j \in [m]}, G \right)$ satisfies \textit{stochastic LAN}  at $\theta \in \Theta$ if, for every $j \in [m]$, and with respect to a non-singular scaling factor $\epsilon_t^j \to 0$, there exists a random vector $\Delta_{t, \theta}^j$ and a non-singular matrix $V_\theta$ such that $\Delta_{t, \theta}^j$ is bounded in probability and for every compact subset $K \subset \R^p$, as $t \to \infty$, 
\begin{equation} \label{eqn:s-lan}
   \sup_{h \in K} \left| - t f_t^j(\theta + \epsilon_t^j h) + t f_t^j(\theta)  -h^T V_\theta \Delta_{t, \theta}^j   + \frac{1}{2} h^T V_\theta h \right| \pto 0. 
\end{equation}

The random vector $\Delta_{t, \theta}^j$ is often called the ``local sufficient statistics,'' and $V_\theta$ in our case corresponds to the average of Fisher information. It's worth noting that the $V_\theta$ matrix is required to be the same across all agents. The scaling factor $\epsilon_t^j$ is chosen to ensure $\Delta_{t, \theta}^j$ is bounded in probability, and $V_\theta$ converges to a nonsingular matrix. A default choice is $\epsilon_t^j = \frac{1}{\sqrt{t}}$. However, an agent-dependent choice of $\epsilon_t^j$ is available if any private statistical model has a different rate of contraction. 

A sufficient assumption for the distributed LAN is a distributed version of the differentiable in quadratic means (DQM) condition. 
\begin{lemma}
\label{lemma:sc-LAN}
Let Assumptions \ref{assumption: graph}, \ref{assumption: regularity}(a), \ref{assumption: regularity}(c) hold. The distributed family of statistical models, denoted as $\left(\{\bP_\Theta^j\}_{j \in [m]}, G \right)$, is stochastically locally asymptotically normal (LAN) at $\theta_0$.
\end{lemma}
we now extend the BvM result to the case where the true data-generating distribution $\bP_0 \neq \bP_{\theta_0}^j$ for some $j \in [m]$; in other words, we allow at least one agent to have a misspecified model. 
\begin{theorem} [Misspecified BvM]\label{thm: bvm2}
Let $\theta_0 \in (\Theta, d)$ be the true parameter defined in \Cref{def:theta_0}. Moreover, let Assumptions \ref{assumption: pmc}(b), \ref{assumption: uct}(b) hold, and assume that the distributed stochastic LAN \eqref{eqn:s-lan} holds at $\theta_0$. Let a sequence of constants $\epsilon_t^j$ satisfy that for any $M_t \to \infty$, 
\begin{equation} \label{assumption:bvm2}
    P_t^j\left(\theta: d(\theta, \theta_0) \geq M_t \epsilon_t^j \right) \pto 0. 
\end{equation}

If $q_t^j$ is the density of $(\theta - \theta_0)/\epsilon_t^j$ when $\theta \sim P_t^j$, then
\begin{equation*}\label{eqn:bvm2-main-result}
    \int_\Theta |q_t^j(x) - N(0, V_{\theta_0}^{-1})| dx \pto 0, 
\end{equation*}
for $V_{\theta_0}$ defined in \Cref{eqn:s-lan}. 
\end{theorem}
Since the covariance matrix $V_{\theta_0}$ in the misspecified Bernstein-Von Mises theorem fails to match the average of sandwich covariance matrices, the posterior credible sets derived from $P_t^j$ do not have valid frequentist coverage. While these sets may be properly centered at $\hat \theta_t^j$, their width may be inaccurate, and they don't typically correspond to confidence sets with level $1 - \alpha$. 

The sequence $\epsilon_t^j$ that satisfies \Cref{assumption:bvm2} is called a \textit{posterior contraction rate} of the distributed Bayes posterior $P_t^j$. This quantity determines the convergence rate of $P_t^j$ to the unknown parameter $\theta_0$. Results to control the contraction rates are provided in the next section.

%%%%%%%%%%%%%%%%%%%%%%%%%%%%%%%%%%%%%%%%%%%%%%%%%%%
%%%%%%%%%%%%%%%%%%%%%%%%%%%%%%%%%%%%%%%%%%%%%%%%%%%
%%%%%%%%%%%%%%%%%%%%%%%%%%%%%%%%%%%%%%%%%%%%%%%%%%%
\section{Contraction Rate} \label{sect: contraction}
%%%%%%%%%%%%%%%%%%%%%%%%%%%%%%%%%%%%%%%%%%%%%%%%%%%
%%%%%%%%%%%%%%%%%%%%%%%%%%%%%%%%%%%%%%%%%%%%%%%%%%%
%%%%%%%%%%%%%%%%%%%%%%%%%%%%%%%%%%%%%%%%%%%%%%%%%%%

Contraction rates quantify the speed at which a posterior distribution approaches the true parameter of the data-generating distribution. Controlling the contraction rates not only refines our understanding of posterior consistency but also helps control the sampling complexity in the misspecified Bernstein von - Mises Theorem (Theorem~\ref{thm: bvm2}). Unlike the previous sections, which focus on asymptotic results, we provide non-asymptotic bounds and scaling laws of the posterior contraction rates in this section. Our results involve the sample size $t$, dimension $p$, and the number of agents $m$.

For two positive sequences $x_t$ and $y_t$, we use $x_t \lesssim y_t$ to denote the existence of a constant $c$, independent of $n$, such that $x_t \leq c y_t$. Furthermore, we write $x_t \asymp y_t$ when $x_t \lesssim y_t$ and $y_t \lesssim x_t$. 

Let $(\Theta, d)$ be a metric space. A sequence of constants $\epsilon_t^j$ is a posterior contraction rate for $P_t^j$ at the parameter $\theta_0$ if, for every $M_t \to \infty$, 
\begin{equation} \label{def:contraction-rate}
    P_t^j\left(\theta: d(\theta, \theta_0) \geq M_t \epsilon_t^j \right) \pto 0, 
\end{equation}

The posterior contraction rate is not a unique quantity. Any rate slower than a contraction rate is also a contraction rate. Although the fastest rate is desirable, it may be hard to find. The natural goal is to establish a close upper bound for the "optimal" rate. We refer to this upper bound as the contraction rate. 

We define a probability measure $\bP_\theta$ on the product space $\X^m$. For a measurable set $A \subseteq \X^m$, we have
\begin{equation*}
    \bP_\theta(A) = \frac{\int_A  \prod_{j = 1}^m [\bp_{\theta}^j(x^j)]^{\frac{1}{m}} dx^1 \ldots d x^m}{\int_\Theta  \prod_{j = 1}^m [\bp_{\theta}^j(x^j)]^{\frac{1}{m}} dx^1 \ldots d x^m}. 
\end{equation*}
The density is given by $\bp_\theta(x) = \prod_{j = 1}^m [\bp_{\theta_0}^j(x^j)]^{\frac{1}{m}}$ for $x \in \X^m$. Given the prior measure $\Pi$,  we define the \textit{distributed ideal posterior} $P_t$ as the posterior measure corresponding to $\bP_\theta$ and $\Pi$. For a measurable set $B \subseteq \Theta$, we have 
\begin{equation} \label{def-ideal-posterior}
    P_t(B) =\frac{\int_B \prod_{k = 1}^t \bp_\theta(x_k) \Pi(d\theta)}{\int_\Theta \prod_{k = 1}^t \bp_\theta(x_k) \Pi(d\theta)} = \frac{\int_B \prod_{k = 1}^t \prod_{j = 1}^m [\bp_\theta^j(x_k^j)]^{\frac{1}{m}}\Pi(d\theta)}{\int_\Theta \prod_{k = 1}^t \prod_{j = 1}^m [\bp_\theta^j(x_k^j)]^{\frac{1}{m}}\Pi(d\theta)}, 
\end{equation}

Let $D_{\rho}(p \parallel q)$ denote the $\rho-$Rényi divergence, as defined in Section \eqref{subsect:divergence} of the Appendix. 
We establish a contraction rate of the distributed Bayes posterior$P_t^j$ given by:
\begin{equation} \label{eqn:contraction-rate}
    \epsilon_{m,t}^2 + \frac{1}{mt} \bP_0 D_{\text{KL}}(P_t^j \parallel P_t) + \frac{1}{m^2 t} \sum_{i = 1}^m D_{\text{KL}}(\bP_0 \parallel \bP_{\theta_0}^i).
\end{equation}

Here $\epsilon_{m,t}$ is the contraction rate of the distributed ideal posterior $P_t$. The second term quantifies the approximation error when the distributed Bayes posterior is approximated by the corresponding distributed ideal posterior under the true distribution $\bP_0$. The last term captures the minimum average discrepancy between the distributed models and the true data-generating distribution.

The main theorem of the section is stated under the “prior mass and testing” framework. The assumptions are sub-exponential refinements of the assumptions for Bernstein von - Mises theorems: (a) The prior is required to put a minimal amount of mass in a neighborhood of the true parameter. (b) Restricted to a subset of the parameter space, there exists a test function that can distinguish the truth from the complement of its neighborhood; (c) The prior is essentially supported on the subset described in (b). 
\begin{theorem}[Contraction Rate] \label{thm:contraction1}
Suppose $\epsilon_{m,t}$ is a sequence such that $mt\epsilon_{m,t}^2 \geq 1$. Let $C_0$, $C_1$, $C_2$, $C_3 > 0$ be constants such that $C_0 > C_2 + C_3 + 2$. Let the following assumptions hold:  
\begin{enumerate}
\item For any $\epsilon > \epsilon_{m,t}$, there exists a set $\Theta_t(\epsilon)$ and a testing function $\phi_t$ such that:
\begin{align}
\bP_{\theta_0} \phi_t(X^{(mt)}) + \sup_{\substack{\theta \in \Theta_t(\epsilon)\\ d(\theta, \theta_0) \geq C_1 \epsilon^2}} \bP_\theta (1 - \phi_t(X^{(mt)})) &\leq \exp\left(-C_0 t \epsilon^2\right). \tag{C1} \label{assumption:C1}
\end{align}

\item For any $\epsilon > \epsilon_{m,t}$, the set $\Theta_t(\epsilon)$ above satisfies:
\begin{equation}
\Pi\left(\Theta_t(\epsilon)^c\right) \leq \exp\left(-C_0 t \epsilon^2\right). \tag{C2} \label{assumption:C2}
\end{equation}

\item For some constant $\rho > 1$:
\begin{equation}
\Pi\left(\theta \in \Theta,  \frac{1}{m} \sum_{j = 1}^m D_\rho(\bP_{\theta_0}^j \parallel \bP_\theta^j) \leq C_3 \epsilon_{m,t}^2 \right) \geq \exp\left(-C_2 t \epsilon_{m,t}^2\right). \tag{C3} \label{assumption:C3}
\end{equation}
\end{enumerate}
 Then, for the distributed Bayes posterior $P_t^j$ defined in \eqref{def-db-posterior}, we have:
\begin{equation}\label{eqn:contraction11}
    \bP_0 P_t^j d(\theta, \theta_0) \leq C \left(\epsilon_{m,t}^2 + \gamma_{j,m,t}^2+ \frac{1}{m^2 t} \sum_{j = 1}^m D_{KL}(\bP_0 \parallel \bP_{\theta_0}^j) \right), 
\end{equation}

for some constant $C$ depending on $C_0, C_1$, where the quantity $\gamma_{j,m,t}^2$ is defined as:

\begin{equation*}
\gamma_{j,m,t}^2 =  \frac{1}{mt} \bP_0 D_{KL}(P_t^j \parallel P_t). 
\end{equation*}
\end{theorem}
The proof of Theorem~\ref{thm:contraction1} can be found in Section~\ref{app: contraction} of the Appendix. It is a consequence of support lemmas that make use of the Gibbs variational representation and the subexponential decay of sub-exponential decay of $d(\theta, \theta_0)$ under $P_t$. 

Assumption \eqref{assumption:C1} and  \eqref{assumption:C2} is a refinement of assumptions \ref{assumption: uct} for the uniform consistent testing and states that there is a sequence of tests such that the sum of Type I and Type II errors decrease exponentially with sample size, where the alternative hypothesis is taken in a large enough set under the prior. Assumption \eqref{assumption:C3} refines the prior mass assumptions \ref{assumption: pmc} by stating that the prior mass decreases exponentially away from a $\rho-$Rényi neighborhood of the true distribution. This assumption is slightly stronger than the equivalent assumption stated with the KL neighborhood because $D_\rho(P \parallel Q) > D_{KL}(P \parallel Q)$ for $\rho >1 $.   

The contraction rate is the sum of three terms. The first term $\epsilon_{m,t}^2$ is the contraction rate of the distributed ideal posterior. The second term $\gamma_{j,m,t}^2$ characterizes the distance between the distributed Bayes posterior $P_t^j$ and the ideal posterior $ P_t$. A larger or less connected communication graph means more deviation between the two distributions, which slows the contraction rate. The last term $\frac{1}{m^2 t} \sum_{j = 1}^m D_{KL}(\bP_0 \parallel \bP_{\theta_0}^j)$ penalizes the rate by the average discrepancy between the truth and its distributed approximation. 

We can show by Markov inequality that the upper bound on $ \bP_0 P_t^j d(\theta, \theta_0)$ is indeed the contraction rate for the distributed Bayes posterior $P_t^j$. This allows us to obtain a point estimate $\hat \theta$ that converges to the KL minimizer $\theta_0$ at the same rate for convex loss functions. 

\begin{corollary}\label{cor:contraction1}
    Under the assumptions of Theorem~\ref{thm:contraction1}, for any diverging sequence $M_t \to \infty$, we have 
    \begin{equation*}
         \bP_0 P_t^j \left(d(\theta, \theta_0) > M_t \left(\epsilon_{m,t}^2 + \gamma_{j,m,t}^2 +  \frac{1}{m^2 t} \sum_{j = 1}^m D_{KL}(\bP_0 \parallel \bP_{\theta_0}^j) \right) \right) \to 0. 
    \end{equation*}
    Furthermore, if $d(\theta, \theta_0)$ is convex in $\theta$, the distributed posterior posterior mean $\hat \theta = \int_\Theta \theta dP_t^j(\theta)$ satisfies 
    \begin{equation*}
       \bP_0  d(\hat \theta, \theta_0) \leq C\left(\epsilon_{m,t}^2 + \gamma_{j,m,t}^2 + \frac{1}{m^2 t} \sum_{j = 1}^m D_{KL}(\bP_0 \parallel \bP_{\theta_0}^j) \right), 
    \end{equation*}
where $C$ is the same constant in \eqref{eqn:contraction11}.  
\end{corollary}

The contraction rate defined in Theorem~\ref{thm:contraction1} is somewhat abstract because the terms such as $\epsilon_{m,t}^2$ and $\gamma_{j,m,t}^2$ do not directly inform the design of the underlying communication network. For practical purposes, it's preferable to characterize the rate in terms of design parameters such as \(m\), \(t\), and \(\nu\) to guide the design of a statistically efficient network structure. To this end, the following section offers more concrete upper bounds on the terms \(\epsilon_{m,t}^2\) and \(\gamma_{j,m,t}^2\).

The upper bounds for $\epsilon_{m,t}^2$ can be directly borrowed from the existing theory on posterior contraction rates under model misspecification \citep{Kleijn2012}. For finite-dimensional models denoted by $\{ \mathbb{P}_\theta, \theta \in \Theta \}$, the optimal contraction rate is given by $t^{-1/2}$. However, this result does not follow from Theorem~\ref{thm:contraction1} because it requires a more restrictive metric entropy assumption involving the $2^{nd}$-order KL divergence. See, for example, Theorem 2.2 of \cite{Kleijn2012}. 

The general theory for posterior contraction rates typically combines a prior mass assumption, often in the form of \ref{assumption:C3}, with either a model entropy assumption or a consistent testing assumption in the form of \ref{assumption:C1} and \ref{assumption:C2}. For a review of the theory of posterior contraction rates, see Chapter 8 of \cite{ghosal2017fundamentals} and the references therein.

As touched upon in Section~\ref{subsect:distributed-Bayes-posterior}, one should expect the distributed Bayes posterior to be well approximated by the distributed ideal posterior as the sample size increases. The next result provides a uniform bound on the approximation error $\gamma_{j,m,t}^2$. 
\begin{lemma}
\label{lemma:contraction11}
    Let Assumptions \ref{assumption: graph} and \ref{assumption: regularity}(a) hold. For $P_t^j$ defined in \eqref{def-db-posterior} and $P_t$ defined in \eqref{def-ideal-posterior}, we have
    \begin{equation*}
\gamma_{j,m,t}^2 \leq  \frac{16m \log m}{\nu t} \left(|\bP_0\log \bp_0| +  \max_{i \in [m]}\inf_{\theta \in \Theta} D_{KL}(\bP_0 \parallel \bP_\theta^i) \right)
    \end{equation*}
\end{lemma}
The upper bounds on $\gamma_{j,m,t}^2$ consist of two terms: the first term depends on the structure of the graph, and the second term depends on the graph and the worst-case model misspecification error, denoted by $\max_{i \in [m]}\inf_{\theta \in \Theta} D_{KL}(\bP_0 \parallel \bP_\theta^i)$. The intuition is that there is a tradeoff between the size of the communication network and statistical efficiency, and the communication cost could be much higher if one or more agents in the network use misspecified models. If all models are correctly specified, the contraction rate degrades with $m$ at a rate of $m \log m$. If the term $\max_{i \in [m]}\inf_{\theta \in \Theta} D_{KL}(\bP_0 \parallel \bP_\theta^i) $ scales at a rate of $m^2$, then the contraction rate decreases at a much higher rate of $m^3 \log m$. 

Suppose we ignore the constants and focus on the scaling law. In that case, the contraction rate of the distributed Bayes posterior is a function of the sample size, the number of agents, the smallest positive adjacency weight (spectral gap), the worst-case model misspecification error, and the average model misspecification error. We formalize this in the following result. 
\begin{theorem}[Practical Contraction Rate]\label{thm:contraction2}
Let Assumptions \ref{assumption: graph},  \ref{assumption: regularity}(a), and the assumptions of Theorem \eqref{thm:contraction1} hold. Let the distributed ideal posterior $P_t$ satisfies a contraction rate of $\epsilon_{m,t}^2 \lesssim t^{-1}$. For the distributed Bayes posterior $P_t^j$ defined in \eqref{def-db-posterior}, we have:
\begin{equation*}
    \bP_0 P_t^j d(\theta, \theta_0) \lesssim \frac{1}{t} + \frac{m \log m}{\nu t} \left(1 + \max_{i \in [m]}\inf_{\theta \in \Theta} D_{KL}(\bP_0 \parallel \bP_\theta^i) \right)  + \frac{1}{m^2 t} \sum_{j = 1}^m D_{KL}(\bP_0 \parallel \bP_{\theta_0}^j). 
\end{equation*}
\end{theorem}
The result follows directly from Theorem~\ref{thm:contraction1} and \ref{lemma:contraction11}. Thus, the proof is omitted.

This theorem outlines how design parameters determine the contraction rate of distributed Bayes posteriors. The second term of the formula is arguably the most interesting. As the number of agents $m$ increases, the contraction rate diminishes proportionally to $m \log m$ and inversely to $\nu$, the smallest positive adjacency weight. This shows that both the scale of the agent network and the minimal communication bandwidth (spectral gap) critically influence the posterior's contraction rate. Additionally, the error scales linearly in the worst-case and average model misspecification errors, which shows that deviations from the ideal model assumptions penalize the contraction rate. 

\section{Extension to Time-Varying Graphs} \label{sect:time-varying-graphs}
This section extends our theory to a specific time-varying connectivity scenario: an alternation between a fully connected graph and a “fully isolated agents” graph. While this is a simplistic setting, it is theoretically rich enough to reveal a \textbf{phase transition} behavior of inference under time-varying communication.

Let $G_t = (V, E_t), t \geq 1$ be a sequence of time-varying graphs, where the node set $V$ is fixed, but the edge set $E_t$ changes over time.
\begin{assumption}
\label{assumption: graph-time-varying1}
Let Assumption \ref{assumption: graph} be satisfied for graph $G$ with the adjacency matrix $A$. We assume that $G_t$ and $A_t$ are independent, random graphs and matrices such that $G_t = G$ and $A_t = A$ with probability $\lambda \in [0, 1]$, respectively, and $A_t = I_m$ otherwise, where $I_m$ is the identity matrix.
\end{assumption}

Assumption \ref{assumption: graph-time-varying1} proposes a setting useful for various reasons. Theoretically, this setting leads to a rigorous study of the tradeoff between statistical efficiency and communication cost. It allows us to explore what constitutes an "optimal" level of communication, given a targeted level of statistical efficiency. On the practical side, this setting can be conceptualized as a network experiencing complete failure with probability $\lambda$. In such scenarios, the primary objective is to quantify how this level of intermittent connectivity impacts the overall learning quality within the system.  

We provide an analogous result to Lemma \ref{lemma:graph_conv} in the Assumption~\ref{assumption: graph-time-varying1} setting.
\begin{proposition}
\label{prop: graph-time-varying1}
Let Assumption \ref{assumption: graph-time-varying1} hold and $m \geq 2$. Then with probability $1$, the sequence of adjacency matrices $(A_t)$ satisfies the following scaling law: 

If $\lambda \geq \frac{2}{m}$, then
\begin{equation*}
   \limsup_{t \to \infty} \sum_{k = 1}^t \sum_{j = 1}^m \left|\left[\prod_{\tau = k}^{t-1}A_\tau \right]_{ij}- \frac{1}{m} \right| \leq  \frac{16 m^2 \log m + 8m^2 \log \lambda}{\lambda \nu}, \qquad \forall i \in [m]
\end{equation*}

If $0<\lambda < \frac{2}{m}$, then
\begin{equation*}
   \limsup_{t \to \infty} \sum_{k = 1}^t \sum_{j = 1}^m \left|\left[\prod_{\tau = k}^{t-1}A_\tau \right]_{ij}- \frac{1}{m} \right| \leq  \frac{4m^3}{\nu}, \qquad \forall i \in [m]
\end{equation*}

If $\lambda = 0$, then
\begin{equation*}
   \limsup_{t \to \infty} \sum_{k = 1}^t \sum_{j = 1}^m \left|\left[\prod_{\tau = k}^{t-1}A_\tau \right]_{ij}- \frac{1}{m} \right|= \infty, \qquad \forall i \in [m]
\end{equation*}
In these inequalities, $\nu$ denotes the smallest positive entry of $A$.
\end{proposition}

Proposition \ref{prop: graph-time-varying1} provides the scaling laws for three communication regimes.  Consensus can be reached as long as the communication occurs with a non-zero probability. However, suppose the goal is to achieve optimal scaling behavior relative to the size of the communication graph ($m$). Then communication needs to happen with a minimum probability of $\frac{2}{m}$, with a higher frequency being more desirable. On the other hand, if there is a need to adhere to the low-frequency regime ($0<\lambda < \frac{2}{m}$) due to constraints such as high communication costs, the optimal strategy becomes one of minimal communication. Notably, within this regime, a decrease in communication frequency does not affect the rate of attaining consensus among the agents.
    
 Recall that the contraction rate of $P_t^j$ is decomposed into three pieces. 
\begin{equation*}
    \epsilon_{m,t}^2 + \frac{1}{mt} \bP_0 D_{\text{KL}}(P_t^j \parallel P_t) + \frac{1}{m^2 t} \sum_{i = 1}^m D_{\text{KL}}(\bP_0 \parallel \bP_{\theta_0}^i).
\end{equation*}
 The structure of the communication network affects the contraction rate through the second term. The following Corollary illustrates this.

\begin{corollary}
\label{cor:time-varying1}
    Let Assumptions \ref{assumption: graph-time-varying1} and \ref{assumption: regularity}(a) hold. For $P_t^j$ defined in \eqref{def-db-posterior} and $P_t$ defined in \eqref{def-ideal-posterior}, the following holds with probability $1$: 
    
    If $\lambda \geq \frac{2}{m}$, then
    \begin{equation*}
\frac{1}{mt} \bP_0 D_{\text{KL}}(P_t^j \parallel P_t) \leq  \frac{16 m \log m + 8m \log \lambda}{\lambda \nu t} \left(|\bP_0\log \bp_0| +  \max_{i \in [m]}\inf_{\theta \in \Theta} D_{KL}(\bP_0 \parallel \bP_\theta^i) \right). 
    \end{equation*}
If $0<\lambda < \frac{2}{m}$, then
    \begin{equation*}
\frac{1}{mt} \bP_0 D_{\text{KL}}(P_t^j \parallel P_t) \leq  \frac{4m^2}{\nu t}\left(|\bP_0\log \bp_0| +  \max_{i \in [m]}\inf_{\theta \in \Theta} D_{KL}(\bP_0 \parallel \bP_\theta^i) \right). 
    \end{equation*}
\end{corollary}

In Corollary \ref{cor:time-varying1}, the term "with probability \(1\)" is relative to the measure defined over the sequence \(A_1, A_2, \ldots\), where each \(A_t\) is an independent random matrix drawn from two deterministic matrices with probability \(\lambda\). Importantly, this is the same probability measure that underlies Proposition \ref{prop: graph-time-varying1}, and we continue to use this measure in the sequel.

Varying the communication frequency $\lambda$ affects the statistical efficiency of performing distributed Bayesian inference over the network. In the following result, we demonstrate the impact of $\lambda$ on the contraction rates. 
\begin{corollary}\label{cor:time-varying2}
Let Assumptions \ref{assumption: graph-time-varying1},  \ref{assumption: regularity}(a), and the assumptions of Theorem~\ref{thm:contraction1} hold.Let $\epsilon_{m,t}^2 \lesssim t^{-1}$. Then for the distributed Bayes posterior $P_t^j$ defined in \eqref{def-db-posterior}, the following holds with probability $1$:

If $\lambda \geq \frac{2}{m}$, then
\begin{equation*}
    \bP_0 P_t^j d(\theta, \theta_0) \lesssim  \frac{1}{t} + \frac{m \log m + m \log \lambda}{\lambda \nu t} \left(1 +  \max_{i \in [m]}\inf_{\theta \in \Theta} D_{KL}(\bP_0 \parallel \bP_\theta^i) \right) + \frac{1}{m^2 t} \sum_{j = 1}^m D_{KL}(\bP_0 \parallel \bP_{\theta_0}^j). 
\end{equation*}
If $0<\lambda < \frac{2}{m}$, then
\begin{equation*}
    \bP_0 P_t^j d(\theta, \theta_0) \lesssim  \frac{1}{t} +\frac{m^2}{\nu t} \left(1 +  \max_{i \in [m]}\inf_{\theta \in \Theta} D_{KL}(\bP_0 \parallel \bP_\theta^i) \right) + \frac{1}{m^2 t} \sum_{j = 1}^m D_{KL}(\bP_0 \parallel \bP_{\theta_0}^j). 
\end{equation*}
\end{corollary}
The distributed ideal posterior $P_t$ does not depend on the communication graph, thus Theorem~\ref{thm:contraction1} holds and the assumption that $\epsilon_{m,t}^2 \lesssim t^{-1}$ is still justified. Since the only term in the contraction rate that depends on the communication graph is $\gamma_{j,m,t}^2$, the argument for Corollary~\ref{cor:time-varying2} directly builds on Theorem~\ref{thm:contraction1} and Corollary~\ref{cor:time-varying1}.  See Section~\ref{app:time-varying-graphs} in the Appendix for the proof. 

\Cref{cor:time-varying2} is, to our knowledge, the first result on the impact of time-varying communication networks on the efficiency of distributed statistical inference. The result somewhat defies intuition. One might naturally expect a communication-statistical tradeoff, where increasing the communication frequency leads to faster posterior contraction rates. While higher communication frequencies ($\lambda \geq \frac{2}{m}$) do indeed result in faster contraction rates, scaled by a factor of $\frac{\log \lambda}{\lambda}$, the situation is different when the frequency is below $\frac{2}{m}$. In this regime, the communication efficiency $\lambda$ no longer has any effect, and the contraction rate depends solely on the number of agents $m$—with a rate of $m^2$ rather than $m \log m$.

This suggests the possibility of a ``phase transition'' phenomenon at $\lambda = \frac{2}{m}$, worth exploring for future research. In practical terms, if the goal is to balance communication efficiency and statistical performance in a large network, a communication frequency of $\lambda^* = \frac{2}{m}$ appears to be the optimal choice.
\section{Illustrative Examples} \label{sect:examples}
\subsection{Exponential Family Distributions} \label{subsect:EF}
 Let the distributed statistical models $\left(\{\bP_\Theta^j\}_{j \in [m]}, G \right)$ be well - specified, and the Assumption \ref{assumption: graph} be satisfied for $G$. Let $\eta: \Theta \to \R^p$ and $T: \mathcal{X} \to \R^p$ be some sufficient statistics and let $\psi^j: \Theta \to \R, h: \X \to \R$ be normalizing functions. We assume that $\bP_\theta^j$ is a member of the canonical exponential family 
\begin{equation} \label{eqn-EF-1}
    \bp_\theta^j(x) = h(x) \exp\left(\langle \theta, T^j(x)\rangle - \psi^j(\theta) \right). 
\end{equation}

The exponential family includes commonly used Gaussian, exponential, gamma, chi-square, Beta, Dirichlet, Bernoulli, categorical, Poisson, Wishart, inverse Wishart, and geometric distributions. In this section, we only consider the canonical exponential family and interchangeably refer to the exponential and canonical exponential families. 

Exponential family distributions are helpful for Bayesian inference because of their conjugate properties. For example, the posterior is a member of the exponential family if the prior and likelihood are members of the exponential family. This property is preserved in the distributed setting. 

Let us consider a scenario where the beliefs of all agents at time $t$ belong to the natural exponential family. In this setting, the belief of agent$i$ concerning the parameters$\theta$ can be described as follows:
\begin{equation*}
    p_t^i(\theta) \propto \exp\left( \langle \theta,  \chi_t^i \rangle - B^i(\theta) \right), 
\end{equation*}
where$\chi_t^i$ represents the sufficient statistic for agent$i$ at time$t$, and$A^i(\theta)$ is the log-partition function.

We update the one-step-ahead posterior $p_{t+1}^j$ with the distributed Bayes rule, 
\begin{equation*}
\begin{aligned}
        p_{t+1}^j (\theta) &\propto \log \bp_\theta(X_{t+1}^j) \prod_{i = 1}^m  p_t^i(\theta) \\
        &\propto \exp(\langle \theta, T(X_{t+1}^j) \rangle - \psi^j(\theta))\exp(\langle \theta,\sum_{i = 1}^m A_{ij} \chi_t^i\rangle - \sum_{i = 1}^m B^i(\theta)) \\
        &\propto  \exp\left(\langle \theta,T^j(X_{t+1}^j)  +  \sum_{i = 1}^m A_{ij} \chi_t^i \rangle - \psi^j(\theta) - \sum_{i = 1}^m B^i(\theta) \right). 
\end{aligned}
\end{equation*}
The distribution $p_{t+1}^j$ is a member of the exponential family with sufficient statistic $T^j(X_{t+1}^j)  +  \sum_{i = 1}^m A_{ij} \chi_t^i$ and the log-partition function $\psi^j(\theta) + \sum_{i = 1}^m B^i(\theta)$. This provides an easy-to-implement algorithm to learn the distributed Bayes posterior.

Let the prior be a member of the natural exponential family. 
\begin{equation} \label{eqn-EF-2}
    \pi(\theta) = g(u) \exp\left(\langle \theta, u\rangle - \psi^0(\theta) \right). 
\end{equation}
Leveraging the conjugate property of the exponential family, we derive a closed-form expression for the density of the distributed Bayes posterior $P_t^j$.

\begin{lemma}\label{lemma:EF-1}
Let assumption \eqref{assumption: graph} hold. Assume that the likelihood $\bp_\theta^i$ has an exponential family form given by Equation \eqref{eqn-EF-1}. Assume that the prior $\Pi$ has an exponential family form given by Equation \eqref{eqn-EF-2}.  Then the distributed Bayes posterior $P_t^j$ defined in \eqref{def-db-posterior} is given by the following formula: 
    \begin{equation}\label{eqn-EF-3}
         p_t^j (\theta) = h(X^{(mt)}) \exp(\langle \theta, \chi_t^j + u\rangle - B_t^j(\theta) - \psi^0(\theta)), 
    \end{equation}
    where 
    \begin{equation}\label{eqn-EF-4}
        \chi_t^j =   \sum_{k = 1}^t \sum_{i = 1}^m [A^{t-k}_{ji}] T^i(X_k^i),  \quad B_t^j(\theta) =  \sum_{k = 1}^t \sum_{i = 1}^m [A^{t-k}_{ji}] \psi^i(\theta) , 
    \end{equation}
    are the sufficient statistic and the log-partition function, respectively. 
\end{lemma}
An exponential family is full-rank if no linear combination of the sufficient statistic is constant. For example, we say the distribution $p_\theta^i$ in Equation \eqref{eqn-EF-1} is full-rank if no linear combination of the $D-$dimensional sufficient statistic $T^i(X) = [T^i_1(X), \cdots, T^i_D(X)]$ leads to a constant. This is a mild assumption on the form of the exponential family. 
\begin{lemma}\label{lemma:EF-2}
Assume that $\text{int}(\Theta) \neq \emptyset$. Assume that the likelihood $\bP_\theta^j$  given by Equation \eqref{eqn-EF-1}  and prior $\Pi$ given by Equation \eqref{eqn-EF-2} belong to exponential families with full rank. Then $P_t^j$ belongs to an exponential family of full rank. Moreover, the gradient of the log-partition function $\nabla B_t^j$ is invertible on the interior of $\Theta$. 
\end{lemma}
Lemma~\ref{lemma:EF-2} establishes the conditions under which the distributed Bayes posterior belongs to the full-rank exponential family. Critically, the invertibility of \(\nabla B_t^j\) is instrumental for constructing a sequence of consistent M-estimators for \(\theta_0\). These estimators serve as the centering sequence for the Laplace approximation, which is key in proving our Bernstein–von Mises result.

Let $\hat \theta_t^j$ be the M-estimators corresponding to $f_t^j$, as defined in Equation \eqref{def:f_t^j}.  
\begin{equation*}
    \hat \theta_t^j = \argmax_{\theta \in \R^p}\langle \theta, \chi_t^j \rangle - B_t^j(\theta). 
\end{equation*}
The closed-form expression for $\hat \theta_t^j$ can be obtained as follows:
\begin{equation}\label{eqn-EF-5}
    \hat \theta_t^j = (\nabla_\theta B_t^j)^{-1}(\chi_t^j), 
\end{equation}
where $\chi_t^j, B_t^j$ are defined in Equation \eqref{eqn-EF-4}.

We now state the asymptotic property of the distributed Bayes posteriors for exponential family distributions.  
\begin{proposition} \label{prop:EF-1}
Let $\Theta$ be an open subset of $\R^p$. Assume that there exists $\theta_0 \in \Theta$ such that $\bP_0 = \bP_{\theta_0}^j$ for every $j \in [m]$. Moreover, let Assumptions \ref{assumption: graph}, \ref{assumption: regularity}(a) hold. Let the sequence of estimators $\hat \theta_t^j$ be defined in Equation \eqref{eqn-EF-5}. Let $q_t^j$ be the density of $\sqrt{t}(\theta -\hat \theta_t^j)$ when $\theta \sim P_t^j$ Then, 
\begin{equation*}
    \int_\Theta |q_t^j(x) - N(0, V_{\theta_0}^{-1})| dx \pto 0, 
\end{equation*}
where $V_{\theta_0}$ is the average of the covariance of $T^i$ evaluated at $\theta_0$, i.e., $V_{\theta_0} = \frac{1}{m} \sum_{i = 1}^m \text{Cov}(T^i)$.  
\end{proposition}
%See Section~\ref{app:examples} for the proof. 
The BvM result suggests a sample-based Laplace approximation to the distributed Bayes posterior that belongs to the canonical exponential family. One example of such an approximation is the normal distribution with mean given by the moment estimator $\hat{\theta}_{t}^j$ and covariance given by as$\left(\frac{1}{m} \sum_{i=1}^{m} \hat V(T^{i})\right)^{-1}$, where $\hat V(T^{i})$ is the bootstrapped sample covariance of $T^{i}$. By Proposition \ref{prop:EF-1} and Slusky's theorem, we obtain that the total variational distance between $P_{t}^j$ and the normal distribution $N\left(\hat{\theta}_{t}^j, \left(\frac{1}{m} \sum_{i=1}^{m} \hat V(T^{i})\right)^{-1}\right)$ asymptotically converges to zero.

 Laplace approximation enables direct calculation of an asymptotically valid credible region.  Let $\chi_{\alpha, p}^2$ be the critical value from a $\chi^2$ distribution with $p$ degrees of freedom.  If $\theta \sim N\left(\hat{\theta}_{t}^j, V_{\theta_0}^{-1}\right)$, a credible region for $\theta$ at level $1 - \alpha$ is given by the set of $\theta$ that satisfies
\begin{equation*}
    (\hat{\theta}_{t}^j - \theta)^T \left(\frac{1}{m} \sum_{i=1}^{m} \hat V(T^{i})\right)  (\hat{\theta}_{t}^j - \theta)^T \leq \frac{\chi_{\alpha, p}^2}{t}, 
\end{equation*}
where $\hat V(T^{i})$ can be replaced by any consistent estimator for $\text{Cov}(T^i)$. By Proposition \ref{prop:EF-1} and Slusky's theorem, the credible region has asymptotic coverage of $1 - \alpha$ under $P_t^j$. 

\subsection{Distributed Logistic Regression} \label{subsect:logistic}
We consider i.i.d. observations $D_k^j = (X_k^j, Y_k^j)$ for $k \in [t] $ and $j \in [m]$ where covariates $X_k^j \in \X$ for $\X \subseteq \R^p$ and responses $Y_k^j \in \{0, 1\}$. Let $\theta_0 \in \R^p$ be the true and unknown parameter. A logistic regression model generates the data.  

\begin{equation}\label{eqn-logistic-1}
    Y_k^j \sim \text{Ber}(\nu_k^j), \quad \log \left( \frac{\nu_k^j}{1 - \nu_k^j} \right) = \theta_0^T X_k^j. 
\end{equation}

%In our simulations, the true parameter $\theta_0$ is a $p$-dimensional vector, where $p \in \{2, 50\}$. Each $p$-dimensional covariate vector $X_k^j$ is independently drawn from $\mathcal{N}(0, I_p)$. For in-depth analysis at $p=10$, refer to the Appendix. For each simulation run, $\beta^*$ is uniformly sampled from the $p$-dimensional unit cube $[0, 1]^p$. 

For a logistic regression model with coefficient $\theta$, the conditional distribution $\bp_\theta(y_k^j \mid x_k^j)$ after marginalizing out $\nu_k^j$ is expressed as:
\begin{equation}\label{eqn-logistic-3}
        \bp_\theta(y_k^j \mid x_k^j) = \exp\left(\langle \theta, x_k^j y_k^j \rangle - \sigma(\langle \theta, x_k^j \rangle)\right),
\end{equation}
where $\sigma(\eta) = \log(1 + e^\eta)$. The function $\sigma(\eta)$ is strictly convex, since $\sigma^{''}(\eta) = \frac{e^\eta}{(1 + e^\eta)^2} > 0$.  

Conditioned on the covariates, the model \eqref{eqn-logistic-3} aligns with the canonical form of the exponential family. This allows us to apply the results from Section~\ref{subsect:EF}.
\iffalse
Although the conditional distribution $p_\theta(y_k^j \mid x_k^j)$ is not a member of the exponential family in $\theta$, we can reparametrize it to fit the canonical form of the exponential family by introducing $\eta_k^j = \theta^T x_k^j$. Under this reparametrization, the distribution takes the form:
\begin{equation}
    \label{eqn-logistic-EF-1}
    p_{\eta_t^j}(y_k^j) = \exp\left(\eta_k^j y_k^j - \sigma(\eta_k^j)\right),
\end{equation}
\fi

For each agent $j$, we let the statistical model $\P_\Theta^j$ be a logistic regression model \eqref{eqn-logistic-3} with prior $\pi(\theta)$ supported on $\R^p$.   By aggregating the log-likelihood, we obtain the closed-form expression for the distributed Bayes posterior: 
\begin{equation} \label{eqn-logistic-3.5}
   p_t^j (\theta) = \exp\left(\langle \theta, T_t^j \rangle - B_t^j(\theta) - \log \pi(\theta)\right), 
\end{equation}
where 
\begin{equation*}
    T_t^j =   \sum_{k = 1}^t \sum_{i = 1}^m [A^{t-k}_{ji}] X_k^i Y_k^i,  \quad B_t^j(\theta) =  \sum_{k = 1}^t \sum_{i = 1}^m [A^{t-k}_{ji}] \sigma(\langle \theta,  X_k^i \rangle). 
\end{equation*}

In settings where sampling from the distribution \eqref{eqn-logistic-3.5} is expensive, a Laplace approximation serves as a useful surrogate for posterior inference. Next, we present the BvM result that provides this approximate distribution.

Let $\hat \theta_t^j$ be the solution to the following maximization problem: 
\begin{equation}\label{eqn-logistic-4}
    \hat \theta_t^j = \argmax_{\theta \in \R^p}\langle \theta, T_t^j \rangle - B_t^j(\theta). 
\end{equation}
The objective function in \eqref{eqn-logistic-4} is strictly concave, which guarantees the uniqueness of $\hat{\theta}_t^j$. Moreover, we can use gradient-based optimization methods to compute $\hat{\theta}_t^j$.

\begin{proposition}\label{prop:logistic-1}
Let $\Theta = \R^p$.  Let $D_k^j = (X_k^j, Y_k^j), k \in [t], j \in [m]$ be a sequence of i.i.d. paired random variables generated according to the model \eqref{eqn-logistic-1}. Let Assumptions \ref{assumption: graph}, \ref{assumption: regularity}(a) hold.  Under these conditions, the estimator sequence $\hat{\theta}_t^j$ defined in Equation \eqref{eqn-logistic-4} converges to $\theta_0$ in $[\bP_0]-$probability.

Define $q_t^j$ as the density of $\sqrt{t}(\theta - \hat{\theta}_t^j)$, where $\theta \sim P_t^j$. If the following assumptions are satisfied:
\begin{enumerate}[label=\roman*)]
    \item $\bP_{\theta_0} X_1^1 X_1^{1^T}$ exists and is finite and nonsingular. 
    \item $\bP_{\theta_0}|X_{k, a}^1 X_{k, b}^1 X_{k, c}^1| < \infty$ holds for all $a, b, c \in [p]$,
\end{enumerate}
then we have
\begin{equation}
\label{eqn-logistic-5}
\int_\Theta \left|q_t^j(x) - N(0, \hat V_{\theta_0}^{-1})\right| \, dx \pto 0.
\end{equation}
The matrix $\hat V_\theta$ is computed as
\begin{equation} \label{eqn-logistic-6}
\hat V_\theta = \frac{1}{m} \sum_{i = 1}^m {X^i}^T W^i(\theta) X^i,
\end{equation}
where $W^i(\theta)$ is the diagonal matrix defined by
\begin{equation*}
W^i(\theta) = \text{diag}\left(
\frac{e^{\sum_{j=0}^{p} \theta_j x_{1j}^i}}{\left(1 + e^{\sum_{j=0}^{p} \theta_j x_{1j}^i}\right)^2},
\ldots,
\frac{e^{\sum_{j=0}^{p} \theta_j x_{tj}^i}}{\left(1 + e^{\sum_{j=0}^{p} \theta_j x_{tj}^i}\right)^2}
\right).
\end{equation*}
\end{proposition}
%See Section~\ref{app:examples} for the proof. 
The proposition relies on two model-specific assumptions. Assumption i) focuses on the regularity of the private Fisher information. It is a standard prerequisite in Generalized Linear Models (GLMs) to ensure that the model parameter $\theta$ is identifiable (Example 16.8, \cite{VanderVaart2000}). Assumption ii) imposes a more stringent condition requiring a bounded third moment for the covariates. This assumption is used to verify Assumption \ref{assumption: regularity}(f) and is generally reasonable in most applications.  

Let $\chi_{\alpha, p}^2$ be the critical value from a $\chi^2$ distribution with $p$ degrees of freedom. Proposition \ref{prop:logistic-1} specifies an asymptotically valid credible region for the parameter $\theta$ under the probability measure $P_t^j$ at level $1 - \alpha$. The region is given by the set of $\theta$ that satisfies
\begin{equation*}
(\hat{\theta}_{t}^j - \theta)^T \hat V_{\theta_0} (\hat{\theta}_{t}^j - \theta)^T \leq \frac{\chi_{\alpha, p}^2}{t}, 
\end{equation*}
where $\hat{\theta}_{t}^j$ is the estimator in Equation \eqref{eqn-logistic-4} and $\hat V_{\theta_0}$ is defined in Equation \eqref{eqn-logistic-6}. 

Proposition \ref{prop:logistic-1} strengthens the robustness of distributed logistic regression models. The averaging Fisher information in the approximate covariance ensures that extreme covariate values do not overly influence the approximate uncertainty. The model-specific assumptions i) and ii) also serve as robustness checks for when the distributed logistic regression model attains the desired asymptotic property. In practice, checking the assumptions and using the Laplace approximation provides an efficient and reliable approach to statistical inference in distributed logistic regression models, making it more resilient to various data irregularities.

\subsection{Distributed Detection}
In this example, we extend the distributed source location scenario from \cite{Nedic2017} to demonstrate our theoretical results in a real-world context. Consider a communication network spread over a unit square $[0,1] \times [0,1]$, comprised of $m$ sensor agents located at points $Z^j, j \in [m]$. The agents are interconnected via a connected undirected graph $G$, represented by the adjacency matrix $A$. The goal is to locate a target at an unknown position $\theta_0 \in (0,1) \times (0,1)$. This target generates data $X_t^1, \ldots, X_t^m \sim \delta_{\theta_0}$ at time $t$, as seen by the agents. However, agents do not receive the data $X_t^j$ directly; instead, the $j^{th}$ agent receives a noisy version of $|X_t^j - Z^j|$, which is the Euclidean distance between the data point $X_t^j$ and the agent's location $Z^j$.

The goal is for the agents to identify the unknown location parameter, $\theta_0$ collectively. Due to noise and data corruption during transmission, the $j^{th}$ agent employs a statistical model to describe the observed signal $|X_t^j - Z^j|$. Specifically, this signal is modeled as a  normal distribution $N(|\theta - Z^j|, \sigma^{j^2}),\sigma^j > 0$ truncated to the interval $[0, |Z^j| + \frac{1}{2}]$. 

Let the parameter space be $\Theta = (0,1)^2$. For each $\theta \in \Theta$, the $j^{th}$ agent's statistical model is given by:
\begin{equation} \label{eqn-detection-1}
\bp_\theta^j(X_k^j) = \frac{\phi\left(\frac{|X_t^j - Z^j| - |\theta - Z^j|}{\sigma^j}\right)}{\sigma^j \left[\Phi\left(\frac{|Z^j| + \frac{1}{2} - |\theta - Z^j|}{\sigma^j}\right) - \Phi\left(\frac{- |\theta - Z^j|}{\sigma^j}\right)\right]}  I\left(0 \leq |X_t^j - Z^j| \leq |Z^j| + \frac{1}{2}\right).
\end{equation}
Before collecting any data, the agents collectively decide on a uniform prior for the unknown location parameter $\theta$: 
\begin{equation}\label{eqn-detection-2}
\pi(\theta) = I(\theta \in [0,1]^2).
\end{equation}

The statistical models from Equation \eqref{eqn-detection-1} are based on truncated normal distributions with fixed support, which belongs to the exponential family. The sufficient statistic for $|\theta - Z^j|$ is $|X_t^j - Z^j|$.

We can express the generalized likelihood function $f_t^j(\theta)$, up to a constant, as:

\begin{align*}
f_t^j(\theta) &= -\frac{1}{t} \sum_{k = 1}^t \sum_{i = 1}^m [A^{t-k}_{ji}] \left[\frac{(|X_t^j - Z^j| - |\theta - Z^j|)^2}{2 \sigma^{j^2}}  + \log \left(\Phi\left(\frac{|Z^j| + \frac{1}{2} - |\theta - Z^j|}{\sigma^j}\right) - \Phi\left(\frac{- |\theta - Z^j|}{\sigma^j}\right)\right) \right].
\end{align*}

The M-estimators for the loss function $f_t^j$ are denoted by $\hat{\theta}_t^j$, and they minimize $f_t^j(\theta)$ over the parameter space $[0,1]^2$:

\begin{equation}\label{eqn-detection-3}
    \hat{\theta}_t^j = \argmin_{\theta \in [0,1]^2} f_t^j(\theta).
\end{equation}

Due to the continuity and strict concavity of the objective function $f_t^j$, there exists a unique $\hat{\theta}_t^j$, which enables efficient numerical methods like the Newton-Raphson algorithm for its calculation.

Care is needed when $f_t^j$ achieves its minimum at the corners of $[0,1]^2$, namely at the points $(0,0), (0,1), (1,0), (1,1)$. In these scenarios, $\hat{\theta}_t^j$ does not belong to $\Theta$. Including this technicality is primarily to ensure that our definition of $\hat{\theta}_t^j$ satisfies the conditions of the Argmax Theorem (Theorem 3.2.2 in \cite{vaart2023empirical}).

Subsequent results confirm that the sequence of estimators $\hat{\theta}_t^j$ is consistent with $\theta_0$, allowing us to disregard the concern about the extreme corner cases with probability one.

\begin{lemma}\label{lemma:detection-1}
Let $\theta_0 \in \Theta$ be given.  Under Assumption \ref{assumption: regularity}(a), the sequence of estimators $\hat \theta_t^j$ defined in Equation \eqref{eqn-detection-3} converges to $\theta_0$ in $[\bP_0]-$probability. Moreover,  $\lim_{t \to \infty}\bP_0(\hat \theta_t^j \in \Theta) = 1$. 
\end{lemma}
We can use the Bernstein von - Mises theorem for the distributed location detection problem. 
\begin{proposition}\label{prop:detection-1}
Let the private statistical models be given by Equation \eqref{eqn-detection-1} and the prior be given by Equation \eqref{eqn-detection-2}. Let Assumption \ref{assumption: graph}, \ref{assumption: regularity}(a) hold.  For the sequence of estimators $\hat \theta_t^j$ defined in Equation \eqref{eqn-detection-3}, define $q_t^j$ as the density of $\sqrt{t}(\theta - \hat{\theta}_t^j)$, where $\theta \sim P_t^j$. Then we have
\begin{equation}
\int_\Theta \left|q_t^j(x) - N(0, V_{\theta_0}^{-1})\right| \, dx \pto 0, 
\end{equation}
where $V_{\theta_0}$ is given by
\begin{equation} \label{eqn-detection-4}
V_{\theta_0} = \frac{1}{m} \sum_{j = 1}^m \frac{(\theta_0 - Z^j) (\theta_0 - Z^j)^T}{\sigma^{j^4} |\theta_0 - Z^j|^2} \left[ \frac{\phi(\frac{|Z^j|+ \frac{1}{2} - |\theta_0 - Z^j|}{\sigma^j}) - \phi(\frac{- |\theta_0 - Z^j|}{\sigma^j})}{ \Phi(\frac{|Z^j|+ \frac{1}{2} - |\theta_0 - Z^j|}{\sigma^j}) - \Phi(\frac{- |\theta_0 - Z^j|}{\sigma^j}) } \right]^2. 
\end{equation}

\end{proposition}
Let $\chi_{\alpha, p}^2$ denote the critical value from a chi-squared distribution with $p$ degrees of freedom. Proposition \ref{prop:detection-1} delineates an asymptotically valid credible region for the parameter $\theta$ under the probability measure $P_t^j$ at a confidence level of $1 - \alpha$. The credible region is defined as all values of $\theta$ that satisfy the following inequality:
\begin{equation*}
    (\hat{\theta}_{t}^j - \theta)^\top \hat{V}_{\theta_0} (\hat{\theta}_{t}^j - \theta) \leq \frac{\chi_{\alpha, p}^2}{t},
\end{equation*}
where $\hat{\theta}_{t}^j$ represents the estimator defined in Equation \eqref{eqn-logistic-4}. Furthermore, $\hat{V}_{\theta_0}$ serves as the empirical analogue of the covariance matrix $V_{\theta_0}$, as defined in Equation \eqref{eqn-detection-4}, with the true parameter $\theta_0$ replaced by its estimate $\hat{\theta}_{t}^j$.

\section{Discussion and Future Directions } \label{sect:discussion}
We have studied the frequentist statistical properties of distributed (non-Bayesian) Bayesian inference in terms of ``posterior'' consistency, Bernstein von—Mises theorems, and ``posterior'' contraction rates. Our results provide the first rigorous insights into the statistical efficiency of distributed Bayesian inference and its dependence on the design parameters of the underlying communication network, such as the number of agents and the network topology. The promising results offer several avenues for future research.

Future work should investigate the frequentist statistical properties of the distributed Bayes posterior under more complex time-varying network structures, such as networks with single link failures \citep{Shahrampour2015} or under other random graph structures such as Erdős–Rényi graphs \citep{erdHos1960evolution}. Understanding how distributed Bayesian inference adapts to such scenarios will be crucial for applying theoretical insights in unpredictable or adversarial settings.

Another compelling direction is to explore the frequentist coverage properties of the distributed Bayes posterior. This could include analyses of confidence intervals or hypothesis tests when the underlying model is well-specified and misspecified. Understanding how the distributed Ba dyes posterior aligns with frequentist criteria can offer additional validation and robustness checks for the model.

There are various settings to extend our results within the context of social learning. For example, \cite{Bordignon2021,Hu2023} introduce a step size in the distributed update rule~\eqref{eqn:dSMD} and establish an asymptotic normality result under a vanishing step size with fixed $n$. One could combine our results with theirs to study limiting distributions under joint asymptotic regimes of time and step size. Another extension could involve proving a Bernstein-von Mises theorem for a discrete, finite parameter space $\Theta$, as explored in \cite{Bordignon2021,Hu2023}. Additionally, agent-specific weights could be introduced in the likelihood term during the posterior update, similar to \cite{Bordignon2023}, to explore how such an updating rule affects the limiting distribution and contraction rate compared to using agent-independent weights.

One could potentially connect our work with the microeconomics theory of social learning, such as the results in \cite{Molavi2018}. Insights from non-Bayesian frameworks like variants of the DeGroot model  \citep{degroot1974reaching,acemoglu2011bayesian} could shed light on the convergence properties of distributed Bayesian methods, especially in settings where agents have heterogeneous prior beliefs or are influenced by external signals.

Our work adds to in the literature on frequentist distributed inference that focuses on communication efficiency. Future studies could integrate insights from this body of work, including communication-efficient methods \citep{jordan2018communication} and high-dimensional distributed statistical inference \citep{battey2015distributed}, to design and analyze distributed Bayesian methods.

Lastly, our theoretical findings could inform the design of communication networks in practical applications. Specifically, an analytical understanding of how the number of agents and communication costs interact could guide the construction of more efficient and robust distributed Bayesian systems.

\bibliography{reference}

\begin{thebibliography}{55}
\providecommand{\natexlab}[1]{#1}
\providecommand{\url}[1]{\texttt{#1}}
\expandafter\ifx\csname urlstyle\endcsname\relax
  \providecommand{\doi}[1]{doi: #1}\else
  \providecommand{\doi}{doi: \begingroup \urlstyle{rm}\Url}\fi

\bibitem[Acemoglu et~al.(2011)Acemoglu, Dahleh, Lobel, and Ozdaglar]{acemoglu2011bayesian}
Daron Acemoglu, Munther~A Dahleh, Ilan Lobel, and Asuman Ozdaglar.
\newblock {B}ayesian learning in social networks.
\newblock \emph{The Review of Economic Studies}, 78\penalty0 (4):\penalty0 1201--1236, 2011.

\bibitem[Battey et~al.(2015)Battey, Fan, Liu, Lu, and Zhu]{battey2015distributed}
Heather Battey, Jianqing Fan, Han Liu, Junwei Lu, and Ziwei Zhu.
\newblock Distributed estimation and inference with statistical guarantees.
\newblock \emph{arXiv preprint arXiv:1509.05457}, 2015.

\bibitem[Bickel and Kleijn(2012)]{Bickel2012}
Peter~J Bickel and Bas J~K Kleijn.
\newblock The semiparametric {B}ernstein–von {M}ises theorem.
\newblock \emph{The Annals of Statistics}, 40:\penalty0 206--237, 2012.
\newblock ISSN 0090-5364.

\bibitem[Bickel and Yahav(1969)]{Bickel1969}
Peter~J Bickel and Joseph~A Yahav.
\newblock Some contributions to the asymptotic theory of {B}ayes solutions.
\newblock \emph{Zeitschrift für Wahrscheinlichkeitstheorie und verwandte Gebiete}, 11:\penalty0 257--276, 1969.
\newblock ISSN 1432-2064.

\bibitem[Bochkina and Green(2014)]{bochkina2014bernstein}
Natalia~A Bochkina and Peter~J Green.
\newblock The {B}ernstein--von {M}ises theorem and nonregular models.
\newblock \emph{The Annals of Statistics}, 42\penalty0 (5):\penalty0 1850--1878, 2014.

\bibitem[Bordignon et~al.(2021)Bordignon, Matta, and Sayed]{Bordignon2021}
Virginia Bordignon, Vincenzo Matta, and Ali~H. Sayed.
\newblock Adaptation in online social learning.
\newblock In \emph{European Signal Processing Conference}, volume 2021-January, 2021.
\newblock \doi{10.23919/Eusipco47968.2020.9287445}.

\bibitem[Bordignon et~al.(2023)Bordignon, Kayaalp, Matta, and Sayed]{Bordignon2023}
Virginia Bordignon, Mert Kayaalp, Vincenzo Matta, and Ali~H Sayed.
\newblock Social learning with non-{B}ayesian local updates.
\newblock In \emph{2023 31st European Signal Processing Conference (EUSIPCO)}, pages 1--5. IEEE, 2023.
\newblock ISBN 9464593601.

\bibitem[Borkar and Varaiya(1982)]{borkar1982asymptotic}
Vivek Borkar and Pravin Varaiya.
\newblock Asymptotic agreement in distributed estimation.
\newblock \emph{IEEE transactions on automatic control}, 27\penalty0 (3):\penalty0 650--655, 1982.

\bibitem[Braca et~al.(2010)Braca, Marano, Matta, and Willett]{Braca2010}
Paolo Braca, Stefano Marano, Vincenzo Matta, and Peter Willett.
\newblock Asymptotic optimality of running consensus in testing binary hypotheses.
\newblock \emph{IEEE Transactions on Signal Processing}, 58, 2010.
\newblock ISSN 1053587X.
\newblock \doi{10.1109/TSP.2009.2030610}.

\bibitem[Cam et~al.(2000)Cam, LeCam, and Yang]{LeCam2000}
Lucien~Le Cam, Lucien~Marie LeCam, and Grace~Lo Yang.
\newblock \emph{Asymptotics in Statistics: Some Basic Concepts}.
\newblock Springer Science \& Business Media, 2000.
\newblock ISBN 0387950362.

\bibitem[Castillo and Rousseau(2015)]{Castillo2015}
Ismaël Castillo and Judith Rousseau.
\newblock A {B}ernstein–von {M}ises theorem for smooth functionals in semiparametric models.
\newblock \emph{The Annals of Statistics}, 43:\penalty0 2353--2383, 2015.
\newblock ISSN 0090-5364.

\bibitem[Cover(1999)]{cover1999elements}
Thomas~M Cover.
\newblock \emph{Elements of Information Theory}.
\newblock John Wiley \& Sons, 1999.

\bibitem[DeGroot(1974)]{degroot1974reaching}
Morris~H DeGroot.
\newblock Reaching a consensus.
\newblock \emph{Journal of the American Statistical association}, 69\penalty0 (345):\penalty0 118--121, 1974.

\bibitem[Epstein et~al.(2010)Epstein, Noor, and Sandroni]{Epstein2010}
Larry~G Epstein, Jawwad Noor, and Alvaro Sandroni.
\newblock Non-{B}ayesian learning.
\newblock \emph{The BE Journal of Theoretical Economics}, 10:\penalty0 1--20, 2010.

\bibitem[Erd{\H{o}}s et~al.(1960)Erd{\H{o}}s, R{\'e}nyi, et~al.]{erdHos1960evolution}
Paul Erd{\H{o}}s, Alfr{\'e}d R{\'e}nyi, et~al.
\newblock On the evolution of random graphs.
\newblock \emph{Publ. math. inst. hung. acad. sci}, 5\penalty0 (1):\penalty0 17--60, 1960.

\bibitem[Ghosal and Van~der Vaart(2017)]{ghosal2017fundamentals}
Subhashis Ghosal and Aad Van~der Vaart.
\newblock \emph{Fundamentals of {N}onparametric {B}ayesian Inference}, volume~44.
\newblock Cambridge University Press, 2017.

\bibitem[Ghosh and Ramamoorthi(2003)]{ghoshramamoorthi}
Jayanta~K. Ghosh and R.V. Ramamoorthi.
\newblock \emph{{B}ayesian Nonparametrics}.
\newblock Springer-Verlag, 2003.
\newblock \doi{10.1007/b97842}.
\newblock URL \url{https://doi.org/10.1007\%2Fb97842}.

\bibitem[Gilardoni and Clayton(1993)]{gilardoni1993reaching}
Gustavo~L Gilardoni and Murray~K Clayton.
\newblock On reaching a consensus using degroot's iterative pooling.
\newblock \emph{The Annals of Statistics}, 21\penalty0 (1):\penalty0 391--401, 1993.

\bibitem[Golub and Sadler(2017)]{Golub2017}
Benjamin Golub and Evan Sadler.
\newblock Learning in social networks.
\newblock \emph{Available at SSRN 2919146}, 2017.

\bibitem[Gubner(1993)]{gubner1993distributed}
John~A Gubner.
\newblock Distributed estimation and quantization.
\newblock \emph{IEEE Transactions on Information Theory}, 39\penalty0 (4):\penalty0 1456--1459, 1993.

\bibitem[Hoffman et~al.(2013)Hoffman, Blei, Wang, and Paisley]{hoffman2013stochastic}
Matthew~D Hoffman, David~M Blei, Chong Wang, and John Paisley.
\newblock Stochastic variational inference.
\newblock \emph{Journal of Machine Learning Research}, 2013.

\bibitem[Hu et~al.(2023)Hu, Bordignon, Vlaski, and Sayed]{Hu2023}
Ping Hu, Virginia Bordignon, Stefan Vlaski, and Ali~H. Sayed.
\newblock Optimal aggregation strategies for social learning over graphs.
\newblock \emph{IEEE Transactions on Information Theory}, 69, 2023.
\newblock ISSN 15579654.
\newblock \doi{10.1109/TIT.2023.3281647}.

\bibitem[Inan et~al.(2022)Inan, Kayaalp, Telatar, and Sayed]{Inan2022}
Yunus Inan, Mert Kayaalp, Emre Telatar, and Ali~H. Sayed.
\newblock Social learning under randomized collaborations.
\newblock In \emph{IEEE International Symposium on Information Theory - Proceedings}, volume 2022-June, 2022.
\newblock \doi{10.1109/ISIT50566.2022.9834621}.

\bibitem[Jadbabaie et~al.(2012)Jadbabaie, Molavi, Sandroni, and Tahbaz-Salehi]{Jadbabaie2012}
Ali Jadbabaie, Pooya Molavi, Alvaro Sandroni, and Alireza Tahbaz-Salehi.
\newblock Non-{B}ayesian social learning.
\newblock \emph{Games and Economic Behavior}, 76:\penalty0 210--225, 2012.
\newblock ISSN 0899-8256.

\bibitem[Jordan et~al.(2018)Jordan, Lee, and Yang]{jordan2018communication}
Michael~I Jordan, Jason~D Lee, and Yun Yang.
\newblock Communication-efficient distributed statistical inference.
\newblock \emph{Journal of the American Statistical Association}, 2018.

\bibitem[Katsevich(2023)]{katsevich2023improved}
Anya Katsevich.
\newblock Improved scaling with dimension in the {B}ernstein-von {M}ises theorem for two statistical models.
\newblock \emph{arXiv preprint arXiv:2308.06899}, 2023.

\bibitem[Kleijn and van~der Vaart(2012)]{Kleijn2012}
Bas J~K Kleijn and Aad~W van~der Vaart.
\newblock The {B}ernstein-von-{M}ises theorem under misspecification.
\newblock \emph{Electronic Journal of Statistics}, 6:\penalty0 354--381, 2012.
\newblock ISSN 1935-7524.

\bibitem[Knoblauch et~al.(2022)Knoblauch, Jewson, and Damoulas]{knoblauch2022optimization}
Jeremias Knoblauch, Jack Jewson, and Theodoros Damoulas.
\newblock An optimization-centric view on {B}ayes’ rule: Reviewing and generalizing variational inference.
\newblock \emph{Journal of Machine Learning Research}, 23\penalty0 (132):\penalty0 1--109, 2022.

\bibitem[Lalitha et~al.(2014)Lalitha, Sarwate, and Javidi]{lalitha2014social}
Anusha Lalitha, Anand Sarwate, and Tara Javidi.
\newblock Social learning and distributed hypothesis testing.
\newblock In \emph{2014 IEEE International Symposium on Information Theory}, pages 551--555. IEEE, 2014.

\bibitem[Laplace(1809)]{Laplace1809}
Pierre-Simon Laplace.
\newblock Supplément au mémoire sur les approximations des formules qui sont fonction de très grands nombres.
\newblock \emph{Mémoires de l’Académie Royale des sciences de Paris}, 1810:\penalty0 559--565, 1809.

\bibitem[LeCam(1953)]{LeCam1953}
Lucien LeCam.
\newblock On some asymptotic properties of maximum likelihood estimates and related {B}ayes estimates.
\newblock \emph{Univ. California Pub. Statist.}, 1:\penalty0 277--330, 1953.

\bibitem[Li et~al.(2017)Li, Srivastava, and Dunson]{Li2017}
Cheng Li, Sanvesh Srivastava, and David~B Dunson.
\newblock Simple, scalable and accurate posterior interval estimation.
\newblock \emph{Biometrika}, 104:\penalty0 665--680, 2017.
\newblock ISSN 0006-3444.

\bibitem[Medina et~al.(2022)Medina, Olea, Rush, and Velez]{medina2022robustness}
Marco~Avella Medina, Jos{\'e} Luis~Montiel Olea, Cynthia Rush, and Amilcar Velez.
\newblock On the robustness to misspecification of $\alpha$-posteriors and their variational approximations.
\newblock \emph{Journal of Machine Learning Research}, 23\penalty0 (147):\penalty0 1--51, 2022.

\bibitem[Miller(2021)]{miller2021asymptotic}
Jeffrey~W Miller.
\newblock Asymptotic normality, concentration, and coverage of generalized posteriors.
\newblock \emph{J. Mach. Learn. Res.}, 22:\penalty0 168--1, 2021.

\bibitem[Minsker et~al.(2014)Minsker, Srivastava, Lin, and Dunson]{Minsker2014}
Stanislav Minsker, Sanvesh Srivastava, Lizhen Lin, and David Dunson.
\newblock Scalable and robust {B}ayesian inference via the median posterior.
\newblock In \emph{International conference on machine learning}, pages 1656--1664. PMLR, 2014.

\bibitem[Minsker et~al.(2017)Minsker, Srivastava, Lin, and Dunson]{Minsker2017}
Stanislav Minsker, Sanvesh Srivastava, Lizhen Lin, and David~B. Dunson.
\newblock Robust and scalable {B}ayes via a median of subset posterior measures.
\newblock \emph{J. Mach. Learn. Res.}, 18:\penalty0 Paper No. 124, 40, 2017.
\newblock ISSN 1532-4435,1533-7928.

\bibitem[Molavi et~al.(2018)Molavi, Tahbaz‐Salehi, and Jadbabaie]{Molavi2018}
Pooya Molavi, Alireza Tahbaz‐Salehi, and Ali Jadbabaie.
\newblock A theory of non‐{B}ayesian social learning.
\newblock \emph{Econometrica}, 86:\penalty0 445--490, 2018.
\newblock ISSN 0012-9682.

\bibitem[Nedić et~al.(2017)Nedić, Olshevsky, and Uribe]{Nedic2017}
Angelia Nedić, Alex Olshevsky, and César~A Uribe.
\newblock Fast convergence rates for distributed non-{B}ayesian learning.
\newblock \emph{IEEE Transactions on Automatic Control}, 62:\penalty0 5538--5553, 2017.
\newblock ISSN 0018-9286.

\bibitem[Neiswanger et~al.(2013)Neiswanger, Wang, and Xing]{Neiswanger2013}
Willie Neiswanger, Chong Wang, and Eric Xing.
\newblock Asymptotically exact, embarrassingly parallel {MCMC}.
\newblock \emph{arXiv preprint arXiv:1311.4780}, 2013.

\bibitem[Panov and Spokoiny(2015)]{panov2015}
Maxim Panov and Vladimir Spokoiny.
\newblock Finite sample {B}ernstein--von {M}ises theorem for semiparametric problems.
\newblock \emph{{B}ayesian Analysis}, 10\penalty0 (3):\penalty0 665--710, 2015.

\bibitem[Rabinovich et~al.(2015)Rabinovich, Angelino, and Jordan]{Rabinovich2015}
Maxim Rabinovich, Elaine Angelino, and Michael~I Jordan.
\newblock Variational consensus {M}onte {C}arlo.
\newblock \emph{Advances in Neural Information Processing Systems}, 28, 2015.

\bibitem[Rosenthal(1995)]{rosenthal1995convergence}
Jeffrey~S Rosenthal.
\newblock Convergence rates for {M}arkov chains.
\newblock \emph{Siam Review}, 37\penalty0 (3):\penalty0 387--405, 1995.

\bibitem[Schwartz(1965)]{Schwartz1965}
Lorraine Schwartz.
\newblock On {B}ayes procedures.
\newblock \emph{Zeitschrift für Wahrscheinlichkeitstheorie und verwandte Gebiete}, 4:\penalty0 10--26, 1965.
\newblock ISSN 1432-2064.

\bibitem[Scott et~al.(2016)Scott, Blocker, Bonassi, Chipman, George, and McCulloch]{Scott2016}
Steven~L Scott, Alexander~W Blocker, Fernando~V Bonassi, Hugh~A Chipman, Edward~I George, and Robert~E McCulloch.
\newblock {B}ayes and big data: The consensus {M}onte {C}arlo algorithm.
\newblock \emph{International Journal of Management Science and Engineering Management}, 11:\penalty0 78--88, 2016.
\newblock ISSN 1750-9653.

\bibitem[Shahrampour et~al.(2015)Shahrampour, Rakhlin, and Jadbabaie]{Shahrampour2015}
Shahin Shahrampour, Alexander Rakhlin, and Ali Jadbabaie.
\newblock Distributed detection: Finite-time analysis and impact of network topology.
\newblock \emph{IEEE Transactions on Automatic Control}, 61:\penalty0 3256--3268, 2015.
\newblock ISSN 0018-9286.

\bibitem[Shen(2002)]{Shen2002}
Xiaotong Shen.
\newblock Asymptotic normality of semiparametric and nonparametric posterior distributions.
\newblock \emph{Journal of the American Statistical Association}, 97:\penalty0 222--235, 2002.
\newblock ISSN 0162-1459.

\bibitem[Tsitsiklis and Athans(1984)]{tsitsiklis1984convergence}
John Tsitsiklis and Michael Athans.
\newblock Convergence and asymptotic agreement in distributed decision problems.
\newblock \emph{IEEE Transactions on Automatic Control}, 29\penalty0 (1):\penalty0 42--50, 1984.

\bibitem[Uribe et~al.(2022{\natexlab{a}})Uribe, Olshevsky, and Nedich]{Uribe2022-1}
Cesar~A Uribe, Alexander Olshevsky, and Angelia Nedich.
\newblock Non-asymptotic concentration rates in cooperative learning part i: variational non-{B}ayesian social learning.
\newblock \emph{IEEE Transactions on Control of Network Systems}, 2022{\natexlab{a}}.
\newblock ISSN 2325-5870.

\bibitem[Uribe et~al.(2022{\natexlab{b}})Uribe, Olshevsky, and Nedich]{Uribe2022-2}
Cesar~A Uribe, Alexander Olshevsky, and Angelia Nedich.
\newblock Non-asymptotic concentration rates in cooperative learning part ii: inference on compact hypothesis sets.
\newblock \emph{IEEE Transactions on Control of Network Systems}, 2022{\natexlab{b}}.
\newblock ISSN 2325-5870.

\bibitem[Van~der Vaart(2000)]{VanderVaart2000}
Aad~W. Van~der Vaart.
\newblock \emph{Asymptotic Statistics}, volume~3.
\newblock Cambridge University Press, 2000.
\newblock ISBN 0521784506.

\bibitem[van~der Vaart and Wellner(2023)]{vaart2023empirical}
Aad~W. van~der Vaart and Jon~A. Wellner.
\newblock \emph{Weak Convergence and Empirical Processes}.
\newblock Springer Series in Statistics. Springer-Verlag, New York, 2023.
\newblock ISBN 0-387-94640-3.

\bibitem[Walker(2006)]{Walker2006}
Stephen~G Walker.
\newblock {B}ayesian inference via a minimization rule.
\newblock \emph{Sankhyā: The Indian Journal of Statistics}, pages 542--553, 2006.
\newblock ISSN 0972-7671.

\bibitem[Wang and Dunson(2013)]{Wang2013}
Xiangyu Wang and David~B Dunson.
\newblock Parallelizing {MCMC} via {W}eierstrass sampler.
\newblock \emph{arXiv preprint arXiv:1312.4605}, 2013.

\bibitem[Wang et~al.(2015)Wang, Guo, Heller, and Dunson]{Wang2015}
Xiangyu Wang, Fangjian Guo, Katherine~A Heller, and David~B Dunson.
\newblock Parallelizing {MCMC} with random partition trees.
\newblock \emph{Advances in neural information processing systems}, 28, 2015.

\bibitem[Wang and Blei(2019)]{wang2019Frequentist}
Yixin Wang and David~M Blei.
\newblock Frequentist consistency of variational {B}ayes.
\newblock \emph{Journal of the American Statistical Association}, 114\penalty0 (527):\penalty0 1147--1161, 2019.

\end{thebibliography}
\appendix
\section{Review of concepts}
\subsection{Distances and Divergences}\label{subsect:divergence}
This section reviews a class of divergence functions. 
\begin{definition}[Rényi divergence]
Let $\rho > 0$ and $\rho \neq 1$. The $\rho$-Rényi divergence between two probability measures $P_1$ and $P_2$ is defined as
\begin{equation*}
D_\rho(P_1 \parallel P_2) =
\begin{cases}
\frac{1}{\rho - 1} \log \int \left(\frac{dP_1}{dP_2}\right)^{\rho -1} dP_1 & \text{if } P_1 \ll P_2, \\
+\infty & \text{otherwise.}
\end{cases}
\end{equation*}
The relations between the Rényi divergence and other divergence functions are summarized below:
\begin{enumerate}
\item When $\rho \to 1$, the Rényi divergence converges to the Kullback--Leibler divergence, defined as
\begin{equation*}
D_{KL}(P_1 \parallel P_2) =
\begin{cases}
\int \log \left(\frac{dP_1}{dP_2}\right) dP_1 & \text{if } P_1 \ll P_2, \\
+\infty & \text{otherwise.}
\end{cases}
\end{equation*}
\item When $\rho = 1/2$, the Rényi divergence is related to the Hellinger distance by
\begin{equation*}
D_{1/2}(P_1 \parallel P_2) = -2 \log \left(1 - H(P_1,P_2)\right),
\end{equation*}
and the Hellinger distance is defined as
\begin{equation*}
H(P_1,P_2) = \sqrt{\frac{1}{2} \int \left(\sqrt{dP_1} - \sqrt{dP_2}\right)^2}.
\end{equation*}
\item When $\rho = 2$, the Rényi divergence is related to the $\chi^2$-divergence by
\begin{equation*}
D_2(P_1 \parallel P_2) = \log \left(1 + \chi^2(P_1 \parallel P_2)\right),
\end{equation*}
and the $\chi^2$-divergence is defined as
\begin{equation*}
\chi^2(P_1 \parallel P_2) = \int \left(\frac{(dP_1)^2}{dP_2} - 1\right).
\end{equation*}
\end{enumerate}
\end{definition}

\begin{definition}[Total variation]
The total variation distance between two probability measures $P_1$ and $P_2$ is defined as
\begin{equation*}
TV(P_1, P_2) = \frac{1}{2} \int |dP_1 - dP_2|.
\end{equation*}
\end{definition}

\subsection{An (incomplete) Summary of Previous BvM Results} \label{subsect:bvm-history}
The history of the Bernstein von Mises (BvM) theorem is marked by steady evolution, expanding its reach to wider contexts with weaker assumptions.  The origins of the Berstein von Mises (BvM) theorem trace back to \cite{Laplace1809}.  Earlier works on the BvM theorem are restricted to well-specified, i.i.d statistical models with fixed, finite-dimensional parameters and use assumptions that involve up to fourth-order derivatives of the log-likelihood \citep{LeCam1953, Bickel1969, LeCam2000}). 

This classical Bernstein von - Mises theorem for i.i.d., finite-dimensional data was brilliantly summarized in~\cite[Chapter 10]{VanderVaart2000}. The Theorem relies on three assumptions: the existence of a sequence of consistent estimators, local asymptotic normality (LAN), and uniform consistent testing. One typical sufficient assumption for LAN is differentiable in quadratic means (DQM), which involves only first-order derivatives of the log-likelihood to achieve a desirable quadratic expansion. In addition to the LAN assumption, Schwartz \cite{Schwartz1965} proposed the concept of uniformly consistent test assumptions around the true parameter $\theta_0$. 

The contemporary wave of research has extended BvM theorems beyond the canonical settings. These studies have broached areas such as model misspecification \citep{Kleijn2012}, semiparametric models \cite{Shen2002, Bickel2012, panov2015, Castillo2015}, and parameters placed at the boundary of the parameter space \citep{bochkina2014bernstein}. Among these explorations, a particularly relevant line of inquiry has focused on BvM for non-standard Bayes procedures \citep{knoblauch2022optimization}, including generalized posteriors \citep{miller2021asymptotic}, Bayes rules derived from optimization perspectives, and variational Bayes posteriors \citep{wang2019Frequentist, medina2022robustness}. Another emerging thread of work establishes BvM theorems for high-dimensional models, with most recent results allowing the dimension to grow at the order $d^2 \lesssim n$ \citep{panov2015,katsevich2023improved}.  

\section{Simulating distributed Bayesian logistic regression}
We illustrate the practical implication of our theory through an exponential family model, where the results are explicitly proven in \Cref{subsect:EF}. 

We posit an improper prior $\pi(\theta) \propto 1$. 
Then the distributed posterior is given by
\begin{equation}
   p_t^j (\theta) = \exp\left(\langle \theta, T_t^j \rangle - B_t^j(\theta)\right), 
\end{equation}
where 
\begin{equation*}
    T_t^j =   \sum_{k = 1}^t \sum_{i = 1}^m [A^{t-k}_{ji}] X_k^i Y_k^i,  \quad B_t^j(\theta) =  \sum_{k = 1}^t \sum_{i = 1}^m [A^{t-k}_{ji}] \sigma(\langle \theta,  X_k^j \rangle). 
\end{equation*}
The gradient of the energy function is given by
\begin{equation*}
   - \nabla_\theta \log p_t^j(\theta) = - \sum_{k = 1}^t \sum_{i = 1}^m [A^{t-k}_{ji}] X_k^i Y_k^i +  \sum_{k = 1}^t \sum_{i = 1}^m [A^{t-k}_{ji}] \frac{\exp\left( \langle \theta,  X_k^j \rangle\right)}{1 + \exp\left( \langle \theta,  X_k^j \rangle\right)} X_k^j , 
\end{equation*}
where $A^0_{jj} = 1, A^0_{ji} = 0$ for $i \neq j$. 

With $ - \nabla_\theta \log p_t^j(\theta)$ in hand, we implement the Langevin Monte Carlo (LMC) algorithm to sample from the distributed posterior $\log p_t^j(\theta)$. Since $\log p_t^j(\theta)$ is strictly concave, the samples generated by LMC are guaranteed to converge to the stationary distribution $p_t^j(\theta)$. The final histograms were generated using 1,000 samples from LMC. \Cref{fig:BvM-hist} shows the 10 marginals for the exact posterior sampled under LMC and the BvM approximation derived in \Cref{prop:logistic-1} for one agent, which are closely aligned.

For the problem setting, we simulate a static connected graph with 10 nodes, generating 500 data points under a logistic regression model with the ground truth $\theta_0 \in \R^5$ sampled from a standard normal distribution.

\begin{figure}[th]
\centering
\begin{minipage}{0.95\linewidth}
    \centering
    \includegraphics[width=1\linewidth]{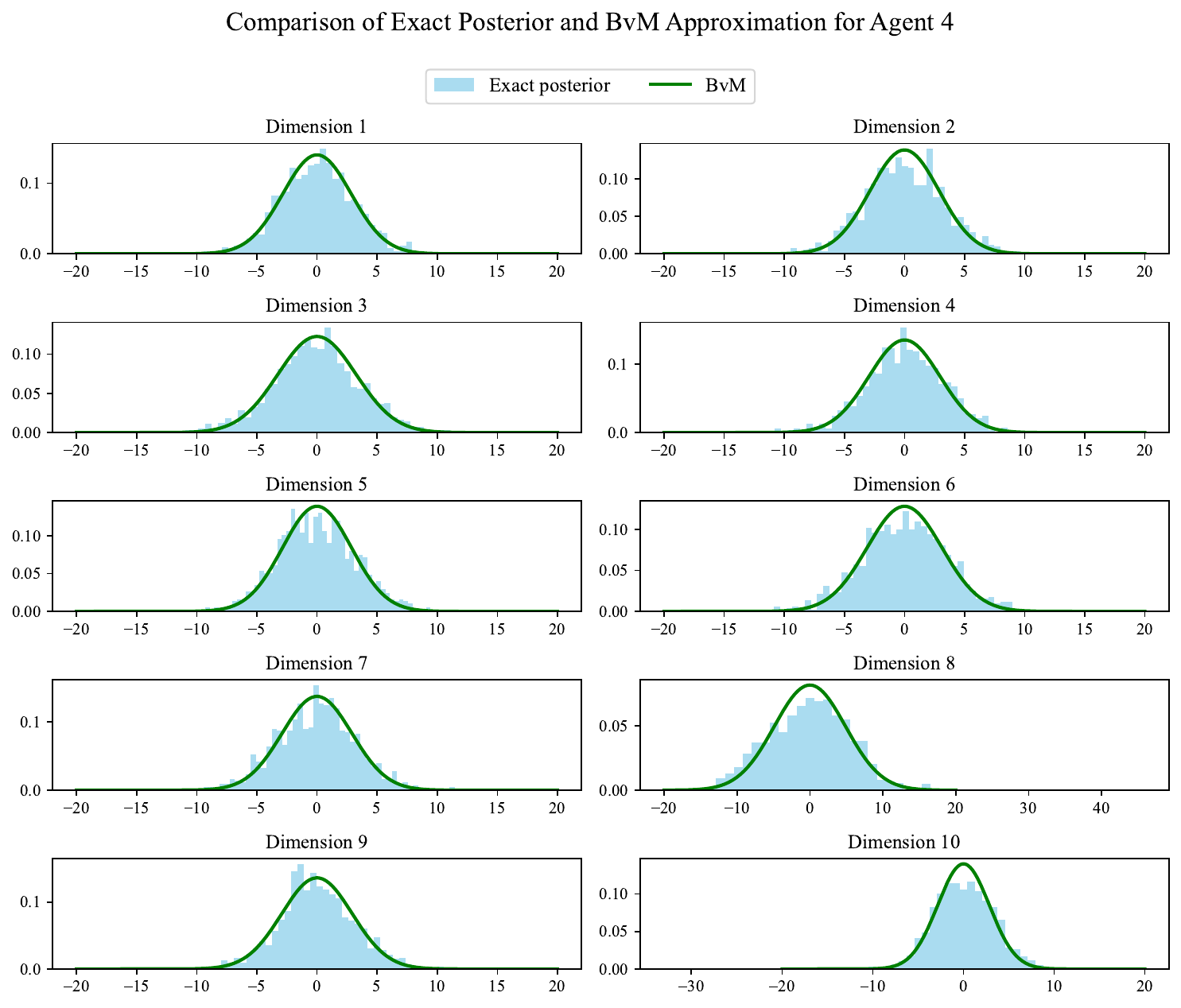}
    \caption{Histograms of the ten marginals computed via Langevin Monte Carlo vs. BvM approximation for a 10-dimensional Bayesian logistic regression example.}
    \label{fig:BvM-hist}
\end{minipage}%
\hfill\begin{minipage}{0.5\linewidth}
    \centering
    \includegraphics[width=\linewidth]{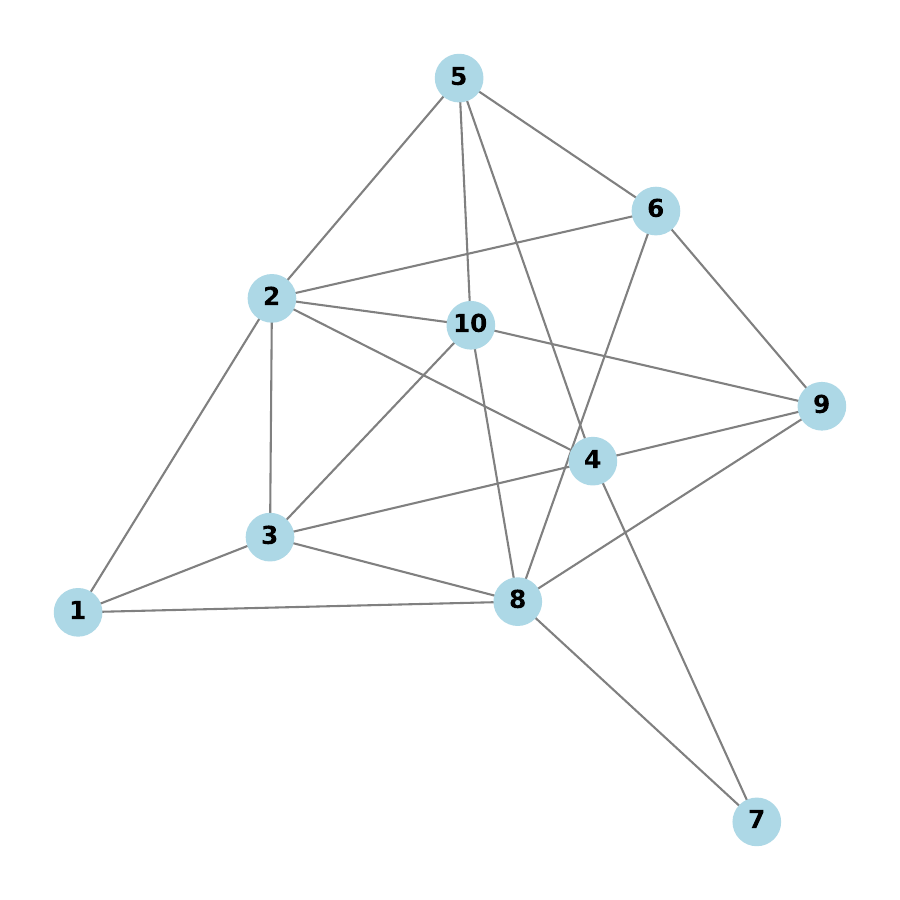}
    \caption{Communication graph $A$.}
    \label{fig:graph}
\end{minipage}
\end{figure}

\FloatBarrier
\section{Proofs}
\label{app:proofs}
\subsection{Proofs of Results in Section~\ref{sect: consistency}} \label{app:consistency}

\begin{proof}[Lemma ~\ref{lemma:ptwise_conv}]
Define vector - valued functions $\phi_t(\theta) = \left[f_t^1(\theta), \cdots, f_t^m(\theta) \right]^T$ and $\log \bp_\theta(X_t) = \left[\log \bP_\theta^1(X_t^1), \cdots, \log \bP_\theta^m (X_t^m)\right]^T$.   \\
Since $\bP_0 |\log \bP_\theta^i| = \bP_0 (\log \bP_\theta^i)_{+} + \bP_0 (\log \bP_\theta^i)_{-} < \infty$, we have
\begin{align*}
f_t^j(\theta) &= -\frac{1}{t} \sum_{k = 1}^t \sum_{i = 1}^m [A^{t-k}_{ji}] \log \bp_\theta^i(X_t^i) = \frac{1}{t} \sum_{k = 1}^t \left< [A^{t-k}_{j,.}] ,\log \bP_\theta\right>. \\
\mbox{ and } \phi_t(\theta) &= \frac{1}{t} \sum_{k = 1}^t A^{t-k} \log \bp_\theta(X_k) 
\end{align*}
By Lemma \ref{lemma:distributed-LLN} [Distributed law of large numbers], 
\begin{equation*}
    \phi_t(\theta) \pto \frac{1}{m}  \sum_{i = 1}^m \bP_0 \log \bP_\theta^i =  f(\theta) \ones. 
\end{equation*}
which implies the coordinate-wise convergence $f_t^j(\theta) \pto f(\theta)$ for every $j \in [m]$. 
\end{proof}

\begin{proof}[Theorem ~\ref{thm:const}]
Let $\epsilon > 0$ be given. If we define $\mu_t^j(B) = \int_B \exp(-t f_t^j(\theta)) \Pi(d\theta) $, then $P_t^j(B) = \frac{\mu_t^j(B)}{\mu_t^j(\Theta)}$ and $\mu_t^j(\Theta)  = z_t^j < \infty$. 

In the proof, we consider the $\Theta = \Theta \backslash U_{null}$ where $\pi(\theta) = 0$ for every $\theta \in U_{null}$. This null set should be of little importance because the set has zero weight under $P_t^j$. 

We now show that there exists $\alpha > 0$ such that $\inf_{\theta \in \U_\epsilon^c} f_t^j (\theta) \geq  f(\theta_0) + \alpha$ for all $t$ large enough.  By lemma \ref{lemma:ptwise_conv}, we get that there exists $T$ such that $f_t^j (\theta) - f (\theta) < \frac{\epsilon}{2}$ for $t > T$. Given that $\inf_{\theta \in \U_\epsilon^c} f_t^j (\theta) -f(\theta_0) = \inf_{\theta \in \U_\epsilon^c} [f_t^j (\theta) - f (\theta)] + [f(\theta) -f(\theta_0)]$, $\inf_{\theta \in \U_\epsilon^c} f_t^j (\theta) -f(\theta_0) > \frac{\epsilon}{2}$ for $t > T$ because for $\theta \in \U_\epsilon^c, f(\theta) -f(\theta_0) > \epsilon$. 

The result shows that for $t > T$, for all $\theta \in \U_\epsilon^c$, $exp(-t(f_t^j(\theta) - f(\theta_0) - \alpha)) \leq 1$. Therefore, for $t > T$, 
\begin{equation*}
    \exp(t(f(\theta_0) + \alpha)) \mu_t^j (\U_\epsilon^c) = \int_{\U_\epsilon^c} \exp(-t(f_t^j(\theta) - f(\theta_0) - \alpha)) \Pi(d\theta) \leq \int_{\U_\epsilon^c} \Pi(d\theta) \leq 1
\end{equation*}
For any $\theta \in A_{\alpha/2}$, $\lim_{t \to \infty} f_t^j(\theta) - f(\theta_0) - \alpha = f(\theta) - f(\theta_0) - \alpha < -\alpha /2 < 0$. Thus, $exp(-t(f_t^j(\theta) - f(\theta_0) - \alpha)) \to \infty$ as $t \to \infty$. By Fatou's lemma, since $\Pi(A_{\alpha/2}) > 0$ by assumption 2, 
\begin{equation*}
    \liminf_{t} e^{t(f(\theta_0) + \alpha)} \mu_t^j(A_{\alpha/2}) =  \liminf_{t} \int_{A_{\alpha/2}} \exp(-t(f_t^j(\theta) - f(\theta_0) - \alpha)) \Pi(d\theta) = \infty. 
\end{equation*}
Since $\mu_t^j(\Theta) \geq \mu_t^j(A_{\alpha/2})$, we obtain $e^{t(f(\theta_0) + \epsilon)} \mu_t^j(\Theta) \to \infty$. 

Combining the two results, 
\begin{equation*}
 1 - P_t^j(\U_\epsilon) = P_t^j(\U_\epsilon^c) =  \frac{\mu_t^j(\U_\epsilon^c)}{\mu_t^j(\Theta)} = \frac{\exp(t(f(\theta_0) + \alpha))\mu_t^j(\U_\epsilon^c)}{\exp(t(f(\theta_0) + \alpha))\mu_t^j(\Theta)} \to 0, 
\end{equation*}
as $t \to \infty$. 
\end{proof}
\begin{proof}[Lemma~\ref{lemma:const_condition}]
Let $j$ and $\theta$ be fixed. If $\bP_0 \ll \bP_\theta^j$, then $D_{KL}(\bP_0 \parallel \bP_\theta^j) < \infty$. This implies that $\bP_0 \log \bp_\theta^j > - \infty$. Likewise, the expectation of $\log \bP_0$ under $\bP_0$ is also finite, i.e., $\bP_0 \log \bP_0 < \infty$.

Since the KL divergence is non-negative, $D_{KL}(\bP_0 \parallel \bP_\theta^j) \geq 0$, it follows that the expectation of $\log \bp_\theta^j$ under $\bP_0$ is less than or equal to the expectation of $\log \bP_0$ under $\bP_0$. Therefore, $\bP_0 \log \bp_\theta^j \leq \bP_0 \log \bp_0 < \infty$, and this completes the proof. 
\end{proof}
\subsection{Proofs of Results in Section~\ref{sect: bvm}} \label{app: bvm}
\begin{proof}[Lemma ~\ref{lemma:M-est-2}]
Suppose $\U$ is a neighborhood of $\theta_0$, within which $\theta \mapsto \log \bP_\theta^i$ is convex for all $i \in [m]$. Being a linear combination of convex mappings, $f_t^j(\theta)$ retains convexity in $\U$. As a convex function, $f_t^j$ is Lebesgue - almost surely differentiable on $\U$. According to Lemma \ref{lemma:ptwise_conv}, $f_t^j \pto f$, and since $f$ is also convex, it remains differentiable a.s. on $\U$.  Given that $\nabla f_t^j(\theta)$ is non-decreasing a.s. on $\U$ for every coordinate, $\nabla f_t^j(\theta_0 - \epsilon) < - \eta$ and $\hat \theta_t^j \leq \theta_0 - \epsilon$ implies $\nabla f_t^j(\hat \theta_t^j) < - \eta$, which has probability tending to $0$ for every $\eta > 0$ if $f_t^j(\hat \theta_t^j) = o_d(1)$. This shows that for every $\epsilon, \eta > 0$, 
\begin{equation*}
    P( \|\hat \theta_t^j - \theta_0 \| > \epsilon) + o(1) \geq  P( \nabla f_t^j(\theta_0 - \epsilon) < -\eta,\nabla f_t^j(\theta_0 + \epsilon) > \eta )
\end{equation*}
Since $\nabla f(\theta_0 - \epsilon) < 0 < \nabla f(\theta_0 + \epsilon)$, taking $2 \eta$ to be the smallest coordinates of $\nabla f(\theta_0 - \epsilon)$ and $\nabla f(\theta_0 + \epsilon)$ makes the right hand side converge to $1$. 

\end{proof}
\begin{proof}[Lemma ~\ref{lemma:z-est-uct}]
Using the consistency of $\hat \theta_t^j$ along with Lemma \ref{lemma:sc-LAN}, given the scaling factor $\epsilon_t^j = \frac{1}{\sqrt{t}} $, we have
\begin{equation*}
f_t^j(\hat\theta_t^j) = f_t^j(\theta_0) + \frac{1}{\sqrt{t}} (\hat \theta_t^j - \theta_0)^T V_{\theta_0} \Delta_{t, \theta}^j + \frac{1}{2} (\hat \theta_t^j - \theta_0)^T V_{\theta_0} (\hat \theta_t^j - \theta_0) + o_d(1) = f_t^j(\theta_0) + o_d(1),
\end{equation*}
by applications of Slusky's theorem.

By Lemma \ref{lemma:ptwise_conv}, $f_t^j(\theta_0)$ converges to $f(\theta_0)$ in probability so $f_t^j(\hat\theta_t^j) = f(\theta_0) + o_d(1)$. 

Assumption \ref{assumption: uct}(a) assures the existence of $\delta > 0$ such that:
\begin{align*}
    \inf_{\theta \in B_\epsilon(\hat\theta_t^j)^c} (f_t^j(\theta) - f_t^j(\hat\theta_t^j)) &= \inf_{\theta \in B_{2\epsilon}(\theta_0)^c} (f_t^j(\theta) - f_t^j(\hat\theta_t^j)) \\
    &=  \inf_{\theta \in B_{2\epsilon}(\theta_0)^c} (f_t^j(\theta) - f(\theta_0)) + o_d(1) \\
    & > 2\delta + o_d(1) \overset{\text{$t$ large}}{>} \delta. 
\end{align*}
hence $\inf_{\theta \in B_\epsilon(\hat\theta_t^j)^c} (f_t^j(\theta) - f_t^j(\hat\theta_t^j)) > \delta$ with probability tending to $1$ under $[\bP_0]$.
\end{proof}

\begin{proof} [Lemma ~\ref{lemma:bvm1}]
For any $\epsilon >0$ and $\delta > 0$, 
\begin{equation*}
    Q_t(B_\delta(0)) = \Pi_t (B_{\delta/a_t}(\theta_t)) \leq \Pi_t(B_\epsilon(\theta_0)), 
\end{equation*}
for all $t$ sufficiently large. Therefore, since $Q_t \to Q$ in total variational distance, we obtain 
\begin{equation*}
      Q(B_\delta(0)) = \lim_{t \to \infty}   Q_t(B_\delta(0)) \leq \liminf_t \Pi_t(B_\epsilon(\theta_0)). 
\end{equation*}
Take $\delta$ arbitrarily large shows that $\lim_{t \to \infty} \Pi_t(B_\epsilon(\theta_0)) = 1$. 
\end{proof}
\begin{proof}[Theorem ~\ref{thm:bvm1}]
For each $j \in [m]$,  the sequence of estimators $\hat \theta_t^j$ are are chosen to satisfy the equation 
\begin{equation*}
    \sum_{k = 1}^t \sum_{i = 1}^m [A^{t - k}_{ij}] \nabla \log p_{\hat \theta_t^j}^i (X_k^i) = 0. 
\end{equation*}
The function $f_t^j(\theta)$ admits the following quadratic expansion.
\begin{align*}
&f_t^j(\theta) =  -\frac{1}{t} \sum_{k = 1}^t \sum_{i = 1}^m [A^{t - k}_{ij}]\log \bp_\theta^i(X_k^i)  \\
&=  -\frac{1}{t} \sum_{k = 1}^t \sum_{i = 1}^m [A^{t - k}_{ij}] \{\log p_{\hat \theta_t^j}^i(X_k^i)  + \nabla \log p_{\hat \theta_t^j}^i(X_k^i)(\theta -\hat \theta_t^j) \\
&- \frac{1}{2}(\theta -\hat \theta_t^j)  \nabla \log p_{\hat \theta_t^j}^i(X_k^i) \left(\nabla \log p_{\hat \theta_t^j}^i(X_k^i)\right)^T (\theta -\hat \theta_t^j)^T  + O_p(|\theta -\hat \theta_t^j|^3)\}\\
&= f_t^j(\hat\theta_t^j) + \frac{1}{2}(\theta -\hat \theta_t^j) \left\{\underbrace{\frac{1}{t} \sum_{k = 1}^t \sum_{i = 1}^m [A^{t - k}_{ij}]     \nabla \log p_{\hat \theta_t^j}^i(X_k^i)   \left( \nabla\log p_{\hat \theta_t^j}^i(X_k^i)\right)^T}_{\hat V_t^j} \right\}(\theta -\hat \theta_t^j)^T - \underbrace{O_p(|\theta -\hat \theta_t^j|^3)}_{r_t^j(\theta -\hat \theta_t^j)}. 
\end{align*}

The sequence of local estimates $\hat \theta_t^j \in \Theta$ satisfies the consistency requirement.  By denoting the remainder term as $r_t^j(\theta -\hat \theta_t^j)$, it immediately holds that
\begin{equation*}
    r_t^j(h) = O_p(|h|^3). 
\end{equation*}
By Theorem 7.2 of \cite{VanderVaart2000},  Assumption \ref{assumption: regularity}(c) implies the existence of the nonsingular Fisher information matrix $V_{\theta_0}^i$ for any $i \in [m]$.  The average of Fisher information matrix, defined as $V_{\theta_0} = \frac{1}{m} \sum_{i = 1}^m V_{\theta_0}^i$, is also nonsingular and positive definite. It remains to show that $\hat V_t^j \pto V_{\theta_0}$.

 We represent the sequence of matrices $\hat V_t^j$ as:
\begin{align*}
    \hat V_t^j &= \frac{1}{t} \sum_{k = 1}^t \sum_{i = 1}^m [A^{t - k}_{ij}]     \nabla \log p_{\hat \theta_t^j}^i(X_k^i)   \left(\nabla \log p_{\hat \theta_t^j}^i(X_k^i)\right)^T .
\end{align*}
Using Lipschitz continuity of the gradient (Assumption \ref{assumption: regularity}(d)), we can decompose $\hat V_t^j$ as:
\begin{align*}
    \hat V_t^j   &=  \frac{1}{t} \sum_{k = 1}^t \sum_{i = 1}^m [A^{t - k}_{ij}]     \left( \nabla \log \bp_{\theta_0}^i(X_k^i)  +s^i(X_k^i)(\hat \theta_t^j - \theta_0) \right) \left( \nabla \log \bp_{\theta_0}^i(X_k^i)  +s^i(X_k^i)(\hat \theta_t^j - \theta_0) \right)^T . 
\end{align*}
Applying the distributive property, we can decompose this into three terms:
\begin{align*}
    \hat V_t^j &=\frac{1}{t} \sum_{k = 1}^t \sum_{i = 1}^m [A^{t - k}_{ij}]  \nabla \log \bp_{\theta_0}^i(X_k^i) \left(\nabla \log \bp_{\theta_0}^i(X_k^i) \right)^T \\
    &+ (\hat \theta_t^j - \theta_0) \frac{1}{t} \sum_{k = 1}^t \sum_{i = 1}^m [A^{t - k}_{ij}]  s^i(X_k^i) \nabla \log \bp_{\theta_0}^i(X_k^i)^T  \\
    &+ (\hat \theta_t^j - \theta_0)  (\hat \theta_t^j - \theta_0)^T \frac{1}{t} \sum_{k = 1}^t \sum_{i = 1}^m [A^{t - k}_{ij}]  [s^i(X_k^i)]^2. 
\end{align*}
By Lemma \ref{lemma:distributed-LLN}, we have
\begin{equation*}
    \frac{1}{t} \sum_{k = 1}^t \sum_{i = 1}^m [A^{t - k}_{ij}]  \nabla \log \bp_{\theta_0}^i(X_k^i) \left(\nabla \log \bp_{\theta_0}^i(X_k^i) \right)^T  \pto V_{\theta_0}, 
\end{equation*}
and 
\begin{equation*}
    \frac{1}{t} \sum_{k = 1}^t \sum_{i = 1}^m [A^{t - k}_{ij}]  s^i(X_k^i) \nabla \log \bp_{\theta_0}^i(X_k^i)^T  \pto \bP_0[s^i \nabla \log \bp_{\theta_0}^i], 
\end{equation*}
and 
By the Cauchy-Schwarz inequality, we have:
\begin{equation*}
\bP_0[s^i \nabla \log \bp_{\theta_0}^i] \leq \sqrt{\bP_0 [s^i]^2} \sqrt{V_{\theta_0}^i} < \infty,
\end{equation*}
thus the second term is $o(1)$. 
\begin{equation*}
    \frac{1}{t} \sum_{k = 1}^t \sum_{i = 1}^m [A^{t - k}_{ij}]  [s^i(X_k^i)]^2 \pto \bP_0 [s^i(X_k^i)]^2. 
\end{equation*}
Finally, since $\hat \theta_t^j - \theta_0 = o_p(1)$, the second and the third term are $o_p(1)$. Combining all these results, we have
\begin{equation*}
\hat V_t^j \pto V_{\theta_0}.
\end{equation*}

Note that $q_t^j(x) = p_t^j(\hat\theta_t^j + x/\sqrt{t}) t^{-p/2}$. Define
\begin{align*}
    g_t^j(x) &= \exp(-t[f_t^j(\hat\theta_t^j+ x/\sqrt{t}) - f_t^j(\hat\theta_t^j)]) \pi(\hat\theta_t^j + x/\sqrt{t}) \\
    &= \exp(t f_t^j(\hat\theta_t^j)) q_t^j (x) z_t^j t^{p/2}. 
    \end{align*}
recalling that $z_t^j < \infty$ by assumption and define
\begin{equation*}
    g_0(x) = \exp(-\frac{1}{2} x^T V_{\theta_0} x) \pi(\theta_0)
\end{equation*}
Let $\alpha \in (0, \lambda)$, where $\lambda$ is the smallest eigenvalue of $V_{\theta_0}$. Let $\epsilon >0$ small enough that $\epsilon < \alpha/ (2 c_0), \epsilon < \epsilon_0$ and $\pi(\theta) \leq 2 \pi(\theta_0)$ for all $\theta \in B_{2 \epsilon}(\theta_0)$. Let $\alpha = \liminf_t \inf_{\theta \in B_\epsilon(\hat\theta_t^j)^c} (f_t^j(\theta) -f_t^j(\hat\theta_t^j))$, noting that $\delta > 0$ with asymptotic probability $1$ by lemma \ref{lemma:z-est-uct}. Let $A_t^j = V_t^j - \alpha I$ and $A_0 = V_{\theta_0} - \alpha I$, define
\begin{align*}
    h_t^j(x) &= \begin{cases} \exp(-\frac{1}{2} x^T A_t^j x) 2 \pi(\theta_0) & \mbox{if }|x| < \epsilon \sqrt{t} \\
\exp(-\frac{1}{2} t \delta) 2\pi(\hat\theta_t^j + x/\sqrt{t}) & \mbox{if }|x| \geq \epsilon \sqrt{t} \end{cases} \\
h_0(x) &= \exp(-\frac{1}{2}x^T A_0 x) 2\pi(\theta_0)
\end{align*}
We show that 
\begin{enumerate}
    \item $g_t^j \to g_0$ and $h_t^j \to h_0$ almost surely
    \item $\int h_t^j$ converges to $\int h_0$
    \item $g_t^j = |g_t^j| \leq h_t^j$ with asymptotic probability $1$, and 
    \item $g_t^j, g_0, h_t^j, h_0 \in L^1(dx)$ with asymptotic probability $1$. 
\end{enumerate}
By the generalized dominated convergence theorem, these results imply that $\int g_t^j \pto \int g_0$ and $\int |(g_t^j){1/m} - g_0| dx \pto 0$ with $[\bP_0]-$probability. Assuming these results are true, we want to show how asymptotic normality follows. Since $\int q_t^j = 1$, the definition of $g_t^j$ implies that 
\begin{equation*}
    \exp(t f_t^j(\hat\theta_t^j)) t^{p/2} z_t^j = \int g_t^j \to \int g_0 = \pi(\theta_0) \frac{(2\pi)^{p/2}}{|m V_{\theta_0}|^{1/2}}. 
\end{equation*}
where $|V_{\theta_0}| = |det(V_{\theta_0})|$ and therefore
\begin{equation*}
    z_t^j \approx \frac{exp(-t f_t^j(\hat\theta_t^j)) \pi(\theta_0)}{|m V_{\theta_0}|^{1/2}} (\frac{2\pi}{t})^{p/2}. 
\end{equation*}
as $t \to \infty$. 
For any $a_t^j \to a \in \R$, we have that $\int |a_t^j  g_t^j - ag_0| \to 0$ since 
\begin{equation*}
    \int |a_t^j  g_t^j - a g_0| \leq \int |a_t^j  g_t^j - a_t^j  g_0| + \int |a_t^j  g_0 - a g_0| = |a_t^j | \int |g_t^j - g_0| + |a_t^j  - a| \int |g_0| \to 0
\end{equation*}
Therefore, if we let $a_t^j  = (exp(tf_t^j(\hat\theta_t^j)) t^{p/2} z_t^j)^{-1}$ and $a = (\pi(\theta_0) \frac{(2\pi)^{p/2}}{|V_{\theta_0}|^{1/2}})^{-1}$, we have shown that $a_t^j \to a$ and thus 
\begin{equation*}
    \int |q_t^j(x) - \frac{|V_{\theta_0}|^{1/2}}{(2\pi)^{p/2}} \exp(-\frac{1}{2}x^T V_{\theta_0} x)| dx \pto 0. 
\end{equation*}
This shows asymptotic normality. Finally, the posterior consistency at $\theta_0$ follows from Lemma \ref{lemma:bvm1}. It remains to show statements 1 to 4 above. 

1) Fix $x \in \R^p$. For $t$ sufficiently large, $|x| \leq \epsilon \sqrt{t}$. It follows that 
\begin{equation*}
    h_t^j(x) = \exp(-\frac{1}{2}x^T A_t^j x) 2\pi(\theta) \to \exp(-\frac{1}{2}x^T V_{\theta_0} x) 2 \pi(\theta_0) = h_0(x). 
\end{equation*}
since $A_t^j \to A_0$ almost surely. For $g_t^j$, first note that $\pi(\hat\theta_t^j+ x/\sqrt{t}) \to \pi(\theta_0)$ since $\pi$ is continuous at $\theta_0$ and $\theta_t^j\to \theta_0$ and $x/\sqrt{t} \to 0$. By the uniform convergence of $f_t^j$ to $f_t$, 
\begin{align*}
    t(f_t^j(\hat\theta_t^j+ x/\sqrt{t}) - f_t^j(\hat\theta_t^j)) 
    &= -\frac{1}{2} x^T V_t^j x +  t r_t^j(x/\sqrt{t}) \to -\frac{1}{2} x^T V_{\theta_0} x. 
\end{align*}
because $V_t^j \to V_{\theta_0}$ and $|x/\sqrt{t}| < \epsilon_0$ for all $t$ sufficiently large. We then infer a bound on $r_t^j$, 
\begin{equation*}
    |t r_t^j(x/\sqrt{t})| \leq t c_0 |x/\sqrt{t}|^3 = c_0 |x|^3 / \sqrt{t} \to 0. 
\end{equation*}
as $t \to \infty$. Therefore, $g_t^j(x) \to g_0(x)$. 

2) By the definition of $h_t^j$, letting $B_t = B_{\epsilon \sqrt{t}}(0)$, 
\begin{equation*}
    \int h_t^j = \int_{B_t} \exp(-\frac{1}{2} x^T A_t^j x) 2 \pi(\theta_0) dx + \int_{B_t^c}exp(-\frac{1}{2} n \delta) 2\pi(\hat\theta_t^j + x/\sqrt{t}) dx. 
\end{equation*}
Since $A_t^j \to A_0$ almost surely and $A_0$ is positive definite, for all $t$ sufficiently large, $A_t^j$ is also positive definite and the first term equals 
\begin{equation*}
    2\pi(\theta_0) \frac{(2\pi)^{p/2}}{|A_t^j|^{1/2}} P(|(A_t^j)^{-1/2}Z | < \epsilon \sqrt{t}) \asto   2\pi(\theta_0) \frac{(2\pi)^{p/2}}{|A_0|^{1/2}} = \int h_0 . 
\end{equation*}
where $Z \sim N(0,1)$. The second integral converges to $0$, since it is nonnegative and has an upper bound as 
\begin{equation*}
  \int_{\R^p}\exp(-\frac{1}{2} t \delta) 2\pi(\hat\theta_t^j + x/\sqrt{t}) dx = \exp^{-\frac{1}{2} n \delta} n^{p/2} \to 0. 
\end{equation*}
using the fact that $\pi(\hat\theta_t^j+ x/\sqrt{t}) t^{-p/2}$ is the density of $\sqrt{t}(\theta -\hat \theta_t^j)$ where $\theta \sim \pi$. 
 %By lemma \ref{lemma: z-est-uct},  for every $\epsilon > 0$, there exists $\delta > 0$ such that $\inf_{\theta \in B_\epsilon(\hat\theta_t^j)^c} (f_t^j(\theta) - f_t^j(\hat\theta_t^j)) > \delta$  as $t \to \infty$ in $[\bP_0]-$probability. This result will be useful later on. 
 
3) With $\bP_0-$probability going to $1$,  $|\theta_t^j - \theta_0| < \epsilon$ thus the bound on $r_t^j$ applies, $\inf_{\theta \in B_\epsilon(\hat\theta_t^j)^c} (f_t^j(\theta) - f_t^j(\hat\theta_t^j)) > \delta$. Let $x \in \R^p$ be given again. If $|x| > \epsilon \sqrt{t}$, then $f_t^j(\hat\theta_t^j+ x/\sqrt{t}) - f_t^j(\hat\theta_t^j) > \delta/2$, and thus, 
\begin{equation*}
    g_t^j(x) \leq \exp(-\frac{1}{2} t \delta) \pi(\hat\theta_t^j+ x/\sqrt{t}) = h_t^j(x). 
\end{equation*}
In the meantime, if $|x| < \epsilon \sqrt{t}$, then $\pi(\hat\theta_t^j+ x/\sqrt{t}) \leq 2 \pi(\theta_0)$ by our choice of $\epsilon$, and 
\begin{align*}
     t(f_t^j(\hat\theta_t^j + x/\sqrt{t}) - f_t^j(\hat\theta_t^j)) 
    &\geq \frac{1}{2} x^T V_t^j x + t r_t^j(x/\sqrt{t}) \\
    &\geq  \frac{1}{2} x^T V_t^j x - \frac{1}{2}\alpha xT x = \frac{1}{2}x^T A_t^j x. 
\end{align*}
since $|t r_t^j (x/\sqrt{t})| \leq 1/2 \alpha |x|^2$, by the fact that $|x/\sqrt{t}| < \epsilon < \epsilon_0$ and $\epsilon < \alpha/(2c_0)$. Therefore, 
\begin{equation*}
    \lim_{t \to \infty}\bP_0\left(g_t^j (x) \leq h_t^j(x)\right) = 1. 
\end{equation*}
4) Since $V_{\theta_0}$ and $A_0$ are positive definite, $\int g_0$ and $\int h_0$ are finite. Since $\int h_t^j \to \int h_0$, we have $\lim_{t \to \infty} \bP_0(\int g_t^j \leq \int h_t^j < \infty)  = 1$. The measurability of $g_t^j, h_t^j$ with respect to $\bP_0$ follows from the measurability of $f_t^j$ and $\pi$. 
\end{proof}

\begin{proof}[Corollary~\ref{cor:bvm1}]
In the proof of Theorem~\ref{thm:bvm1}, the differential in quadratic means assumption (Assumption \ref{assumption: regularity}(c)) is only used to establish Equation \eqref{eqn-bvm-1}. Assumption \ref{assumption: regularity}(e) states that each $\log \bp_\theta^i$ is twice continuously differentiable with nonsingular Hessian matrices. Since $f_t^j(\theta)$ is a positive linear combination of $\log \bp_\theta^i$, $f_t^j(\theta)$ is also twice continuously differentiable with a nonsingular Hessian matrix.  Therefore, Equation \eqref{eqn-bvm-1} holds when we replace Assumption \ref{assumption: regularity}(c) with Assumption \ref{assumption: regularity}(e). 
\end{proof}

\begin{proof}[Corollary~\ref{cor:bvm2}]
Denote $\langle \cdot, \cdot \rangle$ as the tensor inner product, and $\otimes$ as the tensor product. Let Assumption \ref{assumption: regularity}(f) hold. 

We have a third - order Taylor expansion of $f_t^j(\theta)$ around $\hat \theta_t^j$ with nonsingular quadratic terms:  
\begin{equation*}
    f_t^j(\theta) = f_t^j(\hat \theta_t^j) + \frac{1}{2} (\theta - \theta_t^j)^T \nabla_\theta^2 f_t^j(\hat \theta_t^j) (\theta - \theta_t^j) + \frac{1}{6} \langle \nabla_\theta^3 f_t^j(\tilde \theta), (\theta - \theta_t^j) \otimes (\theta - \theta_t^j) \otimes (\theta - \theta_t^j) \rangle, 
\end{equation*}
for some $\tilde \theta$ between $\theta_0$ and $\hat \theta_t^j$. 

Under Assumption \ref{assumption: regularity}(f), we have $|\nabla^3_\theta \log \bp_\theta^j(x)| \leq \phi^j(x)$ for integrable $\phi$ in a neighborhood $\mathcal N^j$ of $\theta_0$.  Define $\mathcal N = \cap_{j = 1}^m \mathcal N^j$. The consistency assumption implies that $\bP_0(\hat \theta_t^j \in \mathcal N) \to 1$. For all $\theta \in \mathcal N$, we have
\begin{equation*}
\bP_0 |\nabla_\theta^3 f_t^j(\tilde \theta)| \leq \bP_0 \frac{1}{t} \sum_{k = 1}^t \sum_{i = 1}^m [A^{t-k}_{ji}] |\nabla^3 \log \bp_\theta^j(X_k^j)| \leq \bP_0 \frac{1}{t} \sum_{k = 1}^t \sum_{i = 1}^m [A^{t-k}_{ji}] \phi^j(X_k^j) < \infty. 
\end{equation*}
By the uniform large of large number (Theorem 1.3.3, \cite{ghoshramamoorthi}) and Lemma \ref{lemma:distributed-LLN},  we have 
\begin{equation*}
    \sup_{\theta \in \mathcal N} \left\|\nabla_\theta^3 f_t^j(\tilde \theta) - \frac{1}{m}  \sum_{i = 1}^m \bP_0 \nabla^3 \log \bp_\theta^i \right\| \to 0. 
\end{equation*}
In the limit, we have $\frac{1}{m} \sum_{i = 1}^m \bP_0 |\nabla^3 \log \bp_\theta^i| \leq \frac{1}{m} \sum_{i = 1}^m \bP_0 \phi^i$, hence $\nabla_\theta^3 f_t^j(\tilde \theta)$ is asymptotically tight and the cubic term is $o_p(1)$. 

We also have that
\begin{equation*}
    \nabla_\theta^2 f_t^j(\hat \theta_t^j) \leq  \nabla_\theta^2 f_t^j(\theta_0) + \frac{1}{t} \sum_{k = 1}^t \sum_{i = 1}^m [A^{t-k}_{ji}] |\langle \phi^j(X_k^j) ,  \hat \theta_t^j - \theta_0 \rangle|, 
\end{equation*}
hence $\nabla_\theta^2 f_t^j(\hat \theta_t^j) = \nabla_\theta^2 f_t^j(\theta_0) + o_p(1)$.  

Finally, we have $f_t^j(\hat \theta_t^j) = f_t^j(\theta_0) + o_p(1)$ by continuous mapping theorem (Theorem 18.11, \cite{VanderVaart2000}).  

Putting the terms together, we conclude that
\begin{equation*}
        f_t^j(\theta) = f_t^j(\theta_0) + \frac{1}{2} (\theta - \theta_t^j)^T \nabla_\theta^2 f_t^j(\theta_0) (\theta - \theta_t^j) + o_p(1). 
\end{equation*}
Thus, Equation \eqref{eqn-bvm-1} holds we substitute Assumption \ref{assumption: regularity}(c) and \ref{assumption: regularity}(d). 
\end{proof}
\begin{proof}[Lemma ~\ref{lemma:sc-LAN}]
By Theorem 7.2 from \cite{VanderVaart2000}, Differentiability in quadratic means implies that $\bP_0 \nabla \log \bp_{\theta_0}^j = 0$ and the Fisher information matrix $V_{\theta_0}^j = \bP_0 \nabla \log \bp_{\theta_0}^j (\nabla \log \bp_{\theta_0}^j)^T$ exists and is non-singular. 

By the central limit theorem, we have 
\begin{equation*}
    \frac{1}{\sqrt{t}} \sum_{k = 1}^t \nabla \log \bp_{\theta_0}^j(X_k^j) \dto N(0, V_{\theta_0}^j)
\end{equation*}
Let's define $V_{\theta_0}$ as the average Fisher information matrix, i.e., $V_{\theta_0} = \frac{1}{m} \sum_{j = 1}^m V_{\theta_0}^j$.

If we choose $\epsilon_t^j = \frac{1}{\sqrt{t}}$, then by Taylor expansion, 
\begin{align*}
& t f_t^j(\theta_0 + \epsilon_t^j h) \\
&=  - \sum_{k = 1}^t \sum_{i = 1}^m [A^{t - k}_{ij}]\log p_{\theta_0 +  \epsilon_t^j h}^i(X_k^i)  \\
&=  - \sum_{k = 1}^t \sum_{i = 1}^m [A^{t - k}_{ij}] \left\{ \log \bp_{\theta_0}^i (X_k^i) +   h^T \frac{ \nabla \log \bp_{\theta_0}^i(X_k^i)}{\sqrt{t}} - \frac{1}{2}h^T \frac{\nabla \log \bp_{\theta_0}^i(X_k^i) {\nabla \log \bp_{\theta_0}^i(X_k^i)}^T}{t} h   + O_p(|\frac{h}{\sqrt{t}}|^3)\right\}\\
&=  t  f_t^j(\theta_0) -h^T \frac{\sum_{k = 1}^t \sum_{i = 1}^m [A^{t - k}_{ij}]\nabla \log \bp_{\theta_0}^i(X_k^i)}{\sqrt{t}}+ \frac{1}{2} h^T \frac{\sum_{k = 1}^t \sum_{i = 1}^m [A^{t - k}_{ij}]\nabla \log \bp_{\theta_0}^i(X_k^i) {\nabla \log \bp_{\theta_0}^i(X_k^i)}^T}{t} h \\
&- O_p(\frac{|h|^3}{\sqrt{t}}). 
\end{align*}
Define $\Delta_{t, \theta_0}^j$ and $V_t$ as:
\begin{align*}
    &\Delta_{t, \theta_0}^j =\frac{1}{\sqrt{t}} V_{\theta_0}^{-1} \left(\sum_{k = 1}^t \sum_{i = 1}^m [A^{t - k}_{ij}]\nabla \log \bp_{\theta_0}^i(X_k^i)\right), \quad \text{and} \\
   & V_t^j = \frac{1}{t}\left( \sum_{k = 1}^t \sum_{i = 1}^m [A^{t - k}_{ij}]\nabla \log \bp_{\theta_0}^i(X_k^i) {\nabla \log \bp_{\theta_0}^i(X_k^i)}^T \right). 
\end{align*}
Substituting these terms in the Taylor expansion, we have:
\begin{equation*}
   t f_t^j(\theta_0 + \epsilon_t^j h)  =  t  f_t^j(\theta_0) - h^T V_{\theta_0} \Delta_{t, \theta_0}^j  + \frac{1}{2} h^T V_t^j h + o_p(|h^2|). 
\end{equation*}
By Lemma \ref{lemma:distributed-CLT} (distributed central limit theorem), we have
\begin{equation*}
    \Delta_{t, \theta_0}^j \dto N(0, (m V_{\theta_0})^{-1}), 
\end{equation*}
thus $\Delta_{t, \theta_0}^j$ is bounded in $[\bP_0]$-probability by Prokhorov's theorem (Theorem 2.4, \cite{VanderVaart2000}). 

By Lemma \ref{lemma:distributed-LLN} (distributed law of large number), 
\begin{equation*}
    V_t^j \pto V_{\theta_0}. 
\end{equation*}
This establishes the stochastic LAN:
\begin{equation*}
\sup_{h \in K} \left| - t f_t^j(\theta_0 + \epsilon_t^j h) + t f_t^j(\theta_0) -h^T V_{\theta_0} \Delta_{t, \theta_0}^j + \frac{1}{2} h^T V_{\theta_0} h \right| = \sup_{h \in K}o_p(|h|) \to 0. 
\end{equation*}
\end{proof}
\begin{proof}[Theorem~\ref{thm: bvm2}]
Since $P_t^j\left(A \right)$ is uniformly bounded by $1$ in $L_\infty$, the set $\{P_t^j(A) \}_{t \in \N}$ is uniformly integrable. Then assumption \eqref{assumption:bvm2} implies that 
\begin{equation*} 
    \bP_0 P_t^j\left(\theta: d(\theta,  \theta_0) \geq M_t \epsilon_t^j \right) \to 0. 
\end{equation*}
The rest of the proof is split into two parts: in the first part, we prove the assertion conditional on an arbitrary compact set $K \subset \R^p$, and in the second part, we use this to prove the asymptotic normality convergence. Throughout the proof, we denote the posterior for $h_t^j = (\theta - \theta_0)/\epsilon_t^j$ when $\theta \sim P_t^j$ by $Q_t^j$. For all Borel sets $B$, the posterior for $h_t^j$ follows from that for $\theta$ by 
\begin{equation*}
    Q_t^j(h_t^j \in B) = P_t^j\left(\theta - \theta_0 \in \epsilon_t^j B \right). 
\end{equation*} 
The assumption that \eqref{eqn:s-lan} holds at $\theta_0$ means that for all $j \in [m]$, there exists non-singular $V_{\theta_0}^j$ and asymptotically tight sequence $\Delta_{t, \theta_0}^j$ such that for every compact set $K \subset \R^p$, as $t \to \infty$, 
\begin{equation} \label{eqn:s-lan-proof}
   \sup_{h \in K} \left| - t f_t^j(\theta_0 + \epsilon_t^j h) + t f_t^j(\theta_0)  -h^T V^j_{\theta_0} \Delta_{t, \theta_0}^j   + \frac{1}{2} h^T V_{\theta_0}^j h \right| \pto 0. 
\end{equation}

We denote the normal distribution $N(\Delta_{t, \theta}^j, V_{\theta_0}^{-1})$ by $\Phi_t^j$. For a compact subset $K \subset \R^p$ such that $Q_t^j(h_t^j \in K) > 0$, we define a conditional measure $Q_{K,t}^j$ of $Q_t^j$ by $Q_{K,t}^j(B) = Q_t^j(B \cap K)/Q_t^j(K)$. Similarly, we define a conditional measure $\Phi_{K,t}^j$ corresponding to $\Phi_t^j$.

Let $K \subset \R^p$ be a compact subset of $\R^p$. For every open neighborhood $U \subset \Theta$ of $\theta_0$, $\theta_0 + K \epsilon_t^j \subset U$ for large enough $t$. Since $\theta_0$ is an interior point of $\Theta$, for large enough $t$, the random function $\psi_t^j: K \times K \mapsto \R$
\begin{equation*}
\psi_t^j(h, h') = \left\lvert 1 - \frac{\phi_t(h')}{\phi_t(h)} 
\frac{q_t^j(h')}{q_t^j(h)} \exp\left(t f_t^j(\theta_0 + \epsilon_t^j h) - t f_t^j(\theta_0 + \epsilon_t^j h')\right)\right\rvert_+
\end{equation*}
is well-defined,  where $\phi_t: K \rightarrow \R$ is the Lebesgue density of the distribution $N(\Delta_{t, \theta_0}^j, V_{\theta_0}^{-1})$, $q_t^j: K \rightarrow \R$ is the Lebesgue density of the prior for the centered and rescaled parameter $h_t^j$, and $f_t^j: K \rightarrow \R$ is the random functions defined in \eqref{def:f_t^j}. 

By the distributed LAN assumption \eqref{eqn:s-lan-proof}, we have for every $h \in K$, 
\begin{equation*}
t f_t^j(\theta_0 + \epsilon_t^j h) - t f_t^j(\theta_0)= h^T V_{\theta_0} \Delta_{t, \theta_0}^j - \frac{1}{2} h^T V_{\theta_0} h + o_{\bP_0}(1).
\end{equation*}
For any two sequences $h_t, h'_t \in K$, $\lim_{t \to \infty}q_t^j(h_t)/q_t^j(h'_t) =  1$ implies that 
\begin{align*}
&\log \left(\frac{\phi_t(h'_t)}{\phi_t(h_t)}  \frac{q_t^j(h_t)}{q_t^j(h'_t)} \right) + t f_t^j(\theta_0 + \epsilon_t^j h_t) - t f_t^j(\theta_0 + \epsilon_t^j h'_t) \\
&= (h'_t - h_t)^T V_{\theta_0} \Delta_{t, \theta_0}^j - \frac{1}{2} h_t^T V_{\theta_0} h_t + \frac{1}{2} {h'_t}^T V_{\theta_0} h'_t + o_{\bP_0}(1)  \\
& + \frac{1}{2} (h_t' - \Delta_{t, \theta_0}^j )^T V_{\theta_0}(h_t' - \Delta_{t, \theta_0}^j )- \frac{1}{2} (h_t - \Delta_{t, \theta_0}^j )^T V_{\theta_0}(h_t - \Delta_{t, \theta_0}^j )
\\
&= o_{\bP_0}(1),
\end{align*}
as $t \to \infty$. 

Since all functions $\psi_t^j$ depend continuously on $(h, h')$ and $K \times K$ is compact, we conclude that,
\begin{equation} \label{eqn:psi-uniform-convergence}
\sup_{h,h' \in K} \psi_t^j(h, h') \pto 0, \text{ as } t \to \infty,
\end{equation}
where the convergence is in outer probability.

Assume that $K$ contains a neighborhood of $0$ (to guarantee that $\Phi_t^j(K) > 0$) and let $\Xi_t^j$ denote the event that $Q_t^j(K) > 0$. Let $\eta > 0$ be given and define the events:
\begin{equation*}
\Omega_t^j = \left\{ \sup_{h,h' \in K} \psi_t^j(h, h') \leq \eta \right\}.
\end{equation*}
Consider the inequality (recall that the total-variation norm $||\cdot||_{TV}$ is bounded by $2$):
\begin{align}
\bP_0 \| Q_{K,t}^j - \Phi_{K,t}^j \|_{TV} 1_{\Xi_t^j}
&\leq \bP_0 \| Q_{K,t}^j - \Phi_{K,t}^j \|_{TV} 1_{\Omega_t^j \cap \Xi_t^j} + 2\bP_0(\Xi_t^j \setminus \Omega_t^j). \label{eqn:inequality}
\end{align}

As a result of Equation \eqref{eqn:psi-uniform-convergence}, the second term is $o(1)$ as $t \to \infty$. The first term on the r.h.s. is calculated as follows:
\begin{align*}
\frac{1}{2}\bP_0 \| Q_{K,t}^j - \Phi_{K,t}^j \|_{TV} 1_{\Omega_t^j \cap \Xi_t^j} 
&= \bP_0 \int \left(1 - \frac{d\Phi_{K,t}^j}{dQ_{K,t}^j}\right)_+ dQ_{K,t}^j 1_{\Omega_t^j \cap \Xi_t^j} \\
&= \bP_0 \int_K \left(1 - \int_K\frac{ \phi_{K,t}^j(h) q_t^j(h') s_t^j(h')}{\phi_{K,t}^j(h') q_t^j(h)s_t^j(h)} d\Phi_{K,t}^j(h') \right)_+  dQ_{K,t}^j(h) 1_{\Omega_t^j \cap \Xi_t^j},
\end{align*}
where $s_t^j: K \rightarrow \R$ is the function defined as $s_t^j(h) = \exp(t f_t^j(\theta_0 + \epsilon_t^j h))$, and $q_t^j: K \rightarrow \R$ is the Lebesgue density of the prior for the centered and rescaled parameter $h_t^j$.

Note that for all $h, h' \in K$, $\phi_{K,t}^j(h)/\phi_{K,t}^j(h') = \phi_t(h)/\phi_t(h')$, since on $K$, $\phi_{K,t}^j$ differs from $\phi_t$ only by a normalization factor. Using Jensen's inequality (with respect to the $\Phi_{K,t}^j$-expectation) for the convex function $x \mapsto (1 - x)_+$, we derive:

\begin{align*}
&\frac{1}{2} \bP_0 \| Q_{K,t}^j - \Phi_{K,t}^j \|_{TV} 1_{\Omega_t^j \cap \Xi_t^j} \\
&\leq \bP_0 \int \left(1 - \frac{s_t^j(h') q_t^j(h') \phi_t(h')}{s_t^j(h') q_t^j(h') \phi_t(h')}\right)_+ d\Phi_{K,t}^j(h') dQ_{K,t}^j(h') 1_{\Omega_t^j \cap \Xi_t^j} \\
&\leq \bP_0 \int_{\Omega_t^j \cap \Xi_t^j}\sup_{h,h' \in K} \psi_t^j(h, h')  d\Phi_{K,t}^j(h') dQ_{K,t}^j(h') \\
&\leq \eta.
\end{align*}

Combination with \eqref{eqn:inequality} shows that for all compact $K \subset \R^p$ containing a neighborhood of $0$, 
\begin{equation} \label{eqn:tv-l1-convergence}
    \bP_0 \| Q_{K,t}^j - \Phi_{K,t}^j \|_{TV} 1_{\Xi_t^j} \to 0. 
\end{equation}
Now, let $(K_n)$ be a sequence of balls centered at $0$ with radii $M_n \to \infty$. For each $n \geq 1$, the Equation \eqref{eqn:tv-l1-convergence} holds. Hence, the intuition is that if we choose a sequence of balls $(K_t)$ that traverses the sequence $K_n$ slowly enough, the convergence of the expected total variational distance to zero can still be guaranteed. Moreover, the corresponding events $\Xi_t^j = \{Q_t^j(K_t) > 0\}$ satisfy $\bP_0(\Xi_t^j) \to 1$ as a result of assumption \eqref{assumption:bvm2}.

We conclude that there exists a sequence of radii $(M_t)$ such that $M_t \to \infty$ and $\bP_0 \| Q_{K_t,t}^j - \Phi_{K_t,t}^j \|_{TV} \pto 0$ (where the conditional probabilities on the l.h.s. are well-defined on sets of probability growing to one). The total variation distance between a measure and its conditional version on a set $K$ satisfies $\|Q - Q_K\|_{TV} \leq 2Q(K^c)$. Combining this with assumption \eqref{assumption:bvm2} and the Hellinger integral test, we conclude that 
\begin{equation*}
    \bP_0 \| Q_t^j - \Phi_t^j \|_{TV} \to 0, 
\end{equation*}
which implies the main result.
\end{proof}

\subsection*{Proofs of Results in Section~\ref{sect: contraction}}\label{app: contraction}
\begin{lemma}\label{lemma:contraction1}
    For any function $f \geq 0$ and two probability measures $P, Q$, we have 
    \begin{equation*}
        \int f(x) dQ(x) \leq D_{KL}(Q \parallel P) + \log \int \exp(f(x)) dP(x). 
    \end{equation*}
\end{lemma}
\begin{proof}[Lemma~\ref{lemma:contraction1}]
    By the definition of KL divergence, 
    \begin{align*}
        D_{KL}(Q \parallel P) + \log \exp(f(x)) d P(x) &= \int \log\left(\frac{dQ(x) \int \exp(f(y)) dP(y)}{dP(x)} dQ(x) \right) \\
        &= \int \log\left(\frac{dQ(x) \int \exp(f(y)) dP(y)}{\exp(f(x)) dP(x)} dQ(x) \right)  + \int f(x) dQ(x) \\
        &= D(Q \parallel P') + \int f(x) dQ(x) \geq \int f(x) dQ(x)
    \end{align*}
where $P'$ is given by
\begin{equation*}
    dP'(x) = \frac{\exp(f(x)) dP(x)}{\int \exp(f(y) dP(y)}. 
\end{equation*}
\end{proof}
\begin{lemma} \label{lemma:contraction2}
For $P_t^j$ defined in \eqref{def-db-posterior} and $P_t$ defined in \eqref{def-ideal-posterior}, we have the following upper bound on the expected loss under $P_t^j$ and $P_0$. 
 \begin{equation*}
     \bP_0 P_t^j d(\theta, \theta_0) \leq  \inf_{a > 0} \frac{1}{amt} \left[ \bP_0 D_{KL}(P_t^j \parallel P_t) + \log \bP_{\theta_0} P_te^{a mt d(\theta, \theta_0) } + \frac{1}{m}\sum_{j = 1}^m D_{KL}(\bP_0 \parallel \bP_{\theta_0}^j) \right]. 
 \end{equation*}
\end{lemma}
\begin{proof}[Lemma~\ref{lemma:contraction2}]
    By Lemma \ref{lemma:contraction1}, we have 
    \begin{equation*}
        amtP_t^jd(\theta, \theta_0)  \leq D_{KL}(P_t^j \parallel P_t) + \log P_t\exp(amt d(\theta, \theta_0)),  
    \end{equation*}
    for all $a > 0$. 
    Taking expectations on both sides, we have
 \begin{equation*}
     amt\bP_0 P_t d(\theta, \theta_0)  \leq \bP_0 D_{KL}(P_t^j \parallel P_t) + \bP_0 \log P_t\exp(amt d(\theta, \theta_0)),  
 \end{equation*}
 By Lemma \ref{lemma:contraction1} and the tenderization property of KL divergence, we have
 \begin{align*}
    \bP_0 \log P_t\exp(amtd(\theta, \theta_0)) &\leq \left( \frac{1}{m} \sum_{j = 1}^m D_{KL}(\bP_0 \parallel \bP_{\theta_0}^j) \right) +  \log \bP_{\theta_0} P_t\exp(amt d(\theta, \theta_0)). 
 \end{align*}
 
 Putting it all together, we have
 \begin{equation*}
    \bP_0 P_t^jd(\theta, \theta_0) \leq  \inf_{a > 0} \frac{1}{a mt} \left(\bP_0 D_{KL}(P_t^j \parallel P_t) +  \frac{1}{m} \sum_{j = 1}^m D_{KL}(\bP_0 \parallel \bP_{\theta_0}^j) +  \log \bP_{\theta_0} P_t\exp\left(amt d(\theta, \theta_0) \right) \right) . 
 \end{equation*}
\end{proof}
\begin{lemma}\label{lemma:contraction3}
   Let the assumptions of Theorem~\ref{thm:contraction1} hold. We have
   \begin{equation*}
    \bP_{\theta_0} P_t(d(\theta, \theta_0) > C_1 \epsilon^2) \leq  \exp(-C_0 t \epsilon^2) + \exp(-\lambda t \epsilon^2) + 2 \exp(-t \epsilon^2), 
   \end{equation*}
   for all $\epsilon \geq \epsilon_{m,t}^2$, where $\lambda = \rho -1$ in assumption \eqref{assumption:C3}. 
\end{lemma}
\begin{proof}[Lemma~\ref{lemma:contraction3}]
    We define the sets
\begin{equation*}
    U = \{\theta: d(\theta, \theta_0)  > C_1 \epsilon^2\}, \qquad K_t = \{\theta: \frac{1}{m}\sum_{j = 1}^m D_{1 + \lambda}(\bP_{\theta_0}^j \parallel \bP_\theta^j) \leq C_3 \epsilon_{m,t}^2\}. 
\end{equation*}
Let $\tilde \Pi$ be a probability measure defined as $\tilde \Pi(B) = \frac{\Pi(B \cap K_t)}{\Pi(K_t)}$.  Define the event
\begin{equation*}
    A_t = \{X^{(mt)}: \int_\Theta \frac{\bp_\theta}{\bp_{\theta_0}}(X^{(mt)}) d \tilde \Pi(\theta) \leq \exp(-(C_3 + 1) t \epsilon^2)\}, 
\end{equation*}
Let $\Pi_t^j$ and $\phi_t$ be the set and testing function in \eqref{assumption:C1}. We bound $\bP_{\theta_0}  P_t (U)$ by 
\begin{align*}
     \bP_{\theta_0}  P_t (U)& \leq \bP_{\theta_0} \phi_t(X^{(mt)}) + \bP_{\theta_0}(A_t) + \bP_{\theta_0} (1 - \phi_t(X^{(mt)})) P_t(U) I_{(A_t)^c} \\
     &= \bP_{\theta_0} \phi_t(X^{(mt)}) 
     + \bP_{\theta_0}(A_t) 
     + \bP_{\theta_0} (1 - \phi_t(X^{(mt)})) I_{(A_t)^c} \frac{\int_{U} \frac{\bp_\theta}{\bp_{\theta_0}}(X^{(mt)}) d \Pi(\theta)}{\int_{\Theta} \frac{\bp_\theta}{\bp_{\theta_0}}(X^{(mt)}) d \Pi(\theta)}. 
\end{align*}
By \eqref{assumption:C1}, we bound the first term by
\begin{equation*}
    \bP_{\theta_0} \phi_t(X^{(mt)})  \leq \exp(-C_0 t \epsilon^2). 
\end{equation*}
By Jensen's inequality, we have
\begin{align*}
   D_\rho(\bP_{\theta_0} \parallel \bP_\theta) &= \sum_{i = 1}^m \frac{1}{\rho - 1} \log \int [dP_{\theta_0}^i]^{\frac{\rho}{m}} [dP_{\theta_0}^i]^{\frac{1}{m} - \frac{\rho}{m}} d \mu \\
   &\leq \frac{1}{m} \sum_{i = 1}^m \frac{1}{\rho - 1} \log \int [dP_{\theta_0}^i]^\rho [dP_{\theta_0}^i]^{1-\rho} d \mu  = \frac{1}{m} \sum_{i = 1}^m  D_\rho(\bP_{\theta_0}^i \parallel \bP_\theta^i) 
\end{align*}
Using the definitions of event $A_t$ and Markov inequality, we have
\begin{align*}
    \bP_{\theta_0} (A_t) &=  \bP_{\theta_0}\left(\left(\int_\Theta \frac{\bp_\theta}{\bp_{\theta_0}}(X^{(mt)}) d \tilde \Pi(\theta) \right)^{-\lambda } < \exp(\lambda(C_3 + 1) t \epsilon^2 \right)  \\
    &\leq \exp(-\lambda (C_3 + 1) t \epsilon^2)  \bP_{\theta_0}\left(\int_\Theta \frac{\bp_\theta}{\bp_{\theta_0}}(X^{(mt)}) d \tilde \Pi(\theta) \right)^{-\lambda } \\
    &\leq \exp(-\lambda(C_3 + 1) t \epsilon^2) \int_\Theta \left( \int_{\X^{mt}}\frac{(\bp_{\theta_0}(x^{(mt)}))^{\lambda }}{(\bp_\theta(x^{(mt)}))^{\lambda } }  d \bP_{\theta_0}(x^{(mt)})\right) d \tilde \Pi(\theta) \\
    &= \exp(-\lambda(C_3 + 1) t \epsilon^2) \int_\Theta \exp(\lambda  D_{1 + \lambda}(\prod_{k = 1}^t  \bP_{\theta_0} \parallel \prod_{k = 1}^t\bP_\theta)) d \tilde \Pi(\theta) \\
    &= \exp(-\lambda(C_3 + 1) t \epsilon^2) \int_{K_t} \exp(\lambda  t  D_{1 + \lambda}(  \bP_{\theta_0} \parallel \bP_\theta)) d \tilde \Pi(\theta) \\
    &\leq \exp(-\lambda(C_3 + 1) t \epsilon^2) \int_{K_t} \exp(\lambda t \frac{1}{m} \sum_{j = 1}^m D_{1 + \lambda}(\bP_{\theta_0}^j \parallel \bP_\theta^j)) d \tilde \Pi(\theta) \\
    &\leq \exp(- \lambda(C_3 + 1) t \epsilon^2 + \lambda C_3 t \epsilon_{m,t}^2) \\
    &\leq \exp(-\lambda t \epsilon^2). 
\end{align*}
Let's analyze the third term. Conditioned on $(A_t)^c$, we have
\begin{equation*}
\int_{\Theta} \frac{\bp_\theta}{\bp_{\theta_0}}(X^{(mt)}) d \Pi(\theta) \geq \Pi(K_t)  \int_{\Theta} \frac{\bp_\theta}{\bp_{\theta_0}}(X^{(mt)}) d \tilde \Pi(\theta)  \geq \exp(-(C_2 + C_3 + 1) t \epsilon^2), 
\end{equation*}
where the last inequality follows from \ref{assumption:C3}. It follows that
\begin{align*}
  &\bP_{\theta_0} (1 - \phi_t(X^{(mt)})) I_{(A_t)^c} \frac{\int_{U} \frac{\bp_\theta}{\bp_{\theta_0}}(X^{(mt)}) d \Pi(\theta)}{\int_{\Theta} \frac{\bp_\theta}{\bp_{\theta_0}}(X^{(mt)}) d \Pi(\theta)} \\
  &\leq \exp((C_2 + C_3 + 1) t \epsilon^2)\bP_{\theta_0} \int_{U} \frac{\bp_\theta}{\bp_{\theta_0}}(X^{(mt)}) d \Pi(\theta) \\
  &\leq \exp((C_2 + C_3 + 1) t \epsilon^2)\left[\int_{U \cap \Theta_t(\epsilon)} \bP_\theta (1- \phi_t (X^{(mt)})) d \Pi(\theta) + \Pi\left(\Theta_t(\epsilon)^c\right) \right] \\
 & \leq \exp((C_2 + C_3 + 1) t \epsilon^2)(\exp(-C t \epsilon^2) + \exp(-C t \epsilon^2)), 
\end{align*}
where the last inequality follows from \ref{assumption:C1} and \ref{assumption:C2}. Since $C_0 > C_3 + C_2 + 2$, we obtain the upper bound for the third term. 
\begin{equation*}
    \bP_{\theta_0} (1 - \phi_t(X^{(mt)})) I_{(A_t)^c} \frac{\int_{U} \frac{\bp_\theta}{\bp_{\theta_0}}(X^{(mt)}) d \Pi(\theta)}{\int_{\Theta} \frac{\bp_\theta}{\bp_{\theta_0}}(X^{(mt)}) d \Pi(\theta)}  \leq 2 \exp(- t \epsilon^2). 
\end{equation*}
Putting the bounds all together, we have
\begin{equation*}
    \bP_{\theta_0}  P_t (U) \leq \exp(-C_0 t \epsilon^2) + \exp(-\lambda t \epsilon^2) + 2 \exp(-t \epsilon^2)
\end{equation*}
\end{proof}
\begin{lemma}[Sub-exponential Bound]\label{lemma:contraction4}
    Let the random variable $X$ satisfies
    \begin{equation*}
        \bP(X \geq \delta) \leq c_1 \exp(-c_2 \delta), 
    \end{equation*}
    for all $\delta \geq \delta_0 > 0$. Then, for any $0 < a \leq \frac{1}{2} c_2$, we have
    \begin{equation*}
        \E \exp(a X) \leq \exp(a \delta_0) + c_1, 
    \end{equation*}
\end{lemma}
\begin{proof}[Lemma~\ref{lemma:contraction4}]
    Define $Y = \exp(aX)$ for some $0 < a \leq \frac{1}{2}c_2$. For any $C > 0$, 
    \begin{align*}
        \E Y \leq C + \int_C^\infty \bP(Y \geq y) dy = C + \int_C^\infty \bP(X \geq \frac{1}{a} \log y) dy \leq C + c_1 \int_C^\infty y^{-c_2/a} dy. 
    \end{align*}
    Take $C = \exp(a \delta_0)$. SInce $a \leq \frac{1}{2} c_2$, we have 
    \begin{equation*}
        \E Y \leq \exp(a_t) + c_1 \exp(-a\delta_0) \leq \exp(a\delta_0) + c_1. 
    \end{equation*}
\end{proof}
\begin{proof}[Proof of Theorem~\ref{thm:contraction1}]
By Lemma \ref{lemma:contraction3}, for all $\delta \geq \delta_0$, we have
\begin{equation*}
    \bP_{\theta_0} P_t(d(\theta, \theta_0) > \delta) \leq c_1 \exp(-c_2 \delta).  
\end{equation*}
Take $c_1 = 4, c_2 = mt\min(\lambda, 1)/C_1$ as $C_0 > C_1 + C_2 + 2 > 1$ and $\delta_0 = C_1 m \epsilon_{m,t}^2$. By Lemma \ref{lemma:contraction4}, we have
\begin{equation*}
    \bP_{\theta_0} P_t\exp(a mtd(\theta, \theta_0))  \leq \exp(a mt C_1\epsilon_{m,t}^2) + 4,  
\end{equation*}
for all $a \leq \min(\lambda, 1)/(2 C_1)$. 

Take $a = \min(\lambda, 1)/(2 C_1)$. By Lemma \ref{lemma:contraction2}, we have
\begin{align*}
    \bP_0  P_t^j d(\theta, \theta_0)  &\leq \frac{mt \gamma_{j,m,t}^2 + \log(4 + \exp(a C_1 mt \epsilon_{m,t}^2)) +  \frac{1}{m}\sum_{j = 1}^m D_{KL}(\bP_0 \parallel \bP_{\theta_0}^j)}{a mt}  \\
    &\leq \frac{\gamma_{j,m,t}^2}{a} + C_1 \epsilon_{m,t}^2 + \frac{4 \exp(-a C_1 mt \epsilon_{m,t}^2)}{amt} + \frac{ \sum_{j = 1}^m D_{KL}(\bP_0 \parallel \bP_{\theta_0}^j)}{a m^2 t}\\
    &\leq C \left(\epsilon_{m,t}^2 + \gamma_{j,m,t}^2 + \frac{1}{m^2 t} \sum_{j = 1}^m D_{KL}(\bP_0 \parallel \bP_{\theta_0}^j) \right) , 
\end{align*}
for some $C$ that depends only on $C_0, C_1, \lambda$. 
\end{proof}
\begin{proof}[Corollary~\ref{cor:contraction1}]
The first result is a consequence of Markov's inequality. 
\begin{align*}
  &\bP_0 P_t^j \left( d(\theta, \theta_0) > M_t (\epsilon_{m,t}^2 + \gamma_{j,m,t}^2 + \frac{1}{m^2 t} \sum_{j = 1}^m D_{KL}(\bP_0 \parallel \bP_{\theta_0}^j)) \right) \\
  &\leq \frac{\bP_0 P_t^j  d(\theta, \theta_0)}{M_t (\epsilon_{m,t}^2 + \gamma_{j,m,t}^2 +  \frac{1}{m^2 t} \sum_{j = 1}^m D_{KL}(\bP_0 \parallel \bP_{\theta_0}^j))} \leq \frac{C}{M_t}  \to 0
\end{align*}
The second result follows from Jensen's inequality
\begin{equation*}
    \bP_0 d(P_t^j \theta, \theta_0) \leq \bP_0 P_t d(\theta, \theta_0) \leq C(\epsilon_{m,t}^2 + \gamma_{j,m,t}^2 +  \frac{1}{m^2 t} \sum_{j = 1}^m D_{KL}(\bP_0 \parallel \bP_{\theta_0}^j))
\end{equation*}
\end{proof}
\begin{proof}[Lemma~\ref{lemma:contraction11}]
We have
    \begin{align*}
\bP_0 D_{KL}(P_t^j \parallel P_t) &=  \bP_0 P_t^j \left( \sum_{k = 1}^t \sum_{i = 1}^m [A^{t - k}_{ij}] \log \bP_\theta^i (X_k^i) -  \frac{1}{m}\sum_{k = 1}^t \sum_{i = 1}^m  \log \bP_\theta^i (X_k^i) \right) \\
     &=P_t^j \bP_0 \left( \sum_{k = 1}^t \sum_{i = 1}^m [A^{t - k}_{ij}] \log \bP_\theta^i (X_k^i) -  \frac{1}{m}\sum_{k = 1}^t \sum_{i = 1}^m  \log \bP_\theta^i (X_k^i) \right) \\
     &= P_t^j  \left( \sum_{k = 1}^t \sum_{i = 1}^m [A^{t - k}_{ij}] \bP_0 \log \bP_\theta^i -  \frac{1}{m}\sum_{k = 1}^t \sum_{i = 1}^m  \bP_0 \log \bP_\theta^i \right) \\
     &= \sum_{k = 1}^t \sum_{i = 1}^m 
 \left\{  [A^{t - k}_{ij}] -  \frac{1}{m} \right\} P_t^j  \bP_0 \log \bP_\theta^i \\
 &= \sum_{k = 1}^t \sum_{i = 1}^m 
 \left\{  [A^{t - k}_{ij}] -  \frac{1}{m} \right\} \left\{\bP_0\log \bp_0  - P_t^j D_{KL}(\bP_0 \parallel \bP_\theta^i)\right\}.
 \end{align*}

For each $j$, we have
\begin{align*}
\bP_0 D_{KL}(P_t^j \parallel P_t) &\leq  \sum_{k = 1}^t \sum_{i = 1}^m | [A^{t - k}_{ij}] -  \frac{1}{m}| \left|\bP_0\log \bp_0 - \max_{i \in [m]}\inf_{\theta \in \Theta} D_{KL}(\bP_0 \parallel \bP_\theta^i) \right| \\
 & \leq  \sum_{k = 1}^t \sum_{i = 1}^m | [A^{t - k}_{ij}] -  \frac{1}{m}| \left\{|\bP_0\log \bp_0| + \max_{i \in [m]}\inf_{\theta \in \Theta} D_{KL}(\bP_0 \parallel \bP_\theta^i) \right\}. 
\end{align*}
By Lemma \ref{lemma:graph_conv}, we get
\begin{equation*}
\bP_0 D_{KL}(P_t^j \parallel P_t)  \leq \frac{16 m^2 \log m}{\nu} \left(|\bP_0\log \bp_0| + \frac{1}{m}\sum_{i = 1}^m D_{KL}(\bP_0 \parallel \bP_{\theta_0}^i)\right). 
\end{equation*}
By Assumption \ref{assumption: regularity}(a), we have
\begin{equation*}
    \gamma_{j,m,t}^2 \leq \frac{16m \log m}{\nu t} \left(|\bP_0\log \bp_0| +  \max_{i \in [m]}\inf_{\theta \in \Theta} D_{KL}(\bP_0 \parallel \bP_\theta^i) \right). 
\end{equation*}
\end{proof}

\subsection{Proofs of Results in Section~\ref{sect:time-varying-graphs}} \label{app:time-varying-graphs}
\begin{proof} [Proposition ~\ref{prop: graph-time-varying1}]

Define $\mathcal T \subseteq [t]$ as the set of $\tau$ where $G_\tau = G$.  This allows us to break the left-hand side of the inequality into two parts: 
    \begin{align*}
        \sum_{k = 1}^t \sum_{j = 1}^m \left|\left[\prod_{\tau = k}^{t-1}A_\tau \right]_{ij}- \frac{\lambda}{m} \right|  &=  
         \sum_{k = 1}^t \sum_{j = 1}^m \left|\left[\prod_{\tau \in [k, t-1]\cap \mathcal T} A \right]_{ij}- \frac{\lambda}{m} \right| =  \sum_{k = 1}^t \sum_{j = 1}^m \left|\left[ A^{|[k, t-1]\cap \mathcal T|} \right]_{ij}- \frac{\lambda}{m} \right|. 
    \end{align*}
As $t \to \infty$, $|\mathcal T|/t \asto \lambda$ and $|[k, t-1]\cap \mathcal T| \asto \lambda (t - k)$.  Then with probability $1$, we have
    \begin{align*} 
       \limsup_{t \to \infty} \sum_{k = 1}^t \sum_{j = 1}^m \left|\left[\prod_{\tau = k}^{t-1}A_\tau \right]_{ij}- \frac{\lambda}{m} \right|  
         &=  \limsup_{t \to \infty}  \sum_{k = 1}^t \sum_{j = 1}^m \left|\left[ A^{\lambda(t - k)} \right]_{ij}- \frac{\lambda}{m} \right|. 
    \end{align*}
Define $\delta = \max(|\lambda_2(A)|, |\lambda_m(A)|)$. The spectral radius of $A$ is $1$, thus $\delta <1$ by Perron-Frobenius theorem.  Under the assumptions \ref{assumption: graph}, by the standard property of stochastic matrices (see e.g. \cite{rosenthal1995convergence}), the diagonalizable matrix $A$ satisfies 
    \begin{equation} \label{eqn:graph_conv-1-1}
        \|e_i^T A^{\lambda t} - \frac{1}{m} \ones \|_1 \leq m \delta^{\lambda t}. 
    \end{equation}
    for any $i \in [m]$, using the fact that $\frac{1}{m} \ones$ is the stationary distribution of the Markov chain with transition matrix $A$.

   Assume that $\lambda \geq \frac{c}{m}$.  For any $t - k \geq \tilde t = \frac{\log \frac{m}{c} + \log \lambda}{- \lambda \log \delta}$, 
    \begin{equation*}
        m \delta^{\lambda(t-k)} \leq  m \delta^{\lambda \tilde t} \leq \frac{m}{\frac{m}{c} \lambda} \leq \frac{c}{\lambda}. 
    \end{equation*}
    Since $ \|e_i^T A^{\lambda t} - \frac{1}{m} \ones \|_1 \leq 2$ by the double stochasticity of $A$, we use \eqref{eqn:graph_conv-1-1} to break the quantity of interest $\sum_{k = 1}^t \sum_{j = 1}^m |[A^{\lambda(t-k)}_{ij}] - \frac{1}{m}|$ into two parts. For any $i \in [m]$, 
    \begin{align*}
        \sum_{k = 1}^t \sum_{j = 1}^m |[A^{\lambda(t-k)}_{ij}] - \frac{1}{m}|  &=  \sum_{k = 1}^t \| e_i^T A^{\lambda(t-k)}- \frac{1}{m} \ones \|_1  \\
        &=  \sum_{k = 1}^{t - \tilde t} \| e_i^T A^{\lambda(t-k)} - \frac{1}{m} \ones \|_1 + \sum_{k > t - \tilde t}^{t} \| e_i^T A^{\lambda(t-k)} - \frac{1}{m} \ones \|_1 \\
        &\leq \sum_{k = 1}^{t - \tilde t}  m \delta^{\lambda(t-k)} + 2 \tilde t \leq \frac{m \delta^{\lambda \tilde t}}{1 - \delta} + 2 \tilde t \leq \frac{c}{\lambda(1 - \delta)} + \frac{2 \log \frac{m}{c}+2 \log \lambda}{-\lambda \log \delta}. 
    \end{align*}
    Noting that $1 - \delta \leq - \log \delta$ and $\delta = 1 - \frac{\nu}{4 m^2}$ for any doubly stochastic matrix $A$,  we get 
    \begin{align*}
         \sum_{k = 1}^t \sum_{j = 1}^m |\left[ A^{\lambda(t - k)} \right]_{ij} - \frac{1}{m}| & \leq  \frac{c + 2 \log \frac{m}{c} + \log \lambda}{\lambda(1 - \delta)}  \leq  \frac{4m^2 c + 8 m^2 \log \frac{m}{c} + 8 m^2\log \lambda}{\lambda \nu}. 
    \end{align*}
When $\lambda \geq \frac{2}{m}$, the optimal bound is achieved when $c = \frac{1}{2}$, 
\begin{align*}
    \limsup_{t \to \infty}  \sum_{k = 1}^t \sum_{j = 1}^m |\left[ A^{\lambda(t - k)} \right]_{ij} - \frac{1}{m}| &\leq   \frac{8m^2 (1 + \log \frac{m}{2}) + 8 m^2\log \lambda}{\lambda \nu} \leq \frac{16m^2 \log m + 8 m^2\log \lambda}{\lambda \nu}. 
\end{align*}
When $\lambda < \frac{2}{m}$,  the optimal bound is achieved when $c = \lambda $, 
\begin{align*}
     \limsup_{t \to \infty}  \sum_{k = 1}^t \sum_{j = 1}^m |\left[ A^{\lambda(t - k)} \right]_{ij} - \frac{1}{m}| &\leq  \frac{4m^3}{\nu}. 
\end{align*}
Finally, when $\lambda = 0$, 
\begin{equation*}
     \limsup_{t \to \infty} \sum_{k = 1}^t \sum_{j = 1}^m \left|\left[\prod_{\tau = k}^{t-1}A_\tau \right]_{ij}- \frac{\lambda}{m} \right|  
       =  \limsup_{t \to \infty} \sum_{k = 1}^t  \frac{m-1}{m} = \infty. 
\end{equation*}
\end{proof}
\begin{proof}[Corollary~\ref{cor:time-varying1}]
We have
    \begin{align*}
\bP_0 D_{KL}(P_t^j \parallel P_t) &=  \bP_0 P_t^j \left( \sum_{k = 1}^t \sum_{i = 1}^m \left[\prod_{\tau = k}^{t-1}A_\tau \right]_{ij} \log \bP_\theta^i (X_k^i) -  \frac{1}{m}\sum_{k = 1}^t \sum_{i = 1}^m  \log \bP_\theta^i (X_k^i) \right) \\
     &=P_t^j \bP_0 \left( \sum_{k = 1}^t \sum_{i = 1}^m \left[\prod_{\tau = k}^{t-1}A_\tau \right]_{ij} \log \bP_\theta^i (X_k^i) -  \frac{1}{m}\sum_{k = 1}^t \sum_{i = 1}^m  \log \bP_\theta^i (X_k^i) \right) \\
     &= P_t^j  \left( \sum_{k = 1}^t \sum_{i = 1}^m \left[\prod_{\tau = k}^{t-1}A_\tau \right]_{ij} \bP_0 \log \bP_\theta^i -  \frac{1}{m}\sum_{k = 1}^t \sum_{i = 1}^m  \bP_0 \log \bP_\theta^i \right) \\
     &= \sum_{k = 1}^t \sum_{i = 1}^m 
 \left\{  \left[\prod_{\tau = k}^{t-1}A_\tau \right]_{ij} -  \frac{1}{m} \right\} P_t^j  \bP_0 \log \bP_\theta^i \\
 &= \sum_{k = 1}^t \sum_{i = 1}^m 
 \left\{  \left[\prod_{\tau = k}^{t-1}A_\tau \right]_{ij} -  \frac{1}{m} \right\} \left\{\bP_0\log \bp_0  - P_t^j D_{KL}(\bP_0 \parallel \bP_\theta^i)\right\}.
 \end{align*}

For each $j$, we have
\begin{align*}
\bP_0 D_{KL}(P_t^j \parallel P_t) &\leq  \sum_{k = 1}^t \sum_{i = 1}^m | \left[\prod_{\tau = k}^{t-1}A_\tau \right]_{ij} -  \frac{1}{m}| \left|\bP_0\log \bp_0 - \max_{i \in [m]}\inf_{\theta \in \Theta} D_{KL}(\bP_0 \parallel \bP_\theta^i) \right| \\
 & \leq  \sum_{k = 1}^t \sum_{i = 1}^m | \left[\prod_{\tau = k}^{t-1}A_\tau \right]_{ij} -  \frac{1}{m}| \left\{|\bP_0\log \bp_0| + \max_{i \in [m]}\inf_{\theta \in \Theta} D_{KL}(\bP_0 \parallel \bP_\theta^i) \right\}. 
\end{align*}
This implies that 
\begin{equation*}
\bP_0 D_{KL}(P_t^j \parallel P_t)  \leq    \limsup_{t \to \infty} \sum_{k = 1}^t \sum_{j = 1}^m \left|\left[\prod_{\tau = k}^{t-1}A_\tau \right]_{ij}- \frac{1}{m} \right| \left(|\bP_0\log \bp_0| + \frac{1}{m}\sum_{i = 1}^m D_{KL}(\bP_0 \parallel \bP_{\theta_0}^i)\right). 
\end{equation*}

Applying Proposition~\ref{prop: graph-time-varying1} gives us the desired results. 
\end{proof}
\begin{proof}[Corollary~\ref{cor:time-varying2}]
    Since $P_t$ does not depend on the communication graph, Theorem~\ref{thm:contraction1} still holds.  This implies the following upper bound on the contraction rate: 
    \begin{equation} \label{eqn-proofcor-time-varying2:1}
    \bP_0 P_t^j d(\theta, \theta_0) \lesssim \epsilon_{m,t}^2 + \gamma_{j,m,t}^2+ \frac{1}{m^2 t} \sum_{j = 1}^m D_{KL}(\bP_0 \parallel \bP_{\theta_0}^j). 
\end{equation}
The only term depending on the communication graph is $\gamma_{j,m,t}^2$. By Corollary~\ref{cor:time-varying2}, we have the following upper bounds with probability $1$ (with respect to the probability measure of $A_1, A_2, \cdots$). 

If $\lambda \geq \frac{2}{m}$, then
    \begin{equation*}
\gamma_{j,m,t}^2 \leq  \frac{16 m \log m + 8m \log \lambda}{\lambda \nu t} \left(|\bP_0\log \bp_0| +  \max_{i \in [m]}\inf_{\theta \in \Theta} D_{KL}(\bP_0 \parallel \bP_\theta^i) \right). 
    \end{equation*}
If $0<\lambda < \frac{2}{m}$, then
    \begin{equation*}
\gamma_{j,m,t}^2 \leq  \frac{4m^2}{\nu t}\left(|\bP_0\log \bp_0| +  \max_{i \in [m]}\inf_{\theta \in \Theta} D_{KL}(\bP_0 \parallel \bP_\theta^i) \right). 
    \end{equation*}
Combining the upper bounds with Equation~\eqref{eqn-proofcor-time-varying2:1}, we have: 

If $\lambda \geq \frac{2}{m}$, then
\begin{equation*}
    \bP_0 P_t^j d(\theta, \theta_0) \lesssim  \frac{1}{t} + \frac{16 m \log m + 8m \log \lambda}{\lambda \nu t} \left(|\bP_0\log \bp_0| +  \max_{i \in [m]}\inf_{\theta \in \Theta} D_{KL}(\bP_0 \parallel \bP_\theta^i) \right) + \frac{1}{m^2 t} \sum_{j = 1}^m D_{KL}(\bP_0 \parallel \bP_{\theta_0}^j). 
\end{equation*}
If $0<\lambda < \frac{2}{m}$, then
\begin{equation*}
    \bP_0 P_t^j d(\theta, \theta_0) \lesssim  \frac{1}{t} + \frac{4m^2}{\nu t}\left(|\bP_0\log \bp_0| +  \max_{i \in [m]}\inf_{\theta \in \Theta} D_{KL}(\bP_0 \parallel \bP_\theta^i) \right)+ \frac{1}{m^2 t} \sum_{j = 1}^m D_{KL}(\bP_0 \parallel \bP_{\theta_0}^j). 
\end{equation*}
This is the desired result after we remove the constants. 
\end{proof}
\subsection{Proofs of Results in Section~\ref{sect:examples}} \label{app:examples}
\begin{proof}[Lemma~\ref{lemma:EF-1}]
   By Definition \eqref{def-db-posterior}, 
   \begin{equation*}
    \log p_t^j(\theta) = t f_t^j(\theta)+ \pi(\theta) + const. 
   \end{equation*}
   for $f_t^j$ defined in Equation \eqref{def:f_t^j}. 
   
   We can compute $f_t^j(\theta)$ explicitly , 
   \begin{align*}
       f_t^j(\theta) &= -\frac{1}{t} \sum_{k = 1}^t \sum_{i = 1}^m [A^{t-k}_{ji}] \log \bp_\theta^i(X_t^i) \\
       &=  -\frac{1}{t} \sum_{k = 1}^t \sum_{i = 1}^m [A^{t-k}_{ji}] \left[ \langle \theta, T^i(X_k^i) \rangle - \psi^i(\theta)  \right] + const  \\
       &= -\frac{1}{t}\left[ \langle \theta,  \sum_{k = 1}^t \sum_{i = 1}^m [A^{t-k}_{ji}]  T^i(X_k^i) \rangle -  \sum_{k = 1}^t \sum_{i = 1}^m [A^{t-k}_{ji}]  \psi^i(\theta)  \right] + const. 
   \end{align*}
   This implies that
   \begin{align*}
        \log p_t^j(\theta) &= -t f_t^j(\theta)+ \pi(\theta)  + const \\
        &=  \langle \theta,   \sum_{k = 1}^t \sum_{i = 1}^m [A^{t-k}_{ji}]  T^i(X_k^i) \rangle - \sum_{k = 1}^t \sum_{i = 1}^m [A^{t-k}_{ji}]  \psi^i(\theta)   + \langle \theta, u\rangle - \psi^0(\theta) + const\\
        &=  \langle \theta,   \sum_{k = 1}^t \sum_{i = 1}^m [A^{t-k}_{ji}]  T^i(X_k^i) + u\rangle -  \sum_{k = 1}^t \sum_{i = 1}^m [A^{t-k}_{ji}]  \psi^i(\theta) - \psi^0(\theta) + const. 
   \end{align*}
   This shows that $p_t^j$ is a member of the exponential family of the form \eqref{eqn-EF-3} with sufficient statistic $\chi_t^j + u$ and log-partition function $ B_t^j(\theta)  + \psi^0(\theta)$ as defined in Equation \eqref{eqn-EF-4}. 
\end{proof}
\begin{proof}[Lemma~\ref{lemma:EF-2}]
By Lemma \ref{lemma:EF-1}, $P_t^j$ is a member of the exponential family given by Equation \eqref{eqn-EF-3}. 
    The full rankness of $P_t^j$ follows directly from the definition of full-rank exponential family and the compact representation \eqref{eqn-EF-3}. By Lemma 4.5 of \cite{VanderVaart2000}, the log-partition function $B_t^j: \Theta \to R$ is infinitely times differentiable with the gradient being the mean parameter of $\chi_t^j$ and the Hessian is the covariance matrix of $\chi_t^j$. Since $P_t^j$ is full-rank, it is equivalent to the nonsingularity of the covariance matrix for $\chi_t^j$ \cite{VanderVaart2000}. This implies that the gradient $\nabla_\theta B_t^j$ is one-to-one in the interior of $\Theta$, thus, $\nabla_\theta B_t^j$ is invertible. 
\end{proof}
\begin{proof}[Proposition~\ref{prop:EF-1}]
By Lemma \eqref{lemma:EF-2}, the sequence of estimators $\hat \theta_t^j$ is well-defined with $[\bP_0]-$probability tending to $1$. By Theorem 4.6 of \cite{VanderVaart2000}, the sequence of estimators $\hat \theta_t^j$ has limiting distribution
\begin{equation*}
    \sqrt{t} (\hat \theta_t^j - \theta_0) \dto N(0, I_{\theta_0}^{j^{-1}}), 
\end{equation*}
where $I_{\theta_0}^j$ is the Fisher information for $\bP_\theta^j$ at $\theta_0$.  This corresponds to the Sandwich covariance matrix. 
\begin{equation*}
    I_{\theta_0}^j = \left[ \phi^{'^{-1}}_{\theta_0} \text{Cov}_{\theta_0}(\chi_t^j) \phi^{-1}_{\theta_0} \right]^{-1}  = \text{Cov}_{\theta_0}(\chi_t^j). 
\end{equation*}
Since $B_t^j$ is infinitely times differentiable in a neighborhood of $\theta_0$, the Fisher information $I_{\theta}^j$ must be uniformly bounded. Thus, $\nabla \log \bp_\theta^j$ is Lipchistz around $\theta_0$ and Assumption \ref{assumption: regularity}(d) is satisfied. 

When $\text{int}(\Theta) \neq \emptyset$, the exponential family $\bP_\Theta^i$ satisfies the differential in the quadratic mean Assumption \ref{assumption: regularity}(c) by Example 7.7 of \cite{VanderVaart2000}. Given that $\theta_0 \in \text{int}(\Theta)$ and $\pi$ have support on $\Theta$, the prior $\pi$ puts positive mass on every neighborhood of $\theta_0$, thus Assumption \ref{assumption: pmc}(b) is satisfied.

Under the full-rank assumptions on $\bP_\theta^j$,  the Hessian $\nabla^2_\theta \log \bp_\theta^j$ is negative definite for all $\theta$. Thus, the Hessian of $f_t^j$ is negative definite as a positive linear combination of $\nabla^2_\theta \log \bp_\theta^j$. We have that $f_t^j$ is strictly concave on $\Theta$ with a unique maximizer at $\tilde \theta_t^j$. The strict concavity implies that 
\begin{equation*}
\lim_{t \to \infty} \bP_0(\inf_{\|\theta - \hat\theta_t^j\| > \delta}|f_t^j(\theta) - f_t^j(\hat\theta_t^j)| \geq \epsilon) = 1
\end{equation*}
By Lemma \ref{lemma:M-est-2}, $\tilde \theta_t^j$ is consistent at $\theta_0$. Then Assumption \ref{assumption: uct}(a) is satisfied as a consequence of  Lemma \ref{lemma:z-est-uct}.

The assumptions of Theorem~\ref{thm:bvm1} are all satisfied. Considering the sequence of random variables $\theta \sim P_t^j$ centered at the moment estimator $\hat \theta_t^j$, we have 
\begin{equation*}
    \int_\Theta |q_t^j(x) - N(0, V_{\theta_0}^{-1})| dx \pto 0, 
\end{equation*}
where $q_t^j$ is the density of $\sqrt{t}(\theta -\hat \theta_t^j)$, and $V_{\theta_0} = \frac{1}{m} \sum_{i = 1}^m \text{Cov}(T^i)$.  
\end{proof}

\begin{proof}[Proposition~\ref{prop:logistic-1}]
The prior $\pi$ is assumed to have full support over $\R^p$, which trivially satisfies Assumption \ref{assumption: pmc}(a).

Define functions $f_t^j$, $f_t$, and $f$ on $\Theta$ as per Equations \eqref{def:f_t^j}--\eqref{def:f}. By Assumption \eqref{assumption: regularity}(a), the function $f$ exists and we note that $f_t^j(\theta) \pto f(\theta)$.

The estimators $\hat \theta_t^j$ and the true parameter $\theta_0$ are respectively given by:
\begin{equation*}
    \hat \theta_t^j = \argmin_{\theta \in \Theta} f_t^j(\theta), \quad \theta_0 = \argmin_{\theta \in \Theta} f(\theta).
\end{equation*}

The gradient of $f_t^j$ is given by
\begin{equation*}
    %\label{eqn:gradient-ftj}
    \nabla  f_t^j(\theta) = -\frac{1}{t} T_t^j + \frac{1}{t}  \sum_{k = 1}^t \sum_{i = 1}^m \frac{ [A^{t-k}_{ji}] X_k^i e^{\langle \theta,  X_k^i \rangle}}{1 + e^{\langle \theta,  X_k^i \rangle}},
\end{equation*}

and the Hessian is
\begin{equation*}
  %  \label{eqn:hessian-ftj}
    \nabla^2 f_t^j(\theta) = \frac{1}{t}  \sum_{k = 1}^t \sum_{i = 1}^m \frac{ [A^{t-k}_{ji}] X_k^i X_k^{i^T} e^{\langle \theta,  X_k^i \rangle}}{\left(1 + e^{\langle \theta,  X_k^i \rangle} \right)^2}.
\end{equation*}
Given that $\nabla^2 f_t^j(\theta) > 0$, the function $f_t^j$ is strictly convex in $\theta$. This implies that the minimizer $\hat \theta_t^j$ is unique, and we have $\hat \theta_t^j \in \text{int}(\Theta)$. 

The gradient of $f$ is given by
\begin{equation*}
    %\label{eqn:gradient-ftj}
    \nabla  f(\theta) = \bP_{\theta_0}\left[\frac{X_1^1 e^{\langle \theta,  X_1^1 \rangle}}{1 + e^{\langle \theta,  X_1^1 \rangle}} - X_1^1 Y_1^1 \right] =  \bP_{\theta_0}\left[X_1^1 \left( \frac{e^{\langle \theta,  X_1^1 \rangle}}{1 + e^{\langle \theta,  X_1^1 \rangle}} - \frac{e^{\langle \theta_0,  X_1^1 \rangle}}{1 + e^{\langle \theta_0,  X_1^1 \rangle}}\right) \right],
\end{equation*}
The differentiation and expectation are interchanged by the Lebesgue-dominated convergence theorem since $ |\nabla f(\theta)| \leq |X_1^1|$ and $\bP_0 |X_1^1| < \infty$ by Assumption ii).

Since $\nabla f(\theta)$ is strictly increasing in $\theta$, we have $\nabla f(\theta - \epsilon) < 0 < \nabla f(\theta + \epsilon)$ in each coordinate. Hence, Assumption \ref{assumption: regularity}(b) is satisfied. By Lemma \ref{lemma:M-est-2}, we have $\hat \theta_t^j \pto \theta_0 $.

%Since $f$ is also convex, then $\tilde f$ and $f$ are continuous functions (Theorem 2.35, \citep{rockafellar2009variational}) that agree on a dense subset of $\Theta$, so they must be equal.
 Let $\mathcal{E} = \{\eta \in \mathbb{R} : |\sigma(\eta)| < \infty \}$. The set $\mathcal{E}$ is open and nonempty. Additionally, $\eta$ is identifiable, as $\sigma(\eta)$ is a one-to-one function. Trivially, the set $\Theta$ is open and convex, and $\theta^T x \in \mathcal{E}$ for all $\theta \in \Theta$ and $x \in \X$.

 For all $\eta \in \mathcal{E}$, we have $0 < \sigma(\eta) < 1$ and
\begin{equation*}
    |\sigma'''(\eta)| = |\frac{\sigma(1 - \sigma(\eta))(1 - 2\sigma(\eta))^2 - 2\sigma^2(1-\sigma(\eta))^2}{(1 - \sigma(\eta))^2}| \leq 3. 
\end{equation*}

After algebraic manipulations, we have
\begin{equation*}
    \nabla_\theta^3 \log\bp_\theta^j(. \mid x_k^j)_{a,b, c} = \sigma^{'''}(\langle \theta, x_k^j \rangle) x_{k,a}^j x_{k,b}^j x_{k,c}^j \leq 3 x_{k,a}^j x_{k,b}^j x_{k,c}^j, 
\end{equation*}
for all $\theta \in \Theta, x_k^j \in \X$ and $a, b,c \in [p]$.  Thus, $\nabla_\theta^3 \log \bp_\theta^j$ is uniformly bounded by an integrable function. 

The Fisher information for agent $j$ is given by
\begin{equation*}
    V_{\theta_0}^j = - \bP_0 \nabla_\theta^2 \log \bp_\theta^j(. \mid X_k^j) = \bP_0  \left[\frac{X_k^{j^T} e^{\langle \theta,  X_k^j \rangle} X_k^j}{\left(1 + e^{\langle \theta,  X_k^j \rangle} \right)^2} \right].  
\end{equation*}
By Assumption i), $V_{\theta_0}^j$ exists and is nonsingular. Hence, Assumption \ref{assumption: regularity}(f) is satisfied. 

Assumption \ref{assumption: regularity}(c), \ref{assumption: uct}(a) are satisfied with the same argument as Proposition \ref{prop:EF-1}.  By Corollary \ref{cor:bvm2} and Slusky's theorem, we have
\begin{equation*}
    \int_\Theta \left|q_t^j(x) - N(0, \hat V_{\theta_0}^{-1})\right| \, dx \pto 0.
\end{equation*}
where $\hat V_{\theta_0}$ is the finite-sample version of $\frac{1}{m} \sum_{j = 1}^m V_{\theta_0}^j$. Specifically,
\begin{equation*}
   \hat V_\theta = \frac{1}{m t} \sum_{i = 1}^m {X^i}^T  \text{diag}\left(
\frac{e^{\sum_{j=0}^{p} \theta_j x_{1j}^i}}{\left(1 + e^{\sum_{j=0}^{p} \theta_j x_{1j}^i}\right)^2},
\ldots,
\frac{e^{\sum_{j=0}^{p} \theta_j x_{tj}^i}}{\left(1 + e^{\sum_{j=0}^{p} \theta_j x_{tj}^i}\right)^2}
\right) X^i.
\end{equation*}
\end{proof}
\begin{proof}[Lemma~\ref{lemma:detection-1}]
    By Lemma \ref{lemma:ptwise_conv}, $f_t^j(\theta) \pto f(\theta)$ for each $\theta \in \Theta$.  For all $\theta \in [0,1]^2$, $f_t^j$ is uniformly bounded. 
\begin{align*}
|f_t^j(\theta)| &\leq \frac{1}{t} \sum_{k = 1}^t \sum_{i = 1}^m [A^{t-k}_{ji}] \left[\frac{(|X_t^j - \theta|)^2}{2 \sigma^{j^2}} + \log \left(\Phi(\frac{|Z^j|+ \frac{1}{2}}{\sigma^j}) - \Phi(\frac{- \sqrt{2} - |Z^j|}{\sigma^j}) \right) \right] \\
&\leq \frac{1}{t} \sum_{k = 1}^t \sum_{i = 1}^m [A^{t-k}_{ji}] \left[\frac{1}{\sigma^{j^2}}   +  \log \left(\Phi(\frac{|Z^j|+ \frac{1}{2}}{\sigma^j}) - \Phi(\frac{- \sqrt{2} - |Z^j|}{\sigma^j}) \right) \right] \\
&\leq \left[\frac{1}{\sigma^{j^2}}   +  \log \left(\Phi(\frac{|Z^j|+ \frac{1}{2}}{\sigma^j}) -    `{\Phi(\frac{- \sqrt{2} - |Z^j|}{\sigma^j}) } \right) \right] \left(1 + \frac{16 m^2 \log m}{\nu} \right). 
\end{align*}

The gradient $\nabla f_t^j(\theta)$ is uniformly bounded over $[0,1]^2$. To see this, consider
\begin{align*}
\nabla f_t^j(\theta) = A + B, 
\end{align*}
where $A, B$ are given by 
\begin{align*}
    A &=  \frac{1}{t} \sum_{k = 1}^t \sum_{i = 1}^m [A^{t-k}_{ji}] \left[\frac{\theta - Z^j}{\sigma^{j^2}} - \frac{|X_t^j - Z^j| (\theta - Z^j)}{|\theta - Z^j|\sigma^{j^2}}\right], \\
    B &=  \frac{1}{t} \sum_{k = 1}^t \sum_{i = 1}^m [A^{t-k}_{ji}] \frac{(\theta - Z^j)\left[\phi\left(\frac{|Z^j| + \frac{1}{2} - |\theta - Z^j|}{\sigma^j}\right) - \phi\left(\frac{-|\theta - Z^j|}{\sigma^j}\right)\right]}{\sigma^j |\theta - Z^j|\left[\Phi\left(\frac{|Z^j| + \frac{1}{2} - |\theta - Z^j|}{\sigma^j}\right) - \Phi\left(\frac{-|\theta - Z^j|}{\sigma^j}\right)\right]}.
\end{align*}
Applying the following upper bounds on $A$ and $B$,
\begin{align*}
    A &\leq \frac{1}{t} \sum_{k = 1}^t \sum_{i = 1}^m [A^{t-k}_{ji}] \frac{\sqrt{2}}{\sigma^{j^2}}, \\
    B &\leq \frac{1}{t} \sum_{k = 1}^t \sum_{i = 1}^m [A^{t-k}_{ji}] \frac{\phi\left(\frac{|Z^j| + \frac{1}{2}}{\sigma^j}\right) - \phi\left(\frac{-\sqrt{2}}{\sigma^j}\right)}{\sigma^j\left[\Phi\left(\frac{|Z^j| + \frac{1}{2} - \sqrt{2}}{\sigma^j}\right) - \Phi\left(\frac{-\sqrt{2}}{\sigma^j}\right)\right]},
\end{align*}
we obtain
\begin{equation*}
    \|\nabla f_t^j(\theta)\| = \|A + B\| \leq \left[ \frac{\sqrt{2}}{ \sigma^{j^2}} + \frac{\phi(\frac{|Z^j|+ \frac{1}{2}}{\sigma^j}) - \phi(\frac{- \sqrt{2}}{\sigma^j}) }{\sigma^j\left[\Phi(\frac{|Z^j|+ \frac{1}{2} - \sqrt{2}}{\sigma^j}) - \Phi(\frac{- \sqrt{2}}{\sigma^j}) \right]} \right]\left(1 + \frac{16 m^2 \log m}{\nu} \right). 
\end{equation*}
Then $f_t^j$ is uniformly equicontinuous in $t$, as $f_t^j$ is Lipschitz.  By Lemma \ref{uniform-convergence-lemma}, we have $\|f_t^j - f\|_\infty \pto 0$.

Note that $\hat{\theta}_t^j = \argmin_{\theta \in [0,1]^2} f_t^j(\theta)$ and $\theta_0 = \argmin_{\theta \in [0,1]^2} f(\theta)$. By the Argmax theorem (Theorem 3.2.2, \cite{vaart2023empirical}), we have $\hat{\theta}_t^j \pto \theta_0$. 

Since $\theta_0 \in \Theta$, for small enough $\epsilon$, we have
\begin{align*}
    \bP_0(|\hat \theta_t^j - \theta_0| > \epsilon) &= \bP_0(|\hat \theta_t^j  - \theta_0| > \epsilon, \hat \theta_t^j  \notin \Theta) + \bP_0(|\hat \theta_t^j - \theta_0| > \epsilon, \hat \theta_t^j \in \Theta)  \\
    &= \bP_0(\hat \theta_t^j \notin \Theta) + \bP_0(|\hat \theta_t^j - \theta_0| > \epsilon, \hat \theta_t^j \in \Theta) \to 0, 
\end{align*}
thus $\bP_0(\hat \theta_t^j \notin \Theta) \to 0$. 
\end{proof}

\begin{proof}[Proposition~\ref{prop:detection-1}]
The prior $\pi$ is assumed to have full support over $\Theta$, which trivially satisfies Assumption \ref{assumption: pmc}(a).

From the proof of Lemma \ref{lemma:detection-1}, we proved that $\nabla f_t^j$ is uniformly bounded for all $t$, thus $f_t^j$ is uniformly equicontinuous. By definition, there exists $\delta > 0$ such that $\|\theta - \hat \theta_t^j\| > \delta$ implies that $|f_t^j(\theta) - f_t^j(\hat \theta_t^j)| > \epsilon$. This verifies Assumption \ref{assumption: uct}(a). 

The gradient of $\log \bp_\theta^j$ is given by
\begin{align*}
    \nabla \log \bp_\theta^j(X_t^j) = \frac{\theta - Z^j}{|\theta - Z^j| \sigma^{j^2}} \left[ |\theta - Z^j| - |X_t^j - Z^j| - \frac{\phi\left(\frac{|Z^j| + \frac{1}{2} - |\theta - Z^j|}{\sigma^j}\right) - \phi\left(\frac{- |\theta - Z^j|}{\sigma^j}\right)}{\Phi\left(\frac{|Z^j| + \frac{1}{2} - |\theta - Z^j|}{\sigma^j}\right) - \Phi\left(\frac{- |\theta - Z^j|}{\sigma^j}\right)} \right].
\end{align*}
This allows us to derive the Fisher information of $p_\theta^j$. 
\begin{align*}
   V_\theta^j &=  \bP_0[\nabla_\theta \log \bp_\theta^j \nabla_\theta \log \bp_\theta^{j^T}] \\
   &=  \frac{(\theta - Z^j) (\theta - Z^j)^T}{\sigma^{j^4} |\theta - Z^j|^2} \left[ |\theta - Z^j| - |\theta_0 - Z^j| - \frac{\phi(\frac{|Z^j|+ \frac{1}{2} - |\theta - Z^j|}{\sigma^j}) - \phi(\frac{- |\theta - Z^j|}{\sigma^j})}{ \Phi(\frac{|Z^j|+ \frac{1}{2} - |\theta - Z^j|}{\sigma^j}) - \Phi(\frac{- |\theta - Z^j|}{\sigma^j}) } \right]^2. 
\end{align*}

At $\theta_0$, the Hessian simplifies to $V_{\theta_0}^j = \frac{(\theta_0 - Z^j) (\theta_0 - Z^j)^T}{\sigma^{j^4} |\theta_0 - Z^j|^2} \left[ \frac{\phi(\frac{|Z^j|+ \frac{1}{2} - |\theta_0 - Z^j|}{\sigma^j}) - \phi(\frac{- |\theta_0 - Z^j|}{\sigma^j})}{ \Phi(\frac{|Z^j|+ \frac{1}{2} - |\theta_0 - Z^j|}{\sigma^j}) - \Phi(\frac{- |\theta_0 - Z^j|}{\sigma^j}) } \right]^2$ which is nonsingular. Thus, Assumption \ref{assumption: regularity}(e) is satisfied. 

Reusing the argument in the proof of Lemma \ref{lemma:detection-1}, we have
\begin{align*}
   \sup_{\theta \in [0,1]^2} \|\nabla \log p_{\theta}^j\| \leq \frac{\sqrt{2}}{\sigma^{j^2}} + \frac{\phi(\frac{|Z^j|+ \frac{1}{2}}{\sigma^j}) - \phi(\frac{- \sqrt{2}}{\sigma^j}) }{\sigma^j\left[\Phi(\frac{|Z^j|+ \frac{1}{2} - \sqrt{2}}{\sigma^j}) - \Phi(\frac{- \sqrt{2}}{\sigma^j}) \right]}. 
\end{align*}

This implies that $\nabla \log \bp_\theta^j$ is a continuously differentiable function over a compact domain $[0,1]^2$. Thus, it is Lipschitz continuous. Assumption \ref{assumption: regularity}(d) is satisfied. 
 
The Hessian for $f_t^j$ is simply the average of agent Fisher information: 
\begin{equation*}
    \nabla^2 f_t^j(\theta) = \frac{1}{t} \sum_{k = 1}^t \sum_{i = 1}^m [A^{t-k}_{ji}] V_\theta^i. 
\end{equation*}
Recall that in the proof of 
Since $f_t^j$ is uniformly equicontinuous, 

We define the event $E_t$ as 
\begin{equation*}
    E = \{X^{(mt)}:\hat \theta_t^j \in \Theta\}. 
\end{equation*}
Conditioning on $E_t$, by Corollary \ref{cor:bvm2} and Slusky's theorem, we have
\begin{equation*}
    \int_\Theta \left|q_t^j(x) - N(0,  V_{\theta_0}^{-1})\right| \, dx \pto 0, 
\end{equation*}
for  $q_t^j$ defined as the density of $\sqrt{t}(\theta - \hat{\theta}_t^j)$ and $V_{\theta_0}$ defined in Equation \eqref{eqn-detection-4}. 

By Lemma \ref{lemma:detection-1}, $\lim_{t \to \infty}\bP_0(E_t) = 1$. This completes our proof. 
\end{proof}

\section*{Appendix C. Supporting Results}
\label{app:support}
\begin{proof} [Lemma ~\ref{lemma:graph_conv}]
Define $\delta = \max(|\lambda_2(A)|, |\lambda_m(A)|)$. The spectral radius of $A$ is $1$, thus $\delta <1$ by Perron-Frobenius theorem.  Under the assumptions \ref{assumption: graph}, by the standard property of stochastic matrices (see e.g. \cite{rosenthal1995convergence}), the diagonalizable matrix $A$ satisfies 
    \begin{equation} \label{eqn:graph_conv-1}
        \|e_i^T A^t - \frac{1}{m} \ones \|_1 \leq m \delta^t
    \end{equation}
    for any $i \in [m]$, using the fact that $\frac{1}{m} \ones$ is the stationary distribution of the Markov chain with transition matrix $A$. 

 For any $t - k \geq \tilde t = \frac{\log \frac{m}{2}}{- \log \delta}$, 
    \begin{equation*}
        m \delta^{t-k} \leq  m \delta^{\tilde t} \leq \frac{m}{\frac{m}{2}} \leq 2 
    \end{equation*}
    Since $ \|e_i^T A^t - \frac{1}{m} \ones \|_1 \leq 2$ by the double stochasticity of $A$, we use \eqref{eqn:graph_conv-1} to break the quantity of interest $\sum_{k = 1}^t \sum_{j = 1}^m |[A^{t - k}_{ij}] - \frac{1}{m}|$ into two parts, that is, for any $i \in [m]$, 
    \begin{align*}
        \sum_{k = 1}^t \sum_{j = 1}^m |[A^{t - k}_{ij}] - \frac{1}{m}|  &=  \sum_{k = 1}^t \| e_i^T A^{t - k} - \frac{1}{m} \ones \|_1  \\
        &=  \sum_{k = 1}^{t - \tilde t} \| e_i^T A^{t - k} - \frac{1}{m} \ones \|_1 + \sum_{k > t - \tilde t}^{t} \| e_i^T A^{t - k} - \frac{1}{m} \ones \|_1 \\
        &\leq \sum_{k = 1}^{t - \tilde t}  m \delta^{t-k} + 2 \tilde t \leq \frac{m \delta^{\tilde t}}{1 - \delta} + 2 \tilde t \leq \frac{2}{1 - \delta} + \frac{2 \log \frac{m}{2}}{- \log \delta}
    \end{align*}

    Noting that $1 - \delta \leq - \log \delta$ and $\delta = 1 - \frac{\nu}{4 m^2}$ for any doubly stochastic matrix $A$,  we get 
    \begin{equation*}
         \sum_{k = 1}^t \sum_{j = 1}^m |[A^{t - k}_{ij}] - \frac{1}{m}|  \leq  \frac{2 + 2 \log \frac{m}{2}}{1 - \delta} \leq  \frac{4 \log m}{1 - \delta} \leq  \frac{16 m^2 \log m}{\nu}
    \end{equation*}
\end{proof}

\begin{lemma}[Distributed Law of Large Numbers]\label{lemma:distributed-LLN}
     Assume that $S_t^i, t \geq 1$ are i.i.d. random variables with $\E[|S_t^i|]$ exists and are finite, for all $ i \in [m]$. Under Assumption \ref{assumption: graph}, for any $j \in [m]$, the random variables
        \begin{equation*}
        Z_t^j = \frac{1}{t} \sum_{i = 1}^m\sum_{k = 1}^t [A_{ij}^{t-k}] S_t^i
    \end{equation*}
        converge in probability to $\frac{1}{m} \sum_{i = 1}^m \E[S_1^i]$ as $t \to \infty$. 
\end{lemma}
\begin{proof}
Let $Z_t= [Z_t^1, \cdots, Z_t^m]^T$ and $S_t= [S_t^1, \cdots, S_t^m]^T$. 
Since $\E[|S_t^i|] < \infty$ exists and is finite, we have 
\begin{equation*}
    Z_t = \frac{1}{t} \sum_{k = 1}^t A^{t-k} S_k. 
\end{equation*}
Using the property of stochastic matrices \citep{rosenthal1995convergence}, the diagonalizable matrix $A$ satisfies 
\begin{equation*}
    \|e_i^T A^t - \frac{1}{m} \ones \|_1 \leq m \delta^t. 
\end{equation*}
for all $i \in [m]$ and some $\delta < 1$. 

This implies hat
\begin{align*}
&\left\|Z_t  - \sum_{k = 1}^t \frac{1}{m} \ones \ones^T S_k \right\| = \left\|\frac{1}{t} \sum_{k = 1}^t A^{t-k} S_k - \frac{1}{m} \ones \ones^T S_k \right\| \\
&\leq \frac{1}{t} \sum_{k = 1}^t \|(A^{t-k} -\frac{1}{m} \ones \ones^T) S_k \| \\
&\leq \frac{1}{t} \sum_{k = 1}^t \max_{i \in [m]} \left\|e_i^T A^{t-k} - \frac{1}{m} \ones \right\| \max_{i \in [m]}|S_k^i|, \quad \text{by Cauchy - Schwarz inequality} \\
&\leq \frac{1}{t} \sum_{k = 1}^t  m \delta^{t-k} \max_{i \in [m]}|S_k^i| \\
&\leq \frac{m}{t} \sum_{k = 1}^t  \sum_{i = 1}^m  \delta^{t-k} |S_k^i| \\
&\leq \frac{m}{t} \sum_{k = t - c + 1}^t  \sum_{i = 1}^m  \delta^{t-k} |S_k^i|  +  \frac{m}{t} \delta^c \sum_{k = 1}^{t - c}  \sum_{i = 1}^m |S_k^i|, 
\end{align*}
for some fixed constant $c$. 

By the strong law of large number,  $\left\|\sum_{k = 1}^t \frac{1}{m} \ones \ones^T S_k - \frac{1}{m}  \sum_{i = 1}^m\E[S_1^i]\right\| = o_d(1)$.  Then
\begin{align*}
\left\|Z_t  - \frac{1}{m}  \sum_{i = 1}^m\E[S_1^i]\right\|  &\leq  \left\|Z_t  - \sum_{k = 1}^t \frac{1}{m} \ones \ones^T S_k \right\|  + \left\|\sum_{k = 1}^t \frac{1}{m} \ones \ones^T S_k  - \frac{1}{m}  \sum_{i = 1}^m\E[S_1^i]\right\| \\
 &=   \frac{m}{t} \delta^c \sum_{k = 1}^{t - c}  \sum_{i = 1}^m   |S_k^i| +\frac{m}{t} \sum_{k = t - c + 1}^t  \sum_{i = 1}^m  \delta^{t-k} |S_k^i| +o_d(1) \\
  &\leq   \frac{m}{t} \delta^c \sum_{k = 1}^{t - c}  \sum_{i = 1}^m   |S_k^i| +o_d(1) \\
&\leq   m \delta^c  \sum_{i = 1}^m\E[S_1^i]+o_d(1). 
\end{align*}
Since this for arbitrarily large $c$, we conclude that $ \left|Z_t  - \frac{1}{m}  \sum_{i = 1}^m\E[S_1^i]\right| = o_d(1)$, that is
\begin{equation*}
    Z_t \to \frac{1}{m}  \sum_{i = 1}^m\E[S_1^i] \quad \text{in probability}. 
\end{equation*}

\end{proof}
\begin{lemma}[Distributed Central Limit Theorem] \label{lemma:distributed-CLT}
    Assume that $S_t^i, t \geq 1$ are i.i.d. random variables with $\E[S_t^i] = \mu^i$ and $cov[S_t^i] = \Sigma^i$ exists and are finite, for all $ i \in [m]$. Under Assumption \ref{assumption: graph}, for any $j \in [m]$, the random variables
    \begin{equation*}
        Z_t^j = \sqrt{\frac{m}{t}}\left(\frac{1}{m}\sum_{i = 1}^m \Sigma^i \right)^{-1/2} \sum_{i = 1}^m\sum_{k = 1}^t [A_{ij}^{t-k}](S_t^i - \mu^i)
    \end{equation*}
    converge in distribution to a standard normal distribution as $t \to \infty$. 
\end{lemma}
\begin{proof}
    The proof applies the Linderberg central limit theorem. Without loss of generality, assume that $\mu^i = 0$. Let $\Sigma_{m,t}$ be the sum of covariance. We have
    \begin{align*}
        \Sigma_{m,t} &=  \sum_{k = 1}^t \sum_{i = 1}^m \text{Cov}([A_{ij}^{t-k}]S_t^i) =   \sum_{k = 1}^t \sum_{i = 1}^m[A_{ij}^{t-k}]^2 \Sigma^i . 
    \end{align*}
The matrix $\Sigma_{m,t}$ is asymptotically equivalent to the scaling factor $\Sigma_{m,t}^* = \frac{t}{m^2} \sum_{i = 1}^m  \Sigma^i$. 
    \begin{align*}
        \left(\Sigma_{m,t}^* \right)^{-1}\Sigma_{m,t} &= \frac{m}{t}  \left(\frac{1}{m} \sum_{i = 1}^m  \Sigma^i\right)^{-1}  \sum_{k = 1}^t \sum_{i = 1}^m[A_{ij}^{t-k}]^2 \Sigma^i  \\
        &=m  \left(\frac{1}{m} \sum_{i = 1}^m  \Sigma^i\right)^{-1} \sum_{i = 1}^m \Sigma^i \frac{\sum_{k = 1}^t[A_{ij}^{t-k}]^2 }{t} \\
        &= m  \left(\frac{1}{m} \sum_{i = 1}^m  \Sigma^i\right)^{-1} \sum_{i = 1}^m \Sigma^i \frac{\sum_{k = 1}^t[A_{ij}^k]^2 }{t}
    \end{align*}
    By Assumption \ref{assumption: graph}, $\lim_{t \to \infty}A_{ij}^k \to \frac{1}{m}$, thus $\lim_{t \to \infty}[A_{ij}^k]^2 \to \frac{1}{m^2}$.Then we have
    \begin{align*}
        \lim_{t \to \infty}   \left(\Sigma_{m,t}^* \right)^{-1}\Sigma_{m,t} &= m  \left(\frac{1}{m} \sum_{i = 1}^m  \Sigma^i\right)^{-1} \sum_{i = 1}^m \Sigma^i  \lim_{t \to \infty} \frac{\sum_{k = 1}^t[A_{ij}^k]^2 }{t} \\
        &=  m  \left(\frac{1}{m} \sum_{i = 1}^m  \Sigma^i\right)^{-1} \sum_{i = 1}^m \Sigma^i  \frac{1}{m^2} = I_m. 
    \end{align*}
It remains to verify the Linderberg condition
\begin{equation*}
    \lim_{t \to \infty}  \sum_{k = 1}^t \sum_{i = 1}^m \E\left[\|\Sigma_{m,t}^{-1/2} [A_{ij}^{t-k}] S_t^i \|^2 I_{\|\Sigma_{m,t}^{-1/2} [A_{ij}^{t-k}] S_t^i\| > \epsilon}\right]  = 0, \quad \forall \epsilon > 0. 
\end{equation*}
The condition is equivalent to 
\begin{equation*}
            \lim_{t \to \infty}  \sum_{k = 1}^t \sum_{i = 1}^m \E\left[\|\Sigma_{m,t}^{*^{-1/2}} [A_{ij}^{t-k}] S_t^i\|^2 I_{\|\Sigma_{m,t}^{-1/2} [A_{ij}^{t-k}] S_t^i\| > \epsilon}\right]  = 0, \quad \forall \epsilon > 0. 
\end{equation*}    
Since $\lim_{t \to \infty}\Sigma^*_{m,t} =  \infty$, $\lim_{t \to \infty}\Sigma_{m,t} = \infty$ and $ I_{\|\Sigma_{m,t}^{-1/2} [A_{ij}^{t-k}] S_t^i\| > \epsilon}$ converges to $0$ almost surely. 

Because $\max_{i \in [m]} \sup_t [A_{ij}^t]^2 < \infty$, we have
\begin{align*}
    \max_{i \in [m]}\sup_{t} t\E\left[ \|\Sigma_{m,t}^{*^{-1/2}} [A_{ij}^{t-k}] S_t^i\|^2\right] &= \max_{i \in [m]} \sup_{t} \E\left[\|m(\frac{1}{m}\sum_{i = 1}^m  \Sigma^i)^{-1} [A_{ij}^{t-k}] S_t^i\|^2\right]  \\
    &\leq m (\max_{i \in [m]} \sup_t [A_{ij}^t]^2 ) \|(\frac{1}{m}\sum_{i = 1}^m  \Sigma^i)^{-1}\|_{op} \max_{i \in [m]} \|\Sigma^i \|_{op}
\end{align*}
By Lebesgue dominated convergence theorem, we have
\begin{equation*}
    \lim_{t \to \infty} t\E\left[\|\Sigma_{m,t}^{*^{-1/2}} [A_{ij}^{t-k}] S_t^i\|^2 I_{\|\Sigma_{m,t}^{-1/2} [A_{ij}^{t-k}] S_t^i\| > \epsilon}\right]  \to 0. 
\end{equation*}
The Cesaro sums also converge to $0$: 
\begin{equation*}
    \lim_{t \to \infty} \sum_{i = 1}^m \frac{1}{t}\sum_{k = 1}^t t \E\left[\|\Sigma_{m,t}^{*^{-1/2}} [A_{ij}^{t-k}] S_t^i\|^2 I_{\|\Sigma_{m,t}^{-1/2} [A_{ij}^{t-k}] S_t^i\| > \epsilon}\right]  = 0, \quad \forall \epsilon > 0. 
\end{equation*}
This completes the proof
\end{proof}

\begin{lemma}[\cite{miller2021asymptotic}] 
\label{uniform-convergence-lemma}
Suppose that $h_n: E \to F$ for $n \in \N$, where $E$ is a totally bounded space and $F$ is a normed space. If $h_n$ converges pointwise and is equicontinuous, then it converges uniformly. 
\end{lemma}

\iffalse
\begin{lemma}
[Proposition 6.7. \cite{ghosal2017fundamentals} ]\label{prop:gv-6.7}
Suppose that the posterior distribution $\Pi_n(\cdot \mid X(n))$ is consistent (or strongly consistent) at $\theta_0$ relative to the metric $d$ on $\Theta$. Then $\hat{\theta}_n$, defined as the center of a (nearly) smallest ball that contains posterior mass at least $\delta$ for $\delta > 0$, satisfies
\begin{equation*}
    d(\hat{\theta}_n, \theta_0) \rightarrow 0, 
\end{equation*}    
in $[P_{\theta_0}]-$probability or $[P_{\theta_0}]-$almost surely, respectively. 
\end{lemma}
\fi

\end{document}